\documentclass{amsart}

\usepackage{etoolbox}

\makeatletter
\let\old@tocline\@tocline
\let\section@tocline\@tocline
\newcommand{\subsection@dotsep}{4.5}
\newcommand{\subsubsection@dotsep}{4.5}
\patchcmd{\@tocline}
  {\hfil}
  {\nobreak
     \leaders\hbox{$\m@th
        \mkern \subsection@dotsep mu\hbox{.}\mkern \subsection@dotsep mu$}\hfill
     \nobreak}{}{}
\let\subsection@tocline\@tocline
\let\@tocline\old@tocline

\patchcmd{\@tocline}
  {\hfil}
  {\nobreak
     \leaders\hbox{$\m@th
        \mkern \subsubsection@dotsep mu\hbox{.}\mkern \subsubsection@dotsep mu$}\hfill
     \nobreak}{}{}
\let\subsubsection@tocline\@tocline
\let\@tocline\old@tocline

\let\old@l@subsection\l@subsection
\let\old@l@subsubsection\l@subsubsection

\def\@tocwriteb#1#2#3{%
  \begingroup
    \@xp\def\csname #2@tocline\endcsname##1##2##3##4##5##6{%
      \ifnum##1>\c@tocdepth
      \else \sbox\z@{##5\let\indentlabel\@tochangmeasure##6}\fi}%
    \csname l@#2\endcsname{#1{\csname#2name\endcsname}{\@secnumber}{}}%
  \endgroup
  \addcontentsline{toc}{#2}%
    {\protect#1{\csname#2name\endcsname}{\@secnumber}{#3}}}%

\newlength{\@tocsectionindent}
\newlength{\@tocsubsectionindent}
\newlength{\@tocsubsubsectionindent}
\newlength{\@tocsectionnumwidth}
\newlength{\@tocsubsectionnumwidth}
\newlength{\@tocsubsubsectionnumwidth}
\newcommand{\settocsectionnumwidth}[1]{\setlength{\@tocsectionnumwidth}{#1}}
\newcommand{\settocsubsectionnumwidth}[1]{\setlength{\@tocsubsectionnumwidth}{#1}}
\newcommand{\settocsubsubsectionnumwidth}[1]{\setlength{\@tocsubsubsectionnumwidth}{#1}}
\newcommand{\settocsectionindent}[1]{\setlength{\@tocsectionindent}{#1}}
\newcommand{\settocsubsectionindent}[1]{\setlength{\@tocsubsectionindent}{#1}}
\newcommand{\settocsubsubsectionindent}[1]{\setlength{\@tocsubsubsectionindent}{#1}}

\renewcommand{\l@section}{\section@tocline{1}{\@tocsectionvskip}{\@tocsectionindent}{}{\@tocsectionformat}}%
\renewcommand{\l@subsection}{\subsection@tocline{1}{\@tocsubsectionvskip}{\@tocsubsectionindent}{}{\@tocsubsectionformat}}%
\renewcommand{\l@subsubsection}{\subsubsection@tocline{1}{\@tocsubsubsectionvskip}{\@tocsubsubsectionindent}{}{\@tocsubsubsectionformat}}%
\newcommand{\@tocsectionformat}{}
\newcommand{\@tocsubsectionformat}{}
\newcommand{\@tocsubsubsectionformat}{}
\expandafter\def\csname toc@1format\endcsname{\@tocsectionformat}
\expandafter\def\csname toc@2format\endcsname{\@tocsubsectionformat}
\expandafter\def\csname toc@3format\endcsname{\@tocsubsubsectionformat}
\newcommand{\settocsectionformat}[1]{\renewcommand{\@tocsectionformat}{#1}}
\newcommand{\settocsubsectionformat}[1]{\renewcommand{\@tocsubsectionformat}{#1}}
\newcommand{\settocsubsubsectionformat}[1]{\renewcommand{\@tocsubsubsectionformat}{#1}}
\newlength{\@tocsectionvskip}
\newcommand{\settocsectionvskip}[1]{\setlength{\@tocsectionvskip}{#1}}
\newlength{\@tocsubsectionvskip}
\newcommand{\settocsubsectionvskip}[1]{\setlength{\@tocsubsectionvskip}{#1}}
\newlength{\@tocsubsubsectionvskip}
\newcommand{\settocsubsubsectionvskip}[1]{\setlength{\@tocsubsubsectionvskip}{#1}}

\patchcmd{\tocsection}{\indentlabel}{\makebox[\@tocsectionnumwidth][l]}{}{}
\patchcmd{\tocsubsection}{\indentlabel}{\makebox[\@tocsubsectionnumwidth][l]}{}{}
\patchcmd{\tocsubsubsection}{\indentlabel}{\makebox[\@tocsubsubsectionnumwidth][l]}{}{}

\newcommand{\@sectypepnumformat}{}
\renewcommand{\contentsline}[1]{%
  \expandafter\let\expandafter\@sectypepnumformat\csname @toc#1pnumformat\endcsname%
  \csname l@#1\endcsname}
\newcommand{\@tocsectionpnumformat}{}
\newcommand{\@tocsubsectionpnumformat}{}
\newcommand{\@tocsubsubsectionpnumformat}{}
\newcommand{\setsectionpnumformat}[1]{\renewcommand{\@tocsectionpnumformat}{#1}}
\newcommand{\setsubsectionpnumformat}[1]{\renewcommand{\@tocsubsectionpnumformat}{#1}}
\newcommand{\setsubsubsectionpnumformat}[1]{\renewcommand{\@tocsubsubsectionpnumformat}{#1}}
\renewcommand{\@tocpagenum}[1]{%
  \hfill {\mdseries\@sectypepnumformat #1}}

\let\oldappendix\appendix
\renewcommand{\appendix}{%
  \leavevmode\oldappendix%
  \addtocontents{toc}{%
    \protect\settowidth{\protect\@tocsectionnumwidth}{\protect\@tocsectionformat\sectionname\space}%
    \protect\addtolength{\protect\@tocsectionnumwidth}{2em}}%
}
\makeatother

\makeatletter
\settocsectionnumwidth{2em}
\settocsubsectionnumwidth{2.5em}
\settocsubsubsectionnumwidth{3em}
\settocsectionindent{1pc}%
\settocsubsectionindent{\dimexpr\@tocsectionindent+\@tocsectionnumwidth}%
\settocsubsubsectionindent{\dimexpr\@tocsubsectionindent+\@tocsubsectionnumwidth}%
\makeatother

\settocsectionvskip{10pt}
\settocsubsectionvskip{0pt}
\settocsubsubsectionvskip{0pt}

\settocsectionformat{\bfseries}
\settocsubsectionformat{\mdseries}
\settocsubsubsectionformat{\mdseries}
\setsectionpnumformat{\bfseries}
\setsubsectionpnumformat{\mdseries}
\setsubsubsectionpnumformat{\mdseries}

\let\oldtableofcontents\tableofcontents
\renewcommand{\tableofcontents}{%
  \vspace*{-\linespacing}% Default gap to top of CONTENTS is \linespacing.
  \oldtableofcontents}

\usepackage{amssymb,amsrefs}
\usepackage[shortlabels]{enumitem}
\usepackage{mathtools}
\usepackage{verbatim}
\usepackage{hyperref}

\usepackage{enumitem}
\makeatletter
\newcommand{\mylabel}[2]{#2\def\@currentlabel{#2}\label{#1}}
\makeatother

\usepackage{caption,subcaption}
\usepackage{tikz,tikz-3dplot}
\newcommand{\x}{\scalebox{1.2}{$\chi$} } 

\newcommand{\br}{\overline}
\newcommand{\R}{\mathbb R}
\newcommand{\C}{\mathbb C}

\newcommand{\N}{\mathbb N}

\newcommand{\UHP}{\mathbb H}

\theoremstyle{plain}
\newtheorem{theorem}{Theorem}
\newtheorem{lemma}[theorem]{Lemma}
\newtheorem{conjecture}{Conjecture}
\newtheorem{prop}[theorem]{Proposition}

\theoremstyle{definition}
\newtheorem{definition}[theorem]{Definition}
\newtheorem{claim}{Claim}

\theoremstyle{remark}
\newtheorem{remark}[theorem]{Remark}
\newtheorem{question}{Question}

\DeclareMathOperator{\diam}{\textup{\text{diam}}}

\DeclareMathOperator{\inter}{\textup{\text{int}}}

\DeclareMathOperator{\length}{\textup{\text{length}}}

\DeclareMathOperator{\loc}{\textup{loc}}
\DeclareMathOperator*{\osc}{\textup{osc}}

\DeclareMathOperator*{\dis}{\textup{dis}}

\numberwithin{equation}{section}
\numberwithin{theorem}{section}
\numberwithin{figure}{section}

\begin{document}
\title{Non-removability of the Sierpi\'nski Gasket}
\author{Dimitrios Ntalampekos}
\thanks{The author was partially supported by NSF grant DMS-1506099}
\address{Department of Mathematics\\ University of California, Los Angeles\\
CA 90095, USA.}
\curraddr{Institute for Mathematical Sciences \\ Stony Brook University \\ Stony Brook \\ NY 11794, USA.}
\email{dimitrios.ntalampekos@stonybrook.edu}
\date{\today}
\maketitle

\begin{abstract}
We prove that the Sierpi\'nski gasket is non-removable for quasiconformal maps, thus answering a question of Bishop \cite{Bi2}. The proof involves a new technique of constructing an exceptional homeomorphism from $\R^2$ into some non-planar surface $S$, and then embedding this surface quasisymmetrically back into the plane by using the celebrated Bonk-Kleiner Theorem \cite{BK}.  We also prove that all homeomorphic copies of the Sierpi\'nski gasket are non-removable for continuous Sobolev functions of the class $W^{1,p}$ for $1\leq p\leq 2$, thus complementing and sharpening the results of the author's previous work \cite{Nt}.  
\end{abstract}

\tableofcontents

\section{Introduction}\label{Section Introduction}

\subsection{Background and main results}
The object of this paper is to prove that the Sierpi\'nski gasket is non-removable for (quasi)conformal maps and Sobolev functions. The problem of (quasi)conformal and Sobolev removability has been studied extensively. Besides earlier results by  Besicovitch \cite{Be} and Gehring \cite{Ge:remov}, various conditions that guarantee the removability of compact sets have been established by Jones and Smirnov \cite{Jo}, \cite{JS}, Kaufman and Wu \cite{KW}, \cite{Wu}, Koskela and Nieminen \cite{KN}, and recently by the current author \cite{Nt}. Moreover, removability of Julia sets has been studied by Kahn \cite{Kah} and also by Graczyk and Smirnov \cite{GS}. On the other side of the coin, examples of non-removable sets and constructions of exceptional functions/homeomorphisms have been given by Kaufman and Wu \cite{Kau}, \cite{KW}, Bishop \cite{Bi}, \cite{Bi:R3}, and also Koskela, Rajala and Zhang \cite{KRZ}. The present paper provides one more result in this direction. Interestingly, according to a conjecture of He and Schramm \cite{HS}, the problem of removability is also related to the rigidity of circle domains and Koebe's conjecture.

So far, the types of sets which were known to be non-removable were sets of positive Lebesgue measure, some product sets, some Cantor sets, and also ``rough" sets which are topologically ``simple" (e.g.\ simple curves). The task of constructing an exceptional homeomorphism becomes very challenging, as the topology of the set deteriorates. To the best of our knowledge, it is the first time that a non-trivial construction is used to prove that a set with infinitely many complementary components, such as the Sierpi\'nski gasket, is non-removable. This hints that a \textit{generic} set with infinitely many complementary components should be non-removable. Such sets are, for example, Sierpi\'nski carpets, the Apollonian gasket, and also $\textrm{SLE}_\kappa$ for $\kappa \in (4,8)$; Sheffield \cite{Sh} has posed the question whether the latter is removable or not.

We include some background of the problem of removability of sets for (quasi)\-conformal maps in $\R^2$ and Sobolev functions in $\R^n$. We direct the reader \cite{Yo} for a thorough survey on the topic of (quasi)conformal removability and for proofs of some of the facts that we state here.

\begin{definition}\label{def:removable}
We say that a compact set $K\subset U \subset \R^2$ is (quasi)conformally  removable inside the domain $U$ if any homeomorphism of $U$, which is (quasi)conformal on $U\setminus K$, is (quasi)conformal on $U$.
\end{definition} 

Here we mention some basic facts. A set $K$ is quasiconformally removable inside $U$ if and only if $K$ is conformally removable inside $U$. Hence, from now on, we will be using the term quasiconformal removability. Furthermore, a set $K$ is quasiconformally removable  inside a domain $U$ if and only if $K$ is quasiconformally removable inside the entire plane $\R^2$. Two fundamental open questions are the following:

\begin{question}[p.\ 264, \cite{JS}]\label{Question union} Is the union of two intersecting compact sets quasiconformally removable, whenever each one of them is  removable?
\end{question}

For a partial result in this direction see \cite[Theorem 4]{YoKoebe}. In the same paper \cite[p.\ 1306]{YoKoebe} the author discusses the problem of local removability. A set $K$ is \textit{locally quasiconformally removable} if for \textit{any} open set $U$ (not necessarily containing $K$) and for any homeomorphism $f$ of $U$ that is quasiconformal on $U\setminus K$ we have that $f$ is quasiconformal on $U$.

\begin{question}\label{Question local} If a set is quasi\-conformally removable, is it also locally quasi\-conformally removable?
\end{question}

A stronger notion of removability is the notion of $W^{1,2}$-removability. We give the general definition of $W^{1,p}$-removability in $\R^n$, where $p\in [1,\infty]$. Recall that a function $f$ lies in $W^{1,p}(U)$, where $U$ is an open subset of $\R^n$, if $f\in L^p(U)$ and also $f$ has weak derivatives in $U$ that lie in $L^p(U)$.

\begin{definition}\label{Intro Sobolev Removable}
Let $p\in [1,\infty]$. We say that a compact set $K\subset \R^n$ is $W^{1,p}$-removable if any real-valued function that is continuous in $\R^n$ and lies in $W^{1,p}(\R^n\setminus K)$, also lies $W^{1,p}(\R^n)$.
\end{definition}

Using partitions of unity one can show that this definition is local, and thus the answer to the analog of Question \ref{Question local} is positive in this case. Furthermore, H\"older's inequality implies that if a set of measure zero is $W^{1,p}$-removable, then it is also $W^{1,q}$-removable for $q>p$.

\begin{question}[p.\ 264, \cite{JS}]\label{Question Equivalence}
Is $W^{1,2}$-removability in the plane equivalent to quasiconformal removability?
\end{question}

Interestingly, so far the techniques used in the two different notions of removability are the same, but there is no further indication whether the answer to the preceding question should be positive or negative.

If a set $K\subset \R^2$ has measure zero, then $W^{1,2}$-removability of $K$ implies quasi\-conformal removability. If a set $K\subset \R^2$ has positive measure then it is non-removable for quasiconformal maps. In \cite{Nt} the author posed the question whether this is true for Sobolev removability. We provide an answer to this question here:

\begin{theorem}\label{Intro Theorem Positive}
Let $K\subset \R^n$ be a compact set of positive Lebesgue measure and $1\leq p<\infty$. Then $K$ is non-removable for $W^{1,p}$.
\end{theorem} 

However, the statement is not true for $W^{1,\infty}$:

\begin{prop}\label{Intro Theorem W^{1,infinity}}
There exists a compact set $K\subset \R^n$ of positive Lebesgue measure that is $W^{1,\infty}$-removable.
\end{prop}

Classes of $W^{1,2}$- and quasiconformally   removable  sets include sets of $\sigma$-finite Hausdorff $1$-measure \cite{Be}, \cite[Section 35]{Va}, quasicircles, boundaries of John domains, of H\"older domains, and of domains satisfying certain quasihyperbolic conditions \cite{Jo}, \cite{JS}, \cite{KN}. Also, some novel techniques for the removability of Julia sets of quadratic polynomials appeared in \cite{Kah}.

On the other hand, as already remarked, all sets of positive measure are non-removable for (quasi)conformal maps and $W^{1,2}$ functions. Furthermore, there exist non-removable Jordan curves of Hausdorff dimension $1$ \cite{Bi}, and also non-removable graphs of functions \cite{Kau}. 

\begin{figure}
	\centering
	\input{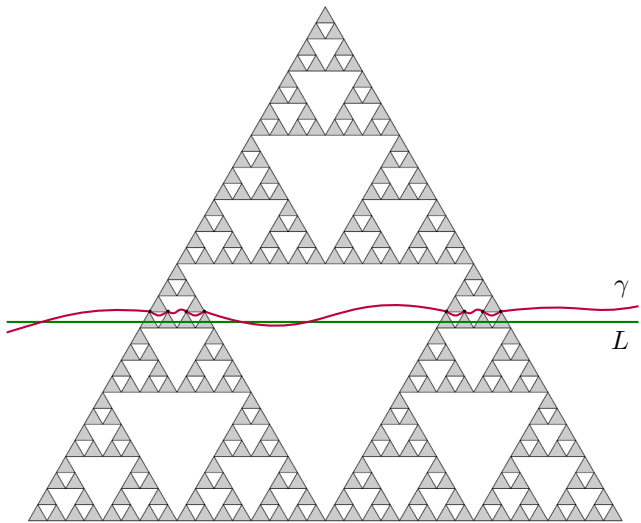}	
	\caption{The Sierpi\'nski gasket, and a detour path $\gamma$ near the line $L$.} \label{fig:gasket}
\end{figure}

Most of these results refer to compact sets that are the boundary of the union of \textit{finitely many} connected open sets. Until recently, there had been no general result on sets with \textit{infinitely many} complementary components, not falling into the preceding category. The task of proving that such a set is (non)-removable requires the development of different tools.  In \cite{Nt} the author studied this problem and derived a condition that guarantees $W^{1,p}$-removability of a set $K\subset \R^n$ for $p>n$. Roughly speaking, the condition is the following:
\begin{enumerate}
\item The complementary components of $K$ are uniform H\"older domains (see e.g.\ \cite{SS} and also \cite{Nt} for the definition), and
\item for ``almost every" line $L$ intersecting $K$, and for every $\varepsilon>0$ there exists a ``detour path" $\gamma$ that $\varepsilon$-follows the line $L$, but intersects only finitely many of the complementary components of $K$ that the line $L$ intersects.
\end{enumerate}
In other words,  (2) says that we can ``travel" in the direction of the line $L$ using only finitely many components in the complement of $K$, but still staying arbitrarily close to the line $L$; see Figure \ref{fig:gasket}. We call such sets \textit{detour sets}, and the Sierpi\'nski gasket, depicted in Figure \ref{fig:gasket}, is one such set. 

\begin{theorem}[Corollary 1.4, \cite{Nt}]\label{Intro gasket Removable}
The Sierpi\'nski gasket is $W^{1,p}$-removable for $p>2$.
\end{theorem}

Other sets that fall into the same category are the Apollonian gasket and generalized Sierpi\'nski gasket Julia sets of sub-hyperbolic rational maps; see \cite[Section 7]{Nt}.

The \textit{Sierpi\'nski gasket} is constructed as follows. We consider an equilateral triangle in the plane of sidelength $1$ and subdivide it into four equilateral triangles of sidelength $1/2$. After removing the middle triangle, we proceed inductively with subdividing each of the remaining three triangles into four equilateral triangles of sidelength $1/2^2$, and so on. The remaining compact set $K$ is the Sierpi\'nski gasket.

In this work we first prove that Theorem \ref{Intro gasket Removable} is sharp:
\begin{theorem}\label{Intro gasket non-removable}
The Sierpi\'nski gasket is non-removable for $W^{1,p}$, for $1\leq p\leq 2$.
\end{theorem}
In fact, by the monotonicity discussed after Definition \ref{Intro Sobolev Removable}, it suffices to prove that the gasket is non-removable for $W^{1,2}$.

The method used is very flexible and we obtain the following result:
\begin{theorem}\label{Intro homeo gasket non-removable}
Let $h\colon \R^2\to \R^2$ be a homeomorphism and $K$ be the Sierpi\'nski gasket. Then $h(K)$ is non-removable for $W^{1,2}$.
\end{theorem}
Of course, if $h(K)$ has positive Lebesgue measure, then the conclusion follows immediately from Theorem \ref{Intro Theorem Positive}, so the measure zero case is the interesting one.

Finally, we boost the proof of Theorem \ref{Intro gasket non-removable} and add a new ingredient to obtain the main result:

\begin{theorem}\label{Intro gasket Quasi non-removable}
The Sierpi\'nski gasket is non-removable for quasiconformal maps.
\end{theorem}
In other words, there exists a homeomorphism of $\R^2$ that is quasiconformal in the complement of the gasket, but not globally quasiconformal. This answers a question raised by Bishop \cite[Question 13]{Bi2}. We were not able to show the analog of Theorem \ref{Intro homeo gasket non-removable} in this case, i.e., that all homeomorphic copies of the gasket are non-removable for quasiconformal maps, but we believe that a modification of the techniques used here can provide the answer.

We now  discuss a natural approach to the problem, which, however, seems extremely hard to pursue; then, in Section \ref{Section Sketch} we give a brief outline of the proof of Theorem \ref{Intro gasket Quasi non-removable} and explain what  this ``new ingredient" that we use in our approach is.

As explained in \cite[p.\ 15]{Bi2}, one may try to construct an exceptional homeomorphism as in Theorem \ref{Intro gasket Quasi non-removable} as follows. Note that between any two triangles $V_1$ and $V_2$ there exists a conformal map, and this map is unique if one requires that each vertex of $V_1$ is mapped to a prescribed vertex of $V_2$, and the vertex correspondence is orientation-preserving. Here, we allow $V_i$, $i=1,2$, to be an ``unbounded" triangle, i.e., the unbounded complementary component of a triangle. Hence, one can first map the unbounded complementary component of the Sierpi\'nski gasket onto an unbounded non-equilateral triangle, and then (inductively) require that every bounded complementary equilateral triangle of the gasket is mapped uniquely by a conformal map to a triangle. Note that at each level of this construction the image of the vertices of a triangle is prescribed by the map of the previous level. This process will uniquely determine a map that is conformal on each complementary component of the gasket. If this map extends to a homeomorphism of $\R^2$, then it cannot be globally conformal, since it changes angles. However, it is not clear at all whether this map can be extended to a continuous map on $\R^2$, and thus to a homeomorphism of $\R^2$.

Before discussing the outline of the proof of Theorem \ref{Intro gasket Quasi non-removable} in Section \ref{Section Sketch}, we conclude this section with some remarks on the (non)-removability of another related type of fractals whose complement has infinitely many components, namely, Sierpi\'nski carpets.

The \textit{standard Sierpi\'nski carpet} $S_3$ is constructed by subdividing the unit square $[0,1]^2$ into nine squares of sidelength $1/3$ and removing the middle square, and then proceeding inductively in each of the remaining eight squares. It is easy to see that the standard Sierpi\'nski carpet is non-removable for quasiconformal and $W^{1,p}$ maps for $1\leq p\leq \infty$. We sketch the proof for quasiconformal non-removability. Note that $S_3$ contains a copy of $C\times [0,1]$, where $C$ is the middle-thirds Cantor set. Let $h\colon \R \to \R$ be the Cantor staircase function and let $\psi\colon \R \to [0,1]$ be a smooth function with $\psi\equiv 0 $ outside $[0,1]$ and $\psi\equiv 1$ in $[1/9,8/9]$. Then $f(x,y)\coloneqq (x+h(x)\psi(y),y)$ is a homeomorphism of $\R^2$ that is quasiconformal on $\R^2\setminus S_3$, but not globally quasiconformal.

A \textit{(generalized) Sierpi\'nski carpet} is a planar set $S\subset \R^2$ that is homeomorphic to $S_3$; the homeomorphism need not be defined on all of $\R^2$. These sets can be characterized as the compact sets of the plane arising by removing from the interior of a Jordan region $\Omega$ countably many Jordan regions $Q_i$ for $i\in \N$, whose closures are disjoint and contained in $\Omega$, such that $\diam(Q_i) \to 0$ as $i\to\infty$ and $S\coloneqq \br \Omega \setminus \bigcup_{i\in \N}Q_i$ has empty interior; see \cite{Wh2}.
 
It is not known, in general, whether these sets are removable for quasiconformal maps or Sobolev functions, but we conjecture the following:

\begin{conjecture}\label{Intro:conjecture}
All Sierpi\'nski carpets are non-removable for quasiconformal maps and for $W^{1,p}$ functions, for $1\leq p\leq \infty$.
\end{conjecture}

\subsection{Sketch of the proof of Theorem \ref{Intro gasket Quasi non-removable}}\label{Section Sketch}
Let $K$ be the Sierpi\'nski gasket. We quickly sketch the strategy of constructing a homeomorphism $F \colon\R^2\to \R^2$ that is quasiconformal on $\R^2\setminus K$, but not globally quasiconformal. The lemmas introduced here are informal and all details can be found in Section \ref{Section QC non}. We suggest that the reader browse through the colored figures (located in Section \ref{Section QC non}) while reading the sketch.

First, we will define a continuous map $f\colon \R^2\to \R^2$ that is the identity on the unbounded component of $\R^2\setminus K$, it is injective outside the bounded equilateral triangle components of $\R^2\setminus K$, and collapses each equilateral triangle component to a tripod; see Figure \ref{fig:FoldGasket}. This map is defined inductively. More precisely, the map $f$ is the identity on the boundary of the unbounded complementary  triangle. Then, one collapses in a continuous way the ``middle" complementary equilateral triangle $W$ of sidelength $1/2$ to a tripod, whose barycenter is precisely the barycenter of the vertices of $W$ and whose vertices are the vertices of $W$. We also require that the midpoints of the three edges of $W$ are collapsed to the barycenter of the tripod $G$. Note that on the  vertices  of $W$ the function $f$ has already been defined to be the identity, by the previous step. Inductively, all complementary equilateral triangles of the gasket will be collapsed to tripods. The proof of the uniform continuity of such a map is a significant hurdle that we have to deal with. In fact, one has to choose very carefully the collapsing maps in each step, so that they satisfy a certain modulus of continuity. Note that after this procedure the gasket $K$ is blown up by $f$ to a set of full measure, since the tripods have measure zero. Summarizing, we have:
\begin{lemma}
There exists a continuous map $f\colon \R^2\to \R^2$ that is injective outside the complementary equilateral triangles of the gasket, and collapses each (bounded) complementary equilateral triangle to a tripod. Furthermore, $f(K)$ has positive Lebesgue measure.
\end{lemma}

Of course, this map is not a homeomorphism on the complementary equilateral triangles, so we have to correct it there, but otherwise keep the existing definition. Unfortunately, there is no way to correct our function in this way, if the target is $\R^2$. What we do instead, is change the target to a non-Euclidean metric surface $S$ and correct the map $f$ inside the complementary equilateral triangles in order to obtain a homeomorphism $\Phi\colon \R^2\to S$ that is quasiconformal in $\R^2\setminus K$.

The ``correction" that we do in each equilateral triangle $W$ is the following. We ``fold" $W$ on top of the tripod $f(W)$; see Figure \ref{fig:FoldTriangle}. The folding map will be $M$-quasiconformal  for a universal $M>0$. In fact, the folding map will just be piecewise linear. More precisely, we prove that for an arbitrary equilateral triangle $W$ and an arbitrary tripod $G$ we can find an $M$-quasiconformal map that folds $W$ onto six rectangles that are attached on top of the edges of $G$, with appropriate identifications. We call \textit{flap} the metric space arising by folding a single equilateral triangle over a tripod. In this folding procedure, one can choose the ``height" of the flap to be arbitrarily small, without affecting the constant $M$. Furthermore, a crucial property is that the folding map has to be compatible, in a sense, with $f$ on $\partial W$, because we we wish to paste the two maps. In particular, the folding map has to have a certain modulus of continuity (the one that ensures the uniform continuity of $f$) and it must map the midpoints of the edges of $W$ to ``lifts" of the barycenter of the tripod $G$; see Figure \ref{fig:FoldTriangle}. We now summarize:
\begin{lemma}
There exists a universal $M>0$ such that for each equilateral triangle $W$ and for each tripod $G$ there exists a folding map $\phi_W$ from $W$ onto the flap space corresponding to $G$. The height of the flap can be chosen to be arbitrary small. Moreover, the maps $f$ and $\phi_W$ can be chosen to be compatible for all complementary equilateral triangles $W$ of the gasket $K$. 
\end{lemma}

By folding all complementary equilateral triangles over their corresponding tri\-pods, one obtains a \textit{flap-plane} $S$, which is a non-Euclidean surface, and a homeomorphism $\Phi\colon \R^2\to S$ that is $M$-quasi\-conformal on $\R^2\setminus K$ and ``agrees" with $f$ outside the triangles; see Figure \ref{fig:FoldGasket} (we remark here that in the figure the edges of the green rectangles are not glued to the red rectangles, except possibly at one point; see also Figure \ref{fig:TwoTripodflaps} for the gluing pattern). The map $\Phi$ is the result of pasting the map $f$ with all the folding maps $\phi_W$. The map $\Phi$ maps the gasket to a subset of $S$ that has positive measure. In brief:
\begin{lemma}
There exists a homeomorphism $\Phi$ from $\R^2$ onto a metric surface $S$ that is $M$-quasiconformal on $\R^2\setminus K$ and maps $K$ to a subset of $S$ that has positive Hausdorff $2$-measure. 
\end{lemma}

If the target of $\Phi$ were not $S$ but it were $\R^2$, then the proof of non-removability would be finished. Hence, we have to find a way to change the target to $\R^2$.  This is facilitated by the Bonk-Kleiner Theorem \cite{BK}, which allows us to embed $S$ into $\R^2$ with a quasisymmetric map. The Bonk-Kleiner Theorem asserts that a metric sphere that is Ahlfors $2$-regular and linearly locally connected is quasisymmetrically equivalent to the standard Euclidean sphere. We develop a theory of flap-planes constructed similarly to $S$. These are just spaces arising by gluing to the plane an infinite collection of rectangles (or flaps), which are  ``perpendicular" to the plane. Using the Bonk-Kleiner Theorem  we will show that flap-planes can be quasisymmetrically embedded to the plane, provided that the heights of the flaps are sufficiently small. In our case, we can obtain a quasisymmetry $\Psi\colon S\to \R^2$. Note that $\Psi$ is a quasisymmetry on \textit{all} of $S$.

\begin{lemma}There exists a quasisymmetry $\Psi\colon S\to \R^2$. Moreover, $\Psi$ maps sets of positive Hausdorff $2$-measure to sets of positive Lebesgue measure.
\end{lemma}

The composition $F= \Psi\circ \Phi$ will be a homeomorphism of $\R^2$ that is $M'$-quasiconformal on $\R^2\setminus K$ for some uniform $M'>0$, but it cannot be globally quasiconformal, because it has to blow the gasket $K$ to a set of positive area. This is because $\Phi(K)$ had positive measure in $S$ and in our setting the quasisymmetry $\Psi$ has to map sets of positive measure to sets of positive measure.

\subsection{Organization of the paper}
In Section \ref{Section Preliminaries} we introduce our notation and discuss some preliminaries regarding quasiconformal and quasisymmetric maps, and also convergence of metric spaces in the pointed Gromov-Hausdorff sense.

In Section \ref{Section Flap-planes} we develop a theory of \textit{flap-planes}, which are surfaces arising by gluing rectangles, or else, \textit{flaps}, on top of the Euclidean plane. Under some assumptions, our goal is to apply the Bonk-Kleiner Theorem to show that these surfaces can be quasisymmetrically embedded into the plane; see Theorem \ref{General-QS Embedding}. Hence, we focus on proving that they are Ahlfors $2$-regular and linearly locally connected. These proofs occupy most of the section. Also, this section is independent of the other sections, and can be mostly skipped in a first reading of the paper. We will only need the definition and some general properties of flap-planes from Section \ref{Section General} and also we will use the embedding Theorem \ref{General-QS Embedding}.

The main content of Section \ref{Section Continuous} is the proof of  Theorem \ref{Intro gasket non-removable}, i.e., the non-remov\-ability of the gasket for continuous functions of the class $W^{1,2}$. The proof consists of several steps, which are organized in the subsections. The heart of the argument is Lemma \ref{Basic lemma}. The proof of Theorem \ref{Intro homeo gasket non-removable} is contained in Section \ref{Section Generalization}. There, we also include the proofs of the general statements in Theorem \ref{Intro Theorem Positive} and Proposition \ref{Intro Theorem W^{1,infinity}}. Moreover, in Section \ref{Section Definitions} we include basic terminology and geometric properties of the gasket that we use repeatedly throughout the paper.

Finally, in Section \ref{Section QC non} we prove the quasiconformal non-removability in Theorem \ref{Intro gasket Quasi non-removable}. First, in Section \ref{Section QC Collapsing} we show how to collapse the complementary equilateral triangles to tripods in a continuous way with a map $f\colon \R^2\to \R^2$, as described in Section \ref{Section Sketch}. The proof of the continuity of $f$ is the same as the proof of continuity for the Sobolev non-removability in Section \ref{Section Continuous}, so we recommend the reader to read first that proof. 

Then, in Section \ref{Section QC Folding} we explain how to fold a single equilateral triangle on top of a tripod with a quasiconformal map. In Section \ref{Section QC homeo} these folding maps are pieced together with $f$ to obtain a homeomorphism $\Phi$ from $\R^2$ onto a flap-plane $S$. Finally, in Section \ref{Section QC Finish} we finish the proof of non-removability by embedding $S$ into the plane and obtaining the desired exceptional homeomorphism $F\colon \R^2\to \R^2$.

\subsection*{Acknowledgements}
I am grateful to my advisor at UCLA, Mario Bonk, not only for the numerous conversations and useful comments during this project, but also for introducing me to the world of analysis on metric spaces and constantly keeping me motivated to learn mathematics and work on fascinating problems.

I also thank Huy Tran for bringing the problem of (non)-removability of the gasket to my attention, Malik Younsi for several motivating conversations, Pekka Koskela for his suggestions regarding the proof of Theorem \ref{Intro Theorem Positive}, and Guy C.\ David for explaining the different notions of convergence of metric spaces that appear in the literature. Moreover I would like to thank Matthew Romney, Raanan Schul, Jang-Mei Wu, Malik Younsi, and the anonymous referee for their comments and corrections.

\subsection*{Update} 
Since the completion and distribution of the first version of this paper, there has been some further progress. The current author in \cite{Nt2} has proved that all Sierpi\'nski carpets are non-removable for quasiconformal maps, providing an answer to Conjecture \ref{Intro:conjecture}.

Moreover, further connections between the problems of rigidity of circle domains and removability of their boundary have been established in \cite{NY}, in the spirit of the conjecture of He and Schramm \cite{HS}. In particular, it is proved that all circle domains satisfying the quasihyperbolic condition of \cite{JS} are rigid.

\section{Preliminaries}\label{Section Preliminaries}

\subsection{Notation}\label{Section Notation}
We will say that a statement (A) implies a statement (B) \textit{quantitatively} if the implicit constants or functions in  statement (B) depend only on the implicit constants or functions in  statement (A).

We use the notation $a\lesssim b$ if there exists an implicit constant $C>0$ such that $a\leq Cb$, and $a\simeq b$ if there exists a constant $C>0$ such that $C^{-1}b\leq a\leq Cb$. We will be mentioning the parameters on which the constant $C$ depends, unless $C$ is a universal constant.

The Lebesgue measure in $\R^n$ is denoted by $m_n$. If a function $f\colon \R^n\to \R$ is integrable, then we will denote its  integral against Lebesgue measure by $\int f$. The open ball around $x\in \R^n$ of radius $r>0$ is denoted by $B(x,r)$. For visual purposes, the closure of a set $U_1$ is denoted by $\br U_1$, instead of $\br {U_1}$ and the closure  of a Euclidean ball $B(x,r)$ is denoted by $\br B(x,r)$. We use the notation $A(x;r,R)$ for the annulus $B(x,R)\setminus \br B(x,r)$, where $0<r<R$.

If $(X,d)$ is a metric space and $Q>0$, then we denote the Hausdorff $Q$-measure by $\mathcal H^Q_d$; see e.g.\ \cite[Section 11.2]{Fo} for the definition. Also, we use the notation $B_d(x,r)$ for the open ball of radius $r>0$, centered at $x\in X$. In Section \ref{Section Flap-planes} we will be endowing planar sets with different metrics. If there is no subscript $d$ in the ball notation, then the ball will always refer to the Euclidean metric. 

We normalize the Hausdorff $2$-measure, so that it agrees with the Lebesgue measure in $\R^2$. This $\R^2$-normalization will also be used when we study the Hausdorff $2$-measure of an arbitrary metric space. For example, if $X$ is a metric space and $U \subset X$ is a set that is isometric to a measurable subset of $\R^2$, then its Hausdorff $2$-measure agrees with the Lebesgue measure of its isometric image in $\R^2$.

\subsection{Quasiconformal and quasisymmetric maps}\label{Section Preliminaries QC QS}

We first recall the definition of a quasiconformal map; we direct the reader to \cite{Va} and \cite{AIM} for background on quasiconformal maps.

\begin{definition}\label{Pre Definition QC plane}
Let $U,V\subset \R^2$ be open sets. An orientation-preserving homeomorphism $f\colon  U\to V$ is $M$-quasiconformal for some $M>0$ if $f\in W^{1,2}_{\loc}(U)$ and 
\begin{align*}
\|Df(z)\|^2 \leq M J_f(z)
\end{align*}
for a.e.\ $z\in U$, where $\|Df(z)\|$ denotes the operator norm of the differential of $f$ at $z$, and $J_f$ is the Jacobian of $f$. We also say that $f$ is quasiconformal if it is $M$-quasiconformal for some $M>0$. The number $M>0$ is called the quasiconformal distortion of $f$.
\end{definition}

\begin{remark}
A priori, if a set $K\subset U$ is removable in the sense of Definition \ref{def:removable}, it could be the case that the quasiconformal distortion of a map $f$ in $U$ is larger than the distortion in $U\setminus K$. However, as remarked in the Introduction, removable sets necessarily have measure zero. Since the inequality $\|Df(z)\|^2 \leq M J_f(z)$ is required to hold a.e., it follows that the quasiconformal distortion of $f$ on $U$ is the same as the distortion on $U\setminus K$.
\end{remark}

Quasiconformal maps have the important property that they preserve null sets: 
\begin{lemma}[Theorem 33.2, \cite{Va}]\label{Pre Measure zero}
Let $U,V\subset \R^2$ be open sets and let $f\colon U\to V$ be a quasiconformal map. A Borel set $A\subset U$ is mapped by $f$ to a set of measure zero if and only if $A$ has measure zero.
\end{lemma}

\begin{definition}\label{Pre Definition QC metric}
If two metric spaces $(X,d_X)$ and $(Y,d_Y)$ are locally isometric to open subsets of $\R^2$, then we say that a homeomorphism $f\colon X\to Y$ is $M$-quasiconformal if the following holds. For each $x \in X$ there exist open neighborhoods $U_x$ of $x$ and $V_{f(x)}$ of $f(x)$ and  there exist isometries $\phi\colon U_x\to U $ and $\psi\colon V_{f(x)} \to V $, where $U$ and $V$ are open subsets of $\R^2$ such that $\psi \circ f\circ \phi^{-1} \colon U\to V$ is $M$-quasiconformal, in the preceding sense. 
\end{definition}

There exists already a theory of quasiconformal maps between metric spaces that is compatible with the definition that we gave. Nevertheless, we will not need any deep results from that theory, and we wish to keep our approach as simple as possible, so we do not give the general definition. See, for example, \cite{HK} for more background.

Now, we define the notion of a quasisymmetry between two metric spaces $(X,d_X)$ and $(Y,d_Y)$; see also \cite[Chapters 10--11]{He}.

\begin{definition}\label{Pre Definition QS}
A homeomorphism $f\colon X\to Y$ is $\eta$-quasisymmetric if there exists a homeomorphism $\eta\colon [0,\infty)\to[0,\infty)$ such that for every triple of distinct points $x,y,z\in X$ and for their images $x'=f(x)$, $y'=f(y)$, $z'=f(z)$ we have
\begin{align*}
\frac{d_Y(x',y')}{d_Y(x',z')} \leq \eta \left( \frac{d_X(x,y)}{d_X(x,z)}\right).
\end{align*}
The function $\eta$ is called the distortion function associated to $f$.
\end{definition}

If $X=U$ and $Y=V$ are open subsets of the plane, then an $\eta$-quasisymmetric orientation-preserving homeomorphism $f\colon U\to V$ is $M$-quasiconformal, where $M$ depends only on $\eta$; this is proved in Chapter 4 of \cite{Va}, and in particular in Theorem 34.1. The converse does not hold without extra assumptions. More generally, we have:

\begin{lemma}\label{Pre QS implies QC}
If $X$ and $Y$ are locally isometric to open subsets of the plane, then an $\eta$-quasisymmetric map $f\colon X\to Y$ is $M$-quasiconformal, in the sense of Definition \ref{Pre Definition QC metric}. The constant $M>0$ depends only on the distortion function $\eta$.
\end{lemma}

\begin{lemma}\label{Pre Compositions}
Let $X,Y,Z$ be metric spaces that are locally isometric to open subsets of the plane. Also, consider homeomorphisms  $f\colon X\to Y$ and $g\colon Y\to Z$ such that $f$ is $M$-quasiconformal and $g$ is $M'$-quasiconformal. Then the composition $g\circ f\colon  X\to Z$ is $M\cdot M'$- quasiconformal.
\end{lemma}

See \cite[Theorem 13.2]{Va} for the preceding fact, in case the spaces $X,Y,Z$ are Euclidean. We also need the following removability lemma:

\begin{lemma}[Theorem 35.1, \cite{Va}]\label{Pre Removability}
Let $f\colon U\to V$ be an orientation-preserving homeomorphism between open subsets of the plane. Let $A\subset \R^2$ be a closed set, and assume that $f\big|_{U\setminus A}$ is $M$-quasiconformal. If $A$ has $\sigma$-finite Hausdorff $1$-measure, then $f$ is $M$-quasiconformal on $U$.
\end{lemma}

Note that this lemma implies that sets of $\sigma$-finite Hausdorff $1$-measure are locally removable, in the sense of Question \ref{Question local} of the Introduction.

Finally, we need a lemma for quasisymmetric maps from a metric space onto the plane. A metric space $(X,d)$ is \textit{Ahlfors $Q$-regular} for some $Q>0$ if there exists a constant $C\geq 1$ such that for all $x\in X$ and $0<r<\diam(X)$ we have
\begin{align*}
\frac{1}{C}r^Q \leq \mathcal H^Q_d(B_d(x,r)) \leq Cr^Q.
\end{align*}

\begin{lemma}\label{Pre Absolutely continuous}
Let $(X,d)$ be an Ahlfors $2$-regular metric space and assume that there exists a quasisymmetric map $f$ from $X$ onto $\R^2$.  Then the pushforward measure $f_*(\mathcal H^2_d)$ and the Lebesgue measure on $\R^2$ are mutually absolutely continuous.
\end{lemma}
For the proof see \cite[Proposition 4.3]{BM} and the references therein. The authors prove the above statement for a quasisymmetry from $X$ onto the sphere $\widehat{\C}$, but the same proof applies in our case.

\subsection{Convergence of metric spaces}
Here we discuss the notion of pointed Gromov-Hausdorff convergence of metric spaces and the relevant properties.

We follow the approach of \cite[Chapter 8]{BBI}. For a map $f\colon (X,d_X)\to (Y,d_Y)$ between two metric spaces we define its \textit{distortion} by 
\begin{align*}
\dis(f)= \sup\{ |d_X(x,x') -d_Y(f(x),f(x'))|: x,x'\in X\}.
\end{align*}
A \textit{pointed metric space} is a triple $(X,d,p)$, where $d$ is the metric of $X$ and $p\in X$ is a point. 

\begin{definition}\label{Pre Definition pointed GH}
A sequence $(X_n,d_n,p_n)$ of pointed metric spaces converges to a pointed metric space $(X,d,p)$ in the Gromov-Hausdorff sense if the following holds. For every $r>0$ and $\varepsilon>0$ there exists $n_0\in \N$ such that for every $n>n_0$ there exists a (not necessarily continuous) map $f\colon  B_{d_n}(p_n,r)\to X$ such that the following hold:
\begin{enumerate}
\item $f(p_n)=p$,
\item $\dis(f)<\varepsilon$, and
\item the $\varepsilon$-neighborhood of the set $f(B_{d_n}(p_n,r))$ contains the ball $B_d(p,r-\varepsilon)$.
\end{enumerate}
\end{definition}

A metric space $(X,d)$ is \textit{doubling} if there is a constant $N\in \N$ such that each ball of radius $r$ can be covered by at most $N$ balls of radius $r/2$. It is easy to see that if $(X,d)$ is Ahlfors $Q$-regular for some $Q>0$ then it is also doubling, and the implicit constant depends only on $Q$ and the Ahlfors regularity constant. A family of spaces is \textit{uniformly} Ahlfors $Q$-regular, if the implicit constants are the same for all spaces in the family. Similarly, one defines a \textit{uniformly} doubling family of spaces. We will need the following lemma, regarding the convergence of Ahlfors regular spaces:

\begin{lemma}\label{Pre Convergence Ahlfors regular}
Let $(X_n,d_n,p_n)$ be a sequence of uniformly Ahlfors $Q$-regular pointed metric spaces. Suppose that $(X_n,d_n,p_n)$ converges to a space $(X,d,p)$ in the  Gromov-Hausdorff sense. Then the metric space $(X,d,p)$ is Ahlfors $Q$-regular, with implicit constants depending only on the constants of the spaces $(X_n,d_n,p_n)$.
\end{lemma}

This lemma appears in \cite[Lemma 8.29]{DS}. The authors use a different definition for the convergence of metric spaces; see \cite[Definition 8.9]{DS}. This definition is not very handy in practice and we will not use it. However, their definition agrees with our Definition \ref{Pre Definition pointed GH}, in case the spaces involved are uniformly doubling. Finally, we need the following lemma regarding the convergence of a \textit{mapping package}:

\begin{lemma}\label{Pre Mapping package}
Let $(X_n,d_n,p_n)$ and $(Y_n,\rho_n, q_n)$ be pointed metric spaces for $n\in \N$, which are complete and  uniformly doubling. Moreover, suppose that there exists a sequence of $\eta$-quasisymmetric homeomorphisms $f_n\colon X_n\to Y_n$ such that $f_n(p_n)=q_n$ for all $n\in \N$  and there exists a constant $C>0$ and points $x_n\in X_n$ such that 
\begin{align*}
\frac{1}{C}\leq  d_n(p_n,x_n)\leq C \quad \textrm{and} \quad \frac{1}{C} \leq \rho_n( q_n,f(x_n)) \leq C,
\end{align*}
for each $n\in \N$. Then there exist subsequential Gromov-Hausdorff limits $(X,d,p)$ and $(Y,\rho,q)$ of $(X_n,d_n,p_n)$ and $(Y_n,\rho_n,q_n)$, respectively, and there exists a limiting $\eta$-quasisymmetric homeomorphism $f\colon X\to Y$ with $f(p)=q$.
\end{lemma}

This lemma follows from \cite[Lemma 8.22]{DS}, since our assumptions guarantee equicontinuity and uniform boundedness; see also \cite[Corollary 10.30]{He}. This lemma has also appeared in \cite[Lemma 2.4.7]{KL}.

\section{Flap-planes}\label{Section Flap-planes}
\subsection{Definition and general properties}\label{Section General}

\subsubsection{Constructing a flap-plane out of a single tripod}\label{Section Single Tripod}

A \textit{tripod} $G$ is the union of three line segments (also called edges) in the plane, which have a common endpoint, but otherwise they are disjoint; note that their length need not be the same and their angles could vary. The common endpoint of the edges is called the \textit{center} or \textit{central vertex} of the tripod. 

We cut the plane along a tripod $G$, and then glue two rectangles (or else a rectangular pillow) on each slit that arises from cutting an edge $e$ of $G$ with the identifications shown in Figure \ref{fig:Tripodflaps}, so that we obtain a space homeomorphic to the plane. We write $E\sim G$ to denote that a rectangle $E$ is glued to an edge of $G$. The width of each of these rectangles is equal to the the length of the corresponding edge $e$ of $G$, and the height is a prescribed constant $h>0$. Whenever two rectangles are glued along one of their edges, or a rectangle is glued to a slitted edge of $G$, then the gluing map is taken to be the ``identity", namely an isometry. We direct the reader to \cite[Chapters 3.1--3.2]{BBI} for details on gluing length spaces and constructing polyhedral spaces.

The resulting space $S=S(G)$ (the height $h$ is suppressed in the notation) is equipped with its internal metric $d$ and it is homeomorphic to the plane by construction, while the subset of $S$ consisting of the six rectangles attached to $G$ is homeomorphic to a closed Jordan region $\br \Omega$. Topologically, one can think of this construction as cutting the plane along $G$ and ``inserting" a Jordan region $\Omega$ in the plane, whose boundary consists of the six edges of the slitted $G$; see Figure \ref{fig:Tripodflaps}. We identify $S$ with the union of $\R^2\setminus G$ and the rectangles $E\sim G$, after proper identifications.

Let $P\colon S \to \R^2$ denote the ``orthogonal" projection map. This collapses each point of a rectangle $E$ that is glued on top of a slitted edge $e$ to the corresponding point of the edge $e\subset \R^2$. For instance, if a rectangle $E=[a,b]\times [c,d]$ is glued to the edge $e=[a,b]\times\{0\}\subset G$ along its side $[a,b]\times \{0\} \subset E$ with the identity map, then the projection of a point $x=(s,t)\in E$ is $P((s,t))=(s,0)\in e \subset \R^2$. Outside the rectangles $E\sim G$, the map $P$ is the ``identity". The projection of a point $x\in S$ to the plane will be denoted by $\tilde x=P(x)$. Some immediate properties of $P$ are the following:
\begin{enumerate}[\upshape(i)]
\item $P$ is $1$-Lipschitz, i.e.,
\begin{align*}
|\tilde x-\tilde y|\leq d(x,y)
\end{align*}
for all $x,y\in S$, where $|\tilde x-\tilde y|$ denotes the Euclidean distance between $\tilde x$ and $\tilde y$.
\item For all $x,y\in S$ we have
\begin{align*}
d(x,y) \leq |\tilde x -\tilde y| +6h.
\end{align*}
This is because the line segment $[\tilde x,\tilde y]\subset \R^2$ has a lift under $P$ that is a continuum $\gamma$ connecting $x$ and $y$, and whose length inside each glued rectangle is either $h$ or $0$. On the other hand, the number of rectangles that we glue is six (or three two-sided rectangular pillows).
\item Let $\tilde \gamma \subset \R^2$ be a polygonal path that connects $\tilde x$ and $\tilde y$. If $x\in P^{-1}(\tilde x)$ and $y\in P^{-1}(\tilde y) $, then $P^{-1}(\tilde \gamma)$ is a continuum that connects $x$ and $y$. In fact, $P^{-1}(\tilde \gamma)$ contains a polygonal path $\gamma$ with the same property. Here, $\gamma\subset S$ is a polygonal path in the sense that it is the union of finitely many isometric copies of compact intervals, whose endpoints are glued appropriately.
\end{enumerate}

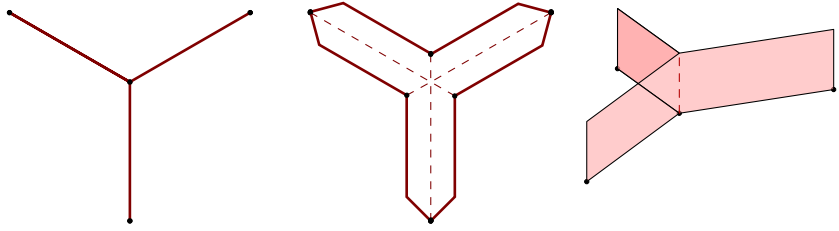
\begin{figure}
\centering
\begin{tikzpicture}

\begin{scope}[scale=.8]

%
%
% tripod
%
%
%

\begin{scope}[xshift=0cm]
\draw [fill=black] (2.,3.4641016151377553) circle (1.0pt) node (A12){};
\draw [fill=black] (6.,3.4641016151377553) circle (1.0pt) node (A13){};
\draw [fill=black] (4.,2.3094010767585043) circle (1.0pt) node (B0){};
\draw [fill=black] (4.,0.) circle (1.0pt) node (A23){};

\draw[color=red!50!black, line width=1pt] (A13.center)--(B0.center)-- (A12.center)--(B0.center)--(A23.center);

\draw [fill=black] (2.,3.4641016151377553) circle (1.0pt) node (A12){};
\draw [fill=black] (6.,3.4641016151377553) circle (1.0pt) node (A13){};
\draw [fill=black] (4.,2.3094010767585043) circle (1.0pt) node (B0){};
\draw [fill=black] (4.,0.) circle (1.0pt) node (A23){};
\end{scope}

%
%
%topological embedding
%
%
%
\begin{scope}[xshift=5cm]

\begin{scope}
\draw [fill=black] (2.,3.4641016151377553) circle (1.0pt) node (A12){};
\draw [fill=black] (6.,3.4641016151377553) circle (1.0pt) node (A13){};
\draw  (4.,2.3094010767585043) node (B0){};
\draw [fill=black] (4.,0.) circle (1.0pt) node (A23){};

\draw[dashed,color=red!50!black] (A13.center)--(B0.center);
\draw[dashed,color=red!50!black] (A12.center)--(B0.center);
\draw[dashed,color=red!50!black] (A23.center)--(B0.center);

\draw[color=red!50!black,line width=1pt] (A23.center)--(4.4,0.4)-- (4.4, 2.08);
\draw[color=red!50!black,line width=1pt]  (3.6,2.08)-- (3.6,0.4)--(A23.center);
\draw[dashed,color=red!50!black] (4.4,2.08)--(B0.center);

\draw [fill=black] (2.,3.4641016151377553) circle (1.0pt) node (A12){};
\draw [fill=black] (6.,3.4641016151377553) circle (1.0pt) node (A13){};
%\draw [fill=black] (4.,2.3094010767585043) circle (1.0pt) node (B0){};
\draw [fill=black] (4.,0.) circle (1.0pt) node (A23){};

\end{scope}

\begin{scope}[shift= {(8,0)},rotate=120]
\draw [fill=black] (2.,3.4641016151377553) circle (1.0pt) node (A12){};
\draw [fill=black] (6.,3.4641016151377553) circle (1.0pt) node (A13){};
%\draw [fill=black] (4.,2.3094010767585043) circle (1.0pt) node (B0){};
\draw [fill=black] (4.,0.) circle (1.0pt) node (A23){};

%\draw (A13.center)--(B0.center)-- (A12.center)--(B0.center)--(A23.center);

\draw[color=red!50!black,line width=1pt] (A23.center)--(4.4,0.4)-- (4.4, 2.08);
\draw[color=red!50!black,line width=1pt]  (3.6,2.08)-- (3.6,0.4)--(A23.center);
\draw[dashed,color=red!50!black] (4.4,2.08)--(B0.center);
\draw [fill=black] (4.4, 2.08) circle (1.0pt); 
\draw [fill=black] (3.6,2.08) circle (1.0pt); 

\draw [fill=black] (2.,3.4641016151377553) circle (1.0pt) node (A12){};
\draw [fill=black] (6.,3.4641016151377553) circle (1.0pt) node (A13){};
%\draw [fill=black] (4.,2.3094010767585043) circle (1.0pt) node (B0){};
\draw [fill=black] (4.,0.) circle (1.0pt) node (A23){};

\end{scope}

\begin{scope}[shift= {(4,6.94)},rotate=-120]
\draw [fill=black] (2.,3.4641016151377553) circle (1.0pt) node (A12){};
\draw [fill=black] (6.,3.4641016151377553) circle (1.0pt) node (A13){};
%\draw [fill=black] (4.,2.3094010767585043) circle (1.0pt) node (B0){};
\draw [fill=black] (4.,0.) circle (1.0pt) node (A23){};

%\draw (A13.center)--(B0.center)-- (A12.center)--(B0.center)--(A23.center);

\draw[color=red!50!black,line width=1pt] (A23.center)--(4.4,0.4)-- (4.4, 2.08);
\draw[color=red!50!black,line width=1pt]  (3.6,2.08)-- (3.6,0.4)--(A23.center);
\draw[dashed,color=red!50!black] (4.4,2.08)--(B0.center);
\draw [fill=black] (4.4, 2.08) circle (1.0pt); 
\draw [fill=black] (3.6,2.08) circle (1.0pt); 

\draw [fill=black] (2.,3.4641016151377553) circle (1.0pt) node (A12){};
\draw [fill=black] (6.,3.4641016151377553) circle (1.0pt) node (A13){};
%\draw [fill=black] (4.,2.3094010767585043) circle (1.0pt) node (B0){};
\draw [fill=black] (4.,0.) circle (1.0pt) node (A23){};

\end{scope}

\end{scope}

%
%
%folded flap
%
%
%

\begin{scope}[yshift=1cm, xshift= 8cm,line cap=round,line join=round,>=triangle 45,x=1.0cm,y=1.0cm, z=0.5cm,rotate around y=10]

%redefine coordinates
%\draw [fill=black] (0.,0,0.) circle (1.0pt) node (A1){};
%\draw [fill=black] (8.,0,0.) circle (1.0pt) node (A2){};
%\draw [fill=black] (4.,0,6.9282032302755105) circle (1.0pt) node (A3){};
\draw [fill=black] (2.,0,3.4641016151377553) circle (1.0pt) node (A12){};
\draw [fill=black] (6.,0,3.4641016151377553) circle (1.0pt) node (A13){};
\draw [fill=black] (4.,0,2.3094010767585043) circle (1.0pt) node (B0){};
\draw [fill=black] (4.,0,0.) circle (1.0pt) node (A23){};
%\draw [fill=black] (3.,0,5.196152422706633) circle (1.0pt) node (){};
%\draw [fill=black] (5.,0,5.196152422706633) circle (1.0pt);
%\draw [fill=black] (7.,0,1.7320508075688776) circle (1.0pt);
%\draw [fill=black] (6.,0,0.) circle (1.0pt);
%\draw [fill=black] (2.,0,0.) circle (1.0pt) node (C3){};
%\draw [fill=black] (4.,0,2.3094010767585043) circle (1.0pt) node (C2){};
%\draw [fill=black] (1.,0,1.7320508075688776) circle (1.0pt) node (C1){};
%\draw [fill=black] (2.333333333333333,0,1.3471506281091274) circle (1.0pt) node (C0){};

%central flaps
\node (B12) at ($(A12.center)+(0,1,0)$){};
\node (B13) at ($(A13.center)+(0,1,0)$){};
\node (B23) at ($(A23.center)+(0,1,0)$){};
\node (B00) at ($(B0.center)+(0,1,0)$){};
\draw[fill=red!40,fill opacity=0.5] (A12.center)--(B12.center)-- (B00.center)-- (B13.center)-- (A13.center)--(B0.center)-- cycle;
\draw[fill=red!40,fill opacity=0.5] (A12.center)--(B12.center)-- (B00.center)-- (B23.center)-- (A23.center)--(B0.center)--cycle;
\draw[dashed,red!70!black] (B0.center)--(B00.center);

\end{scope}

\end{scope}
\end{tikzpicture}
\caption{The tripod $G$ on the left is first slitted and its complement is homeomorphic to the complement of a closed Jordan region $\br \Omega$, as the one bounded by the solid red curve in the middle figure; note that the central vertex of $G$ corresponds to three points in $\partial \Omega$. Then we consider, as depicted, six topological rectangles in this Jordan region $\Omega$, each bounded by the dashed and solid lines. The middle figure shows the (topological) gluing pattern of the six Euclidean rectangles in the figure on the right, which are glued to the slitted edges of the tripod $G$, giving rise to the flap-plane. Although in the figure on the right we only see three rectangles, in fact each of them represents a rectangular pillow having two distinct faces and two slits, along which these two faces are glued to the slitted edge of the tripod and to the other two pillows.}\label{fig:Tripodflaps}
\end{figure}

\subsubsection{Constructing a flap-plane with multiple tripods}\label{Section:flap-plane:multipletripods}

Now, assume that we are given a sequence of tripods in the plane $G_i$, $i\in \N$, such that if the tripods $G_i$ and $G_j$ intersect for $i\neq j$, then $G_i\cap G_j$ is a singleton and more specifically it is a non-central vertex of one of $G_i$ or $G_j$. There are essentially three ways this can occur:
\begin{enumerate}[(i)]
\item a non-central vertex of $G_j$ lies on the central vertex of $G_i$, as in the left of Figure \ref{fig:TwoTripodflaps},
\item a non-central vertex of $G_j$ lies on an open edge of $G_i$,  
\item a non-central vertex of $G_j$ lies on a non-central vertex of $G_i$,
\end{enumerate}
or the above occur with the roles of $i$ and $j$ reversed. More generally, if any collection of planar tripods $\{G_i\}_{i\in I}$ has this property, we say that 
\begin{center}
{ $\{G_i\}_{i\in I}$ possesses \textit{property} $(G)$.}
\end{center}
See Figure \ref{fig:propertyG} for a family of tripods with this property.

\begin{figure}
\centering
\begin{tikzpicture}[scale=.9]

%
%
%tripod gasket
%
%

\begin{scope}[scale=.8, xshift=10cm, line cap=round,line join=round,>=triangle 45,x=1.0cm,y=1.0cm]
\clip(-4.263485552762403,-0.926291939988808) rectangle (13.261987941209707,7.513442620128558);
%\fill[line width=0.8pt,fill=black,fill opacity=0.10000000149011612] (0.,0.) -- (8.,0.) -- (4.,6.9282032302755105) -- cycle;
%\draw  (0.,0.)-- (8.,0.);
%\draw  (8.,0.)-- (4.,6.9282032302755105);
%\draw  (4.,6.9282032302755105)-- (0.,0.);

\draw [line width=1.8pt,color=red!60] (4.,2.309401076758504)-- (6.,3.4641016151377553);
\draw [line width=1.8pt,color=red!60] (4.,2.309401076758504)-- (2.,3.4641016151377553);
\draw [line width=1.8pt,color=red!60] (4.,2.309401076758504)-- (4.,0.);

\draw [line width=1.5pt,color=green!100] (5.666666666666666,1.347150628109127)-- (7.,1.7320508075688776);
\draw [line width=1.5pt,color=green!100] (6.,0.)-- (5.666666666666666,1.347150628109127);
\draw [line width=1.5pt,color=green!100] (5.666666666666666,1.347150628109127)-- (4.,2.309401076758504);

\draw [line width=1.5pt,color=green!40!black] (2.333333333333334,1.347150628109127)-- (4.,2.309401076758504);
\draw [line width=1.5pt,color=green!40!black] (2.333333333333334,1.347150628109127)-- (1.,1.7320508075688776);
\draw [line width=1.5pt,color=green!40!black] (2.333333333333334,1.347150628109127)-- (2.,0.);

\draw [line width=1.5pt,color=green!70!blue] (4.,4.233901974057256)-- (5.,5.196152422706633);
\draw [line width=1.5pt,color=green!70!blue] (4.,4.233901974057256)-- (3.,5.196152422706633);
\draw [line width=1.5pt,color=green!70!blue] (4.,4.233901974057256)-- (4.,2.309401076758504);

\draw [line width=1.pt,color=blue!50!gray] (6.722222222222221,0.7377253439645219)-- (7.5,0.8660254037844388);
\draw [line width=1.pt,color=blue!50!gray] (6.722222222222221,0.7377253439645219)-- (7.,0.);
\draw [line width=1.pt,color=blue!50!gray] (6.722222222222221,0.7377253439645219)-- (5.666666666666666,1.347150628109127);
\draw [line width=1.pt,color=blue!100] (5.579379866483421,2.1948559892660473)-- (6.5,2.5980762113533165);
\draw [line width=1.pt,color=blue!100] (5.579379866483421,2.1948559892660473)-- (5.666666666666666,1.347150628109127);
\draw [line width=1.pt,color=blue!100] (5.579379866483421,2.1948559892660473)-- (4.571472932783599,2.639341128335699);
\draw [line width=1.pt,color=blue!40] (4.888888888888889,0.9990544738012257)-- (4.,1.6500127932945503);
\draw [line width=1.pt,color=blue!40] (4.888888888888889,0.9990544738012257)-- (5.666666666666666,1.347150628109127);
\draw [line width=1.pt,color=blue!40] (4.888888888888889,0.9990544738012257)-- (5.,0.);

\draw [line width=1.pt,color=blue!40] (4.690490977594541,3.7344567071050507)-- (5.5,4.327875);
\draw [line width=1.pt,color=blue!40] (4.690490977594541,3.7344567071050507)-- (4.571472932783599,2.639341128335699);
\draw [line width=1.pt,color=blue!40] (4.690490977594541,3.7344567071050507)-- (4.,4.233901974057256);
\draw [line width=1.pt,color=blue!50!gray] (3.5,6.062177826491071)-- (4.,5.452752542346467);
\draw [line width=1.pt,color=blue!50!gray] (4.5,6.062177826491071)-- (4.,5.452752542346467);
\draw [line width=1.pt,color=blue!50!gray] (4.,5.452752542346467)-- (4.,4.233901974057256);
\draw [line width=1.pt,color=blue!100] (4.,4.233901974057256)-- (3.3107098047776997,3.7337634350792035);
\draw [line width=1.pt,color=blue!100] (3.3107098047776997,3.7337634350792035)-- (2.5,4.3301270189221945);
\draw [line width=1.pt,color=blue!100] (3.3107098047776997,3.7337634350792035)-- (3.4321294143330974,2.637261312258161);

\draw [line width=1.pt,color=blue!40] (3.4321294143330974,2.637261312258161)-- (2.4218209158888104,2.194162717240201);
\draw [line width=1.pt,color=blue!40] (2.4218209158888104,2.194162717240201)-- (1.5,2.5980762113533165);
\draw [line width=1.pt,color=blue!40] (2.4218209158888104,2.194162717240201)-- (2.333333333333334,1.347150628109127);
\draw [line width=1.pt,color=blue!100] (4.,1.6500127932945503)-- (3.1111111111111116,0.9990544738012257);
\draw [line width=1.pt,color=blue!100] (2.333333333333334,1.347150628109127)-- (3.1111111111111116,0.9990544738012257);
\draw [line width=1.pt,color=blue!100] (3.1111111111111116,0.9990544738012257)-- (3.,0.);
\draw [line width=1.pt,color=blue!50!gray] (1.277777777777778,0.7377253439645219)-- (0.5,0.8660254037844388);
\draw [line width=1.pt,color=blue!50!gray] (1.,0.)-- (1.277777777777778,0.7377253439645219);
\draw [line width=1.pt,color=blue!50!gray] (1.277777777777778,0.7377253439645219)-- (2.333333333333334,1.347150628109127);
%\begin{scriptsize}
%\draw [fill=black] (0.,0.) circle (1pt);
%\draw [fill=black] (8.,0.) circle (1pt);
%\draw [fill=black] (4.,6.9282032302755105) circle (1pt);
\draw [fill=black] (2.,3.4641016151377553) circle (1pt);
\draw [fill=black] (4.,0.) circle (1pt);
\draw [fill=black] (6.,3.4641016151377553) circle (1pt);
\draw [fill=black] (4.,2.309401076758504) circle (1pt);
\draw [fill=black] (7.,1.7320508075688776) circle (1pt);
\draw [fill=black] (6.,0.) circle (1pt);
\draw [fill=black] (2.,0.) circle (1pt);
\draw [fill=black] (1.,1.7320508075688776) circle (1pt);
\draw [fill=black] (5.,5.196152422706633) circle (1pt);
\draw [fill=black] (3.,5.196152422706633) circle (1pt);
\draw [fill=black] (5.666666666666666,1.347150628109127) circle (1pt);
\draw [fill=black] (2.333333333333334,1.347150628109127) circle (1pt);
\draw [fill=black] (4.,4.233901974057256) circle (1pt);
\draw [fill=black] (7.5,0.8660254037844388) circle (1pt);
\draw [fill=black] (7.,0.) circle (1pt);
\draw [fill=black] (5.,0.) circle (1pt);
\draw [fill=black] (6.5,2.5980762113533165) circle (1pt);
\draw [fill=black] (4.571472932783599,2.639341128335699) circle (1pt);
\draw [fill=black] (4.,1.6500127932945503) circle (1pt);
\draw [fill=black] (6.722222222222221,0.7377253439645219) circle (1pt);
\draw [fill=black] (5.579379866483421,2.1948559892660473) circle (1pt);
\draw [fill=black] (4.888888888888889,0.9990544738012257) circle (1pt);
\draw [fill=black] (4.690490977594541,3.7344567071050507) circle (1pt);
\draw [fill=black] (5.5,4.327875) circle (1pt);
\draw [fill=black] (4.5,6.062177826491071) circle (1pt);
\draw [fill=black] (3.5,6.062177826491071) circle (1pt);
\draw [fill=black] (2.5,4.3301270189221945) circle (1pt);
\draw [fill=black] (1.5,2.5980762113533165) circle (1pt);
\draw [fill=black] (0.5,0.8660254037844388) circle (1pt);
\draw [fill=black] (1.,0.) circle (1pt);
\draw [fill=black] (3.,0.) circle (1pt);
\draw [fill=black] (3.4321294143330974,2.637261312258161) circle (1pt);
\draw [fill=black] (4.,5.452752542346467) circle (1pt);
\draw [fill=black] (3.3107098047776997,3.7337634350792035) circle (1pt);
\draw [fill=black] (2.4218209158888104,2.194162717240201) circle (1pt);
\draw [fill=black] (3.1111111111111116,0.9990544738012257) circle (1pt);
\draw [fill=black] (1.277777777777778,0.7377253439645219) circle (1pt);
%\end{scriptsize}

\end{scope}

\end{tikzpicture}
\caption{A family of tripods possessing property (G). Different tripods meeting at a point are denoted by different colors.}
\label{fig:propertyG}
\end{figure}
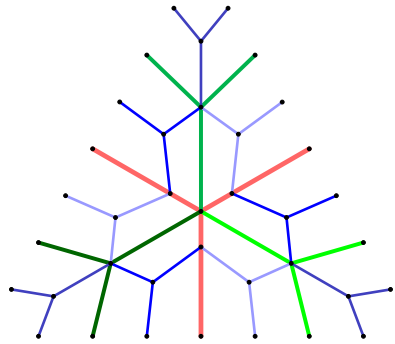

We wish to use the tripods $G_1,\dots,G_n$ in order to construct a flap-plane $S_n=S(G_1,\dots,G_n)$ that ``distinguishes" between the tripods $G_1,\dots,G_n$, in the sense that natural projections can be defined from $S_n$ onto flap-planes $S_k$, $k<n$. To achieve this, we do not glue the rectangles corresponding to the flap-planes of different tripods, even if the tripods intersect each other (the intersection can contain at most one point).
 
More precisely, we can construct a flap plane $S_2=S(G_1,G_2)$ as follows. If $G_2\cap G_1=\emptyset$, then we can construct the flap-plane $S_2$ by cutting the plane along $G_2$ and attaching rectangles of height $h_2$, as before. In this case, the rectangles corresponding to $G_1$ and $G_2$ do not intersect. If, however, $a\in G_2\cap G_1$, then we cut the plane along $G_2$ and attach rectangles to $G_2$, but the rectangles of $G_2$ containing $a$ are not attached to any of the rectangles corresponding to $G_1$, except at the point $a$; see Figure \ref{fig:TwoTripodflaps}. We also provide an alternative way to construct $S_2$, which shows that $S_2$ is homeomorphic to the plane. First consider the flap-plane $S_1=S(G_1)$ corresponding to $G_1$, with associated height $h_1$. Using the property $(G)$ of the tripods $G_1$ and $G_2$ and studying their relative positions we see that $S(G_1)$ contains a (unique) ``distinguished" homeomorphic copy $\widetilde G_2$ of $G_2$ that projects homeomorphically onto $G_2 \subset \R^2$ under the restriction of the projection $P_1\coloneqq P$; namely, if a non-central vertex of $G_2$ lies on $G_1$, then this copy is the closure in $S(G_1)$ of the preimages under $P_1$ of the open edges of $G_2$. We first glue the rectangles of height $h_2$ to $G_2$, disregarding the presence of $G_1$, and obtain a flap-plane $S(G_2)$ in this way. Then one uses the projection $P_1\colon S(G_1)\to \R^2\supset G_2$ in order to glue the rectangles attached to $G_2$ to the ``distinguished" homeomorphic copy of $G_2$ in $S(G_1)$. Topologically, we are just cutting the tripod $\widetilde G_2 \subset S(G_1)\simeq \R^2$ and we are ``inserting" a Jordan region, so the resulting space $S_2$ is also homeomorphic to $\R^2$.

\begin{figure}
\centering
\input{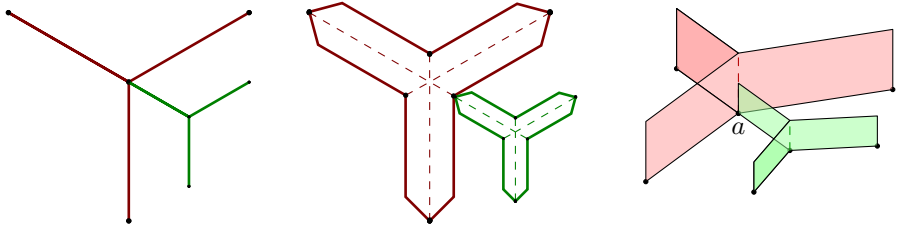}
\caption{The flap-plane (right) corresponding to two tripods $G_1,G_2$ (left) that intersect at a point. The rectangles attached to $G_1$ are not glued to any of the rectangles attached to $G_2$, except at a single point $a$; see the middle figure for the gluing pattern.}\label{fig:TwoTripodflaps}
\end{figure}

We identify $S(G_1,G_2)$ with the union of $\R^2\setminus (G_1\cup G_2)$ and the rectangles attached to $G_1$ and $G_2$, after proper identifications as explained above. One can define natural projections $P_2 \colon S_2\to S_1$ and $P_{2,0}\colon S_2\to \R^2$. Suppose that $E_2\sim G_2$ is a rectangle $E_2=[a,b]\times [c,d]$ and it is attached to the edge $e_2=[a,b]\times \{0\}\subset G_2$. Then for $(s,t)\in E_2$ with $(s,0)\in \R^2\setminus G_1$ we have $P_2((s,t))=(s,0)\in S_1$ (after identifying points of $S_1$ lying outside the rectangles $E\sim G_1$ with $\R^2\setminus G_1$). If $(s,0)\in G_1$, then $P_2((s,t))$ is the point of $S_1$ lying on the base of a rectangle $E_1\sim G_1$, which is in fact the ``intersection" point of $E_2$ and $E_1$ in $S_2$; this would be the point $a$ in Figure \ref{fig:TwoTripodflaps}. For points of $S_2$ that do not lie on rectangles $E\sim G_2$ the map $P_2$ is defined to be the ``identity". The map $P_{2,0}$ is defined to be the ``identity" on $\R^2\setminus (G_1\cup G_2)$ and on the rectangles attached to $G_1$ and $G_2$ it is defined to be the ``orthogonal" projection as above. One sees that $P_{2,0}=P_1\circ P_2$.

Using the property $(G)$ and induction on the number of tripods, for each $n\in \N$ we can define the flap-plane $S_n=S(G_1,\dots,G_n)$, with associated heights $h_1,\dots,h_n$ for the rectangles corresponding to $G_1,\dots,G_n$, respectively, and also define natural projections $P_n\colon S_n\to S_{n-1}$, $P_{n,0}\colon S_n\to \R^2$ so that:
\begin{itemize}
\item $S_n$ is homeomorphic to the plane and
\item if $G_0\subset \R^2$ is a tripod such that the family $\{G_0\}\cup \{G_i: i\in \{1,\dots,n\} \}$ possesses property $(G)$, then $S_n$ contains a ``distinguished" homeomorphic copy $\widetilde G_0$ of $G_0$ that projects homeomorphically onto $G_0 \subset \R^2$ under the restriction of the projection $P_{n,0}\colon S_n\to \R^2$.
\end{itemize}
We provide a sketch of the proof. Suppose that $S_{n-1}$ has the above two properties and let $G_0\subset \R^2$ be a tripod such that the family $\{G_0\}\cup \{G_i: i\in \{1,\dots,n\} \}$ possesses property $(G)$. Since both families $\{G_i : i\in \{1,\dots,n\}\}$ and $\{G_0\}\cup \{G_i :i\in \{1,\dots,n-1\}\}$ have property $(G)$, it follows that $S_{n-1}$ contains ``distinguished" homeomorphic copies $\widetilde G_n$ and $\widetilde G_0$ of $G_n$ and $G_0$, respectively. Since $S_{n-1}$ is homeomorphic to $\R^2$, we may think of $\widetilde G_{n}$ and $\widetilde G_0$ as (topological) tripods in the plane with non-straight edges. Property $(G)$ imposes the same restrictions on the relative positions of $\widetilde G_n$ and $\widetilde G_0$. For the construction of $S_n$, one cuts the plane along this ``distinguished" tripod $\widetilde G_n$ and attaches the rectangles, or else, ``inserts" a Jordan region whose boundary consists of the six edges of the slitted $\widetilde G_n$. This implies that $S_n$ is homeomorphic to the plane. Then the restrictions on the relative positions of $\widetilde G_n$ and $\widetilde G_0$ imply that $S_n$ contains the desired ``distinguished" homeomorphic copy of $\widetilde G_0$ and thus of $G_0$.

We now record some important properties of the flap-planes $S_n$. We endow the space $S_n$ with its internal path metric $d_n$. With this metric $(S_n,d_n)$ is homeomorphic to the plane, and it is also complete and locally compact. We also set $S_0=\R^2$. The notation $E\sim G_i$ is used to denote that a rectangle $E\subset S_n$ is glued to an edge of $G_i$, $i\in \{1,\dots,n\}$. The space $S_n$ is regarded as the union of $\R^2\setminus \bigcup_{i=1}^n G_i$ and the rectangles $E\sim G_i$, $i\in \{1,\dots,n\}$, with proper identifications. The metric $d_n$ has the property that is locally isometric to $d_{n-1}$, away from the rectangles $E\sim G_n$. Furthermore, $d_n$ restricted to a rectangle $E\sim G_i$ is locally isometric to the planar metric on the rectangle $E$.  For any two points $x,y\in S_n$ there exists a (not necessarily unique) geodesic $\gamma$ that connects them, with $d_n(x,y)= \length_{d_n}(\gamma)$; see \cite[Section 2.5.2]{BBI}. We will use the notation $B_n(x,r)$ for a ball in $S_n$, whenever it is more convenient than the notation $B_{d_n}(x,r)$.

Consider now the natural projections $P_{k}\colon S_k\to S_{k-1}$ for $k\geq 1$. These natural projections are defined similarly to $P_1$ and $P_2$ as above. We also define projections $P_{k,l}\colon S_k\to S_{l}$ by $P_{k,l}=P_k\circ \cdots\circ P_{l+1}$ for $0\leq l\leq k-1$ and let $P_{k,k}$ be the identity map on $S_k$. These projections satisfy the following:
\begin{enumerate}[($G$\upshape1)]
\item\label{G-Lipschitz} $P_k$ and $P_{k,l}$ are $1$-Lipschitz. 
\item\label{G-Lipschitz Upper inequality} For all $x,y\in S_k$ we have
\begin{align*}
d_k(x,y)\leq d_{k-1}(P_{k}(x),P_k(y)) +6h_k.
\end{align*}
\item\label{G-lift paths} If $\gamma^* \subset S_{k-1}$ is a polygonal path connecting $x^*$ and $y^*$, then $ P_k^{-1}(\gamma^*)$ contains a polygonal path that connects $P_k^{-1}(x^*)$ to $P_k^{-1}(y^*)$. 
\item\label{G-compatible} The projections are compatible, i.e., for $0\leq m\leq l\leq k$ we have
\begin{align*}
P_{k,m}= P_{l,m}\circ P_{k,l}. 
\end{align*}
\item\label{G-Geodesic} If $\gamma^*$ is a geodesic in $S_{k-1}$ that does not intersect the tripod $G_k$ (or rather the lift $P_{k-1,0}^{-1}(G_k)$), then $\gamma^*$ lifts isometrically under $P_k$ to a unique  geodesic $\gamma \subset S_k$. Conversely, if $\gamma \subset S_k$ is a geodesic that does not intersect the rectangles $E\sim G_k$, then its projection $\gamma^*$ is also a geodesic in $S_{k-1}$ and is isometric to $\gamma$.
\end{enumerate}
The latter property, implies the next important property:
\begin{enumerate}
\item[\mylabel{G-Ball}{($G$6)}] If a ball $B_{k-1}(x^*,r)\subset S_{k-1}$ does not intersect the tripod $G_k$, then $x^*$ has a unique preimage $x$ under $P_k$, and $P_k^{-1}(B_{k-1}(x^*,r))=B_k(x,r)$.  Conversely, if $B_k(x,r) \subset S_k$ does not intersect the rectangles $E\sim G_k$, then it projects onto a ball $B_{k-1}(x^*,r)$.
\end{enumerate}
In the same spirit, we have the following property for the Hausdorff $2$-measure $\mu_k$ on $S_k$:
\begin{enumerate}
\item[\mylabel{G-Measure}{($G$7)}] If $A\subset S_k$ is a Borel set, then 
\begin{align*}
\mu_k( A) \geq \mu_{k-1}(P_k(A)).
\end{align*}
Moreover, if the set $A$ does not intersect $G_k$, then we obtain equality.
\end{enumerate}
We finally record a property of the metric $d_k$ on the rectangles $E$ glued to the tripod $G_i$: 
\begin{enumerate}
\item[\mylabel{G-Rectangles}{($G$8)}] The metric $d_k$ restricted on a rectangle $E\sim G_i$, $i\in\{1,\dots,k\}$, is isometric, not only locally but also globally, to the Euclidean metric on $E$.
\end{enumerate}

\subsubsection{The inverse limit \texorpdfstring{$S_\infty$}{S(infinity)} of \texorpdfstring{$S_n$}{Sn}}\label{Section:flap-plane:inverse}

We consider the set $T_n=\bigcup_{i=1}^n G_i$ as a (possibly disconnected) planar graph, whose vertices are the vertices of $G_1,\dots,G_n$. Note that an edge $e$ of $G_j$ might be ``cut" into two (or more) edges, if a vertex of $G_i$, $i\neq j$, lies in the interior of the edge $e$. See Figure \ref{fig:propertyG} for such a graph $T_n$. We also consider the ``graph" $T_\infty =\bigcup_{i=1}^\infty G_i$. A point $x\in \R^2$ is a vertex of $T_\infty$ (by definition) if $x$ is a vertex of $T_n$ for all sufficiently large $n\in \N$. We suppose that $T_\infty$ has \textit{uniformly bounded degree}, that is, there exists $N_0>0$ such that the degree of each vertex $x$ of $T_n$ is bounded by $N_0$ for all $n\in \N$.

Under these assumptions, we can identify the \textit{inverse limit} $S_\infty$ of the sequence of spaces $S_n$. The set $S_\infty$ is the subset of $\prod_{n=0}^\infty S_n$ consisting of points $z=(z_0,z_1,\dots)$ with the property that $P_n(z_n)=z_{n-1}$ for all $n=1,2,\dots$. We will find a simpler representation for $S_\infty$. 

Note that if $z_0\in \R^2\setminus \bigcup_{i=1}^\infty G_i$ then $z_0=z_1=\dots$, recalling that we have identified points of $S_n$ not lying on any tripod with $\R^2\setminus \bigcup_{i=1}^n G_i$. Hence, in this case we may identify the point $z\in S_\infty$ with $z_0\in \R^2\setminus \bigcup_{i=1}^\infty G_i$ and say that $z\in S_n$ for all $n=0,1,\dots$. 

Now, if $z_0$ lies on a tripod $G_i$ and is not a vertex of $T_\infty$, then $z_i\in S_i$ lies on a rectangle attached to $G_i$ and $z_k$ has a unique preimage under $P_{k+1}$ for all $k\geq i$, since no rectangle is attached to $z_i$ in order to obtain the spaces $S_k$, $k> i$. Hence, we may write that $z_i\in S_k$ for all $k\geq i$ and we may identify the point $z\in S_\infty$ with $z_i$, thus saying that $z$ lies in $S_n$ for all sufficiently large $n$.   

Finally, suppose that $z_0$ lies on a tripod $G_i$ and it is a vertex of $T_\infty$. Since $T_\infty$ has uniformly bounded degree, it follows that there are finitely many tripods containing $z_0$, and we may assume that $i$ is the largest index with $z_0\in G_i$. It follows that $z_i\in S_i$, and as in the previous case, no rectangle is attached to $z_i$ in order to obtain the spaces $S_k$, $k>i$. Hence, we may write again that $z_i\in S_k$ for all $k\geq i$ and we may identify the point $z\in S_\infty$ with $z_i$, thus saying that $z$ lies in $S_n$ for all sufficiently large $n$.   

Summarizing, if $z\in S_\infty$, then $z$ lies in (or rather projects to) $S_n$ for all sufficiently large $n$. With this in mind, we may also identify $S_\infty$ with the union of $\R^n\setminus \bigcup_{i=1}^\infty G_i$ and the rectangles attached to each $G_i$, after proper identifications.

\begin{remark}
If we had not assumed that the degree of $T_\infty$ is uniformly bounded, then we would not be able to represent the inverse limit $S_\infty$ as above, since it could contain points that do not lie on $\R^2\setminus \bigcup_{i=1}^\infty G_i$ or in any rectangle attached to the tripods.
\end{remark}

\begin{prop}\label{General-Completeness}
If the heights $\{h_n\}_{n\in \N}$ of the rectangles attached to the tripods are chosen so that
\begin{align*}
\sum_{i=1}^\infty h_i<\infty
\end{align*}
then the inverse limit $(S_\infty,d_\infty)$ of $(S_n,d_n)$  is a complete metric space.
\end{prop}

\begin{proof}
The metric of $d_\infty$ is defined as the limit of $d_n$. To ensure that this exists, we fix $x,y\in S_\infty$, so $x,y\in S_{n}$ for all sufficiently large $n$, and note that 
\begin{align*}
d_{n-1}(x,y)=d_{n-1}( P_n(x),P_n(y))\leq d_n( x,y)
\end{align*}
if $n$ is sufficiently large, by the Lipschitz property \ref{G-Lipschitz} of the projections; here we have considered the identifications of $x,y$ with $P_n(x),P_n(y)$, respectively, for large $n$. Hence, the sequence $d_n(x,y)$ converges, possibly to $\infty$. To exclude this possibility, we apply repeatedly property \ref{G-Lipschitz Upper inequality} together with the compatibility \ref{G-compatible} and obtain
\begin{align*}
d_n(x,y) \leq d_{n-1}(P_n(x),P_n(y))+ 6h_n \leq |\tilde x-\tilde y|+ 6\sum_{i=1}^n h_i,
\end{align*}
where $\tilde x,\tilde y$ are the projections of $x,y$, respectively, to the plane. By assumption, the last sum is convergent, so our claim is proved.

Now, we show that $(S_\infty,d_\infty)$ is complete. Let $\{x_k\}_{k\in \N}$ be a Cauchy sequence in $S_\infty$. If $x_k$ lies infinitely often in a given rectangle $E \subset S_i$, then $x_k$ has a convergent subsequence, since the metric $d_\infty$, restricted on $E$, is isometric to the Euclidean metric, by the limiting version of \ref{G-Rectangles}. Hence, we may assume that $x_k$ either lies in $\R^2 \setminus \bigcup_{i=1}^\infty G_i$ infinitely often, or it has a subsequence, still denoted by $x_k$, that lies in distinct rectangles $E_k\subset S_{i_k}$. 

In the first case, $x_k$ (or rather its projection to the plane) converges in the Euclidean metric to a point $x\in \R^2$, because the projection of $S_\infty$ to the plane is $1$-Lipschitz.  If $x\in \R^2\setminus \bigcup_{i=1}^\infty G_i$, then the line segment $[x_k,x]\subset \R^2$ does not intersect any given tripod $G_i$ for sufficiently large $k$. Hence, for each $i_0\in \N$ there exists $k_0\in \N$ such that for $k\geq k_0$ the segment $[x_k,x]$ does not intersect $G_1,\dots,G_{i_0}$. For $n\geq i_0+1$, by a repeated application of \ref{G-Lipschitz Upper inequality}, we have  
\begin{align*}
d_{n}(x,x_k) &\leq d_{n-1}(x,x_k) + 6h_n\\
&\leq d_{i_0}(x,x_k) +6\sum_{i=i_0+1}^n h_i.
\end{align*} 
Now, applying repeatedly \ref{G-Geodesic}, we note that $d_{i_0}(x,x_k)= |x-x_k|$, which is the length of the geodesic $[x_k,x]$ in the plane. Hence,
\begin{align*}
d_n(x,x_k) \leq |x-x_k| +6\sum_{i=i_0+1}^\infty h_i,
\end{align*}
and passing to the limit we have
\begin{align*}
d_\infty(x,x_k) \leq |x-x_k| +6\sum_{i=i_0+1}^\infty h_i
\end{align*}
for all $k\geq k_0$. This implies that $d_\infty (x,x_k)\to 0$, as desired. 

On the other hand, if $x$ lies on a tripod, and $x$ is  a vertex of $T_\infty$, then there exists $N_x\in \N$ such that the degree of $x$ in $T_n$ is equal to $N_x\leq N_0$ for all sufficiently large $n\in \N$; recall that $N_0$ is a uniform bound on the degree of the graphs $T_n$. For a small $r>0$ the edges of tripods $G_n$ that meet at $x$ split the ball $B(x,r)\subset \R^2$ into $N_x$ components. Each of these components contains $x$ in its boundary and is a circular sector  if $r$ is sufficiently small, so in particular, it is convex.  One of these components, say $V$, must contain infinitely many points $x_k$, and thus all, after passing to a subsequence. We let $G_{n_0}$ be a tripod such that $x\in G_{n_0}$, and also $G_{n_0}$ contains one of the edges that bounds the sector $V$.  We let $x_0$ be the point of $S_{n_0}$ that projects to $x$, is accessible from $V$, and  lies on  the boundary of a rectangle $E\sim G_{n_0}$. We claim that $x_k \to x_0$ in $d_\infty$. We look at the open segments $(x_k,x)$ and note that they do not intersect $G_{n_0}$ or any other tripod $G_i$ infinitely often. Arguing as before, we have that for all $i_0\in \N$ there exists $k_0\in \N$ such that
\begin{align*}
d_\infty( x_0,x_k) \leq |x-x_k| + 6\sum_{i=i_0+1}^\infty h_i,
\end{align*}
for $k\geq k_0$. This shows convergence. The case that $x\in G_{n_0}$ but $x$ is not a vertex of $T_\infty$ is treated in the same way, and here $B(x,r) \setminus G_{n_0}$ contains only two components, provided that $r$ is small.

The last case is to assume that $x_k\in E_k\sim G_{i_k}$, where $E_k$ are rectangles of height $h_{i_k}$, and $h_{i_k} \to 0$, since the rectangles $E_k$ are distinct. In this case, we can find a point $y_k$ in the ``base" of the rectangle $E_k$ that is glued to $\R^2$ such that 
\begin{align*}
d_\infty(x_k,y_k)\leq h_{i_k},
\end{align*}
by \ref{G-Rectangles}. Hence, $y_k$ is also Cauchy in $d_\infty$, and it suffices to show that it converges, because $h_{i_k}\to 0$. Arbitrarily close to each $y_k$ we can find a point $z_k \in \R^2 \setminus \bigcup_{i=1}^\infty G_i$, with $d_\infty(y_k,z_k) \leq 1/k$. This is justified as in the previous paragraph and using the observation that $\bigcup_{i=1}^\infty G_i$ has empty interior (e.g., using the Baire category theorem). Now, the convergence of $z_k$ is obtained by the previous case.
\end{proof}

\begin{remark}\label{General-Metrics equivalent}
The preceding proof, together with the Lipschitz property \ref{G-Lipschitz}, show that if we restrict the metric $d_\infty$ of $S_\infty$ to $\R^2\setminus \bigcup_{i=1}^\infty G_i$, then it is topologically equivalent to the Euclidean metric, in the sense that  a sequence $x_k \in \R^2\setminus \bigcup_{i=1}^\infty G_i$ converges to a point $x_0\in S_\infty$ if and only if the projections $x_k$ converge to the projection of $x_0$ in the Euclidean metric.
\end{remark}

\begin{remark}\label{General-Properties S infinity}
There exist natural projections from $S_\infty$ onto $S_n$ and $\R^2$ for all $n\in \N$. Analogs of properties \ref{G-Lipschitz}, \ref{G-compatible} and \ref{G-Measure} also hold for these projections. Moreover, property \ref{G-Rectangles} is true for $S_\infty$, so the metric $d_\infty$ restricted on a rectangle $E\sim G_i$ for $i\in \N$ is isometric to the Euclidean metric on $E$. 
\end{remark}

\begin{prop}\label{General-Gromov Hausdorff}
If the heights $\{h_n\}_{n\in \N}$ are chosen as in Proposition \ref{General-Completeness}, then for each $p\in S_\infty$ there exist points $p_n\in S_n$ such that the sequence of spaces $(S_n,d_n,p_n)$ converges to $(S_\infty,d_\infty,p)$ in the pointed Gromov-Hausdorff sense of Definition \ref{Pre Definition pointed GH}.
\end{prop}
\begin{proof}
Let $p\in S_\infty$. Then, by the representation of $S_\infty$ that we gave, $p$ lies in $S_n$ for all sufficiently large $n$, so we set $p_n=p\in S_n$ for, say, $n\geq n_0$. We fix $n\geq n_0$ and $r,\varepsilon>0$  and define $f\colon S_n\to S_\infty$ to be any right inverse of the projection from $S_\infty$ onto $S_n$, which is surjective. Clearly, $f(p_n)=p$. 

Let $\tilde x,\tilde y\in S_n$ be arbitrary. Also, consider arbitrary lifts $x,y\in S_\infty$ of $\tilde x,\tilde y$, respectively. Then $x,y\in S_m$ for sufficiently large $m\in \N$. By \ref{G-Lipschitz} and \ref{G-Lipschitz Upper inequality}, for sufficiently large $m\geq n$ we have
\begin{align*}
0\leq d_{m}(x,y)-d_n(\tilde x,\tilde y) \leq 6\sum_{i=n+1}^\infty  h_i.
\end{align*}
Letting $m\to\infty$ yields
\begin{align}\label{General-Gromov Hausdorff-ineq}
0\leq d_{\infty}(x,y)-d_n(\tilde x,\tilde y) \leq 6\sum_{i=n+1}^\infty  h_i.
\end{align}
Since $\sum_{i=n+1}^\infty h_i \to 0$ as $n\to\infty$, this shows that $\dis(f)$ can be made less than $\varepsilon$, if $n$ is sufficiently large (independent of $\tilde x,\tilde y$). Hence, condition (2) of Definition \ref{Pre Definition pointed GH} holds.

Finally, we check condition (3). If $d_\infty(p,x)<r-\varepsilon$, then the projection $\tilde x$ of $ x$ to $S_n$ satisfies $d_n(p_n,\tilde x) \leq d_\infty(p,x)<r-\varepsilon<r$, for $n\geq n_0$; see property \ref{G-Lipschitz} and Remark \ref{General-Properties S infinity}. This implies that $\tilde x\in B_{d_n}(p_n,r)$. On the other hand, $f(\tilde x)$ is a lift of $\tilde x$ and the points $f(\tilde x)$ and $x$ project to the same point of $S_n$. Thus, by \eqref{General-Gromov Hausdorff-ineq} we obtain
\begin{align*}
d_\infty( f(\tilde x), x) \leq  6\sum_{i=n+1}^\infty  h_i.
\end{align*}
If $n$ is sufficiently large (independent of $x$), then the above is less than $\varepsilon$, so $x$ is contained in the $\varepsilon$-neighborhood of $f(B_{d_n}(p_n,r))$, as desired.
\end{proof}

\begin{remark}\label{General-Graph dependence}
The preceding results, Proposition \ref{General-Completeness} and Proposition \ref{General-Gromov Hausdorff}, do not depend on the geometry of the tripods $G_n$. Hence, they also hold if the tripods $G_n$, $n\in \N$, are not a priori given to us, but they are constructed inductively, based on the previous tripods. In fact, once we have the tripods $G_1,\dots,G_{n-1}$ and the corresponding heights $h_1,\dots,h_{n-1}$, then the tripod $G_n$ can depend both on $G_1,\dots,G_{n-1}$ \textbf{and} the heights $h_1,\dots,h_{n-1}$! In particular, in this more general setting, one can still obtain a limiting space $(S_\infty, d_\infty)$, as long as $\sum_{i=1}^\infty h_i<\infty$, and the degree of $T_\infty$ is uniformly bounded. A different way to think of that is as follows. Suppose we have a ``machine" or an algorithm that produces the tripod $G_n$, based on $G_1,\dots,G_{n-1}$ and $h_1,\dots,h_{n-1}$. Then we can choose $h_n$ to be sufficiently small and feed this back into the ``machine" to obtain the next tripod $G_{n+1}$. The point of this remark will be evident in Section \ref{Section QC homeo}, where a sequence of tripods is constructed inductively and is used to build a flap-plane.
\end{remark}

\subsubsection{Quasisymmetric embedding of \texorpdfstring{$S_\infty$}{S(infinity)} into the plane}
The main result in Section \ref{Section Flap-planes} is the next theorem. 

\begin{theorem}\label{General-QS Embedding}
If the heights $\{h_n\}_{n\in \N}$ are chosen to be sufficiently small (depending on the tripods $G_n$), then there exists a quasisymmetry from $(S_\infty,d_\infty)$ onto $\R^2$ with the Euclidean metric. Furthermore, $(S_\infty,d_\infty)$ is Ahlfors $2$-regular.
\end{theorem}

Recall that a standing assumption is that the degree of the ``graph" $T_\infty$ is uniformly bounded. Again, the heights $h_n$ have to satisfy the condition in Proposition \ref{General-Completeness}, but this time we have more restrictions as it will be evident from the proof. Moreover, as in Remark \ref{General-Graph dependence}, the choice of the heights is to be interpreted as follows: for each $n\in \N$ the height $h_n$ has to be chosen to be sufficiently small, but only depending on the tripods $G_1,\dots,G_n$. Also, the tripod $G_n$ is not necessarily a priori given to us, but it can depend on $G_1,\dots,G_{n-1}$ and the already chosen heights $h_1,\dots,h_{n-1}$. In the latter scenario, it is necessary that $G_n$ is chosen so that the family $\{G_1,\dots,G_n\}$ has the property $(G)$ for all $n\in \N$.

The proof of Theorem \ref{General-QS Embedding} is based on the Bonk-Kleiner Theorem:
\begin{theorem}[Theorem 1.1, \cite{BK}]\label{General-Bonk-Kleiner}
Let $(X,d)$ be an Ahlfors $2$-regular metric space homeomorphic to the sphere $S^2$, equipped with the spherical metric. Then $(X,d)$ is quasisymmetric to $S^2$ if and only if $(X,d)$ is linearly locally connected.
\end{theorem}

See Section \ref{Section LLC} for the definition of linear local connectivity. This theorem gave an excellent criterion for quasisymmetric parametrizability of 2-dimensional surfaces, and a necessary and sufficient condition is still to be found. Since its publication, there have been some improvements and generalizations \cite{Wi}, \cite{MW}, and very recently two new proofs of the theorem were published \cite{Ra}, \cite{LW}, which give an alternative perspective to the problem of finding parametrizations of 2-dimensional surfaces. We direct the reader to \cite{Ra} for more references and background on the problem of uniformization of 2-dimensional surfaces.

In fact, we will need a plane version of this theorem, which was proved by Wildrick:

\begin{theorem}[Theorem 1.2, \cite{Wi}]\label{General-Wildrick}
Let $(X,d)$ be an Ahlfors $2$-regular and linearly locally connected metric space that is homeomorphic to the plane $\R^2$. If $(X,d)$ is unbounded and complete then $(X,d)$ is quasisymmetrically equivalent to $\R^2$ with the Euclidean metric.

The statement is quantitative, in the sense that the distortion function of the quasisymmetry can be chosen to depend only on the constants associated to the Ahlfors $2$-regularity and linear local connectivity. 
\end{theorem}

In the next two sections we prove that the spaces $(S_n,d_n)$ are Ahlfors $2$-regular and linearly locally connected with uniform constants, under the assumptions of Theorem \ref{General-QS Embedding}. The proof of Theorem \ref{General-QS Embedding} is completed in Section \ref{Section Embedding}.

\subsection{Ahlfors regularity}\label{Section Ahlfors}
Recall that a metric space $(X,d)$ is Ahlfors $Q$-regular for some $Q>0$ if there exists a constant $C\geq 1$ such that for each $x\in X$ and for all $0<r<\diam(X)$ we have
\begin{align}\label{Ahlfors-Definition}
\frac{1}{C} r^Q \leq \mathcal H^Q_d(B_d(x,r)) \leq C r^Q,
\end{align}
where $\mathcal H^Q_d$ denotes the Hausdorff $Q$-measure.

In this section we prove, under the assumptions of Theorem \ref{General-QS Embedding}, that the spaces $(S_n,d_n)$ are Ahlfors 2-regular with uniform constants. Remark \ref{General-Graph dependence} also applies here, i.e., at the $n$-the stage the tripod $G_n$ need not be given to us, but it can depend on the tripods $G_1,\dots,G_{n-1}$ and the heights $h_1,\dots,h_{n-1}$.

In fact, we only need to show the right inequality in \eqref{Ahlfors-Definition}. The left inequality will follow from right inequality and the linear local connectivity of $(S_n,d_n)$ that is discussed in the next Section \ref{Section LLC}; see Proposition \ref{LLC implies 2-regular} and Proposition \ref{LLC}.

\begin{prop}\label{Ahlfors}
Assume that the degree of $T_n$ is bounded by $N_0>0$, for all $n\in \N$. If the heights $\{h_n\}_{n\in \N}$ are chosen to be sufficiently small, then the spaces $(S_n,d_n)$ are Ahlfors $2$-regular, with uniform constants, depending only on $N_0$.
\end{prop}

As we remarked, we will only show the upper bound in \eqref{Ahlfors-Definition}. We denote the Hausdorff $2$-measure of $(S_n,d_n)$ by $\mu_n$, and we use the ball notation $B_n(x,r)$, instead of $B_{d_n}(x,r)$.

\begin{lemma}\label{Ahlfors-Upper}
Assume that the degree of $T_n$ is bounded by $N_0>0$, for all $n\in \N$. If the heights $\{h_n\}_{n\in \N}$ are chosen to be sufficiently small, then there exists a constant $c>0$ such that for all $n\in \N$, $x\in S_n$, and $r>0$ we have
\begin{align*}
\mu_n(B_n(x,r)) \leq c r^2.
\end{align*} 
The constant $c$ depends only on $N_0$.
\end{lemma}
\begin{proof}
Using \ref{G-Measure}, we write
\begin{align}\label{Ahlfors-Upper-Equality}
\mu_n(B_n(x,r))= \mu_0(P_{n,0}(B_n(x,r)))+ \sum_{k=1}^n \mu_n\left( \bigcup_{E\sim G_k}B_n(x,r)\cap E\right),
\end{align}
where $\mu_0$ is the Lebesgue measure in $\R^2$ and $P_{n,0}$ is the projection from $S_n$ to $\R^2$. By the Lipschitz property \ref{G-Lipschitz} we have
\begin{align*}
\mu_0(P_{n,0}(B_n(x,r))) \leq \mu_0(B( P_{n,0}(x),r)) =\pi r^2.
\end{align*}
Also, if $E\sim G_k$ and $y\in B_n(x,r)\cap E$, then $B_n(x,r)\cap E \subset B_n(y,2r)\cap E$. On the other hand, $B_n(y,2r)\cap E$ is isometric to the intersection of a Euclidean ball of radius $2r$ with the Euclidean subset $E$ by \ref{G-Rectangles}, therefore
\begin{align}\label{Ahlfors-N_0 Bound}
\mu_n\left( \bigcup_{E\sim G_k}B_n(x,r)\cap E\right) \leq 6 \cdot 4\pi r^2. 
\end{align}
Recall at this point that $6$ rectangles are attached to a given tripod $G_k$.

If the projection of $B_n(x,r)$ to $\R^2$ intersects at most $N_0$ tripods $G_{i_1},\dots,G_{i_{N}}$, $N\leq N_0$, then by \eqref{Ahlfors-Upper-Equality} we have
\begin{align*}
\mu_n(B_n(x,r)) \leq \pi r^2+ N_0 (6\cdot 4\pi r^2) = (\pi+ 24N_0\pi)r^2.
\end{align*}
This bound is independent of the choice of heights $h_i$ and tripods $G_i$ (cf. Remark \ref{General-Graph dependence}).

We claim that we can choose inductively the height $h_n$ of the rectangles $E\sim G_n$ depending only on the tripods $G_1,\dots,G_n$ such that whenever the projection of a ball $B_n(x,r)$ to the plane intersects some tripods $G_{i_1},\dots,G_{i_{N_0}},G_{i_{N_0+1}}$ with $i_1<\dots<i_{N_0+1}\leq n$, we have
\begin{align}\label{Ahlfors-Upper-Measure of E}
\mu_n\left( \bigcup_{E\sim G_{i_{N_0+1}}} E \right) \leq r^2 ,
\end{align}
and also
\begin{align}\label{Ahlfors-Upper-Geometric}
\mu_n\left( \bigcup_{E\sim G_{i+1}} E\right) \leq \frac{1}{2} \mu_n\left(\bigcup_{E\sim G_{i}}E\right),
\end{align}
for all $i=1,\dots,n-1$.

Assuming that, we finish the proof. Let $B_n(x,r)$ be an arbitrary ball in $S_n$, whose projection intersects the tripods $G_{i_1},\dots,G_{i_{N_0}}, G_{i_{N_0+1}}$, $i_1<\dots<i_{N_0+1}\leq n$. Also, assume that these are the smaller possible such indices, namely there exists no $i \notin \{i_1,\dots,i_{N_0+1}\}$ with $i<i_{N_0+1}$ such that the projection of $B_n(x,r)$ intersects $G_i$. Then we have
\begin{align*}
\sum_{k=1}^n \mu_n\left( \bigcup_{E\sim G_k}B_n(x,r)\cap E\right) &= \sum_{j=1}^{N_0} \mu_n\left( \bigcup_{E\sim G_{i_j}}  B_n(x,r)\cap E \right) \\
&\quad\quad + \mu_n\left( \bigcup_{E\sim G_{i_{N_0+1}}}  B_n(x,r)\cap E \right) \\
&\quad\quad\quad\quad +\sum_{i_{N_0+1}<k\leq n} \mu_n\left( \bigcup_{E\sim G_{k}}B_n(x,r)\cap E \right)\\
&\leq N_0  \cdot  24\pi r^2+  r^2+ \sum_{i_{N_0+1}<k\leq n}   \mu_n\left(\bigcup_{E\sim G_{k}}E\right),
\end{align*}
where we used the bound \eqref{Ahlfors-N_0 Bound} for the first term, and condition \eqref{Ahlfors-Upper-Measure of E} for the second term. By condition \eqref{Ahlfors-Upper-Geometric}, the last term is bounded by 
\begin{align*}
\mu_n\left(\bigcup_{E\sim G_{i_{N_0+1}}}E\right) \cdot \sum_{k=1}^\infty \frac{1}{2^k}\leq r^2.
\end{align*}
This concludes the proof, by \eqref{Ahlfors-Upper-Equality}, with constant $c= \pi+ 24N_0\pi +2$.

Now, we focus on our claim, which we will prove by induction on $n\in \N$. For $n=1,\dots,N_0$ we have nothing to show, since  the projection of a ball $B_n(x,r)$ to $\R^2$ will always intersect at most $N_0$ tripods. We can also adjust the heights $h_i$, $i=1,\dots,N_0$, to be so small, depending on $G_1,\dots,G_{N_0}$, that \eqref{Ahlfors-Upper-Geometric} holds. We assume that the statements hold for some $n\in \N$, $n\geq N_0$, and consider the tripod $G_{n+1}$ and the flap-plane $(S_{n+1},d_{n+1})$. We also choose $h_{n+1}$ to be so small that \eqref{Ahlfors-Upper-Geometric} holds; note that by \ref{G-Rectangles} the rectangles $E\sim G_i$, $i<n+1$, with metric $d_{n+1}$ are isometric to Euclidean rectangles, so we only have to choose $h_{n+1}$ to be small small enough so that \eqref{Ahlfors-Upper-Geometric} holds for the (last) index $i=n$. Later, we will make $h_{n+1}$ even smaller in order to achieve \eqref{Ahlfors-Upper-Measure of E}.

If $B_{n+1}(x,r)$ has a projection to the plane that intersects $G_{i_1},\dots,G_{i_{N_0}},G_{i_{N_0+1}}$, with $i_1<\dots<i_{N_0+1}\leq n+1$ then we split in two cases:

\textbf{Case 1:} $i_{N_0+1}\neq n+1$. If $P\colon S_{n+1}\to S_n$ denotes the natural projection, then in this case the projection of $B_{n}(P(x),r)$ to $\R^2$ also intersects $G_{i_1},\dots,G_{i_{N_0}},G_{i_{N_0+1}}$. This is essentially because the projection $P$ from $S_{n+1}$ to $S_n$ is $1$-Lipschitz by \ref{G-Lipschitz}, and also the projections are compatible by \ref{G-compatible}. Now, using the induction assumption we obtain
\begin{align*}
\mu_n\left( \bigcup_{E\sim G_{i_{N_0+1}}}E \right) \leq  r^2.
\end{align*}
The measure $\mu_n$, restricted to $E$, is identical to the Lebesgue measure, and also to $\mu_{n+1}$ by \ref{G-Rectangles}. This completes the proof of \eqref{Ahlfors-Upper-Measure of E} in this case.

\textbf{Case 2:} $i_{N_0+1}=n+1$. Then there exist points $a_{i_j} \in G_{i_j} \subset \R^2$ for $j=1,\dots,N_0$ and a point $a_{n+1}\in G_{n+1}$ such that 
\begin{align*}
|a_{i_j}-a_{n+1}| <2r
\end{align*}
for all $j=1,\dots,N_0$. This is because the projection of $B_{n+1}(x,r)$ to $\R^2$ intersects the corresponding tripods, by assumption, and also it is $1$-Lipschitz by \ref{G-Lipschitz}.

We now introduce auxiliary vertices on $G_{n+1}$ as follows. We partition each edge of $G_{n+1}$ in finitely many edges such that the interior of each (new) edge of $G_{n+1}$ does not  contain any vertex of $G_i$, $i\leq n$, and also each (new) edge of $G_{n+1}$ has one ``free" vertex that does not lie on any $G_i$, $i\leq n$; this is possible because the tripods $\{G_i\}_{i\in \N}$ possess property $(G)$ and in particular the intersection of $G_{n+1}$ with the union of the tripods $G_i$, $i\leq n$, contains finitely many points. For $i\leq n$ we set $\delta_i>0$ to be to be the minimum (Euclidean) distance of the (new) edges of $G_{n+1}$ from $G_i$, excluding the edges of $G_{n+1}$ that intersect $G_i$. We then set $\delta= \min_{1\leq i\leq n} \delta_i>0$. The partitioning of the edges of $G_{n+1}$ is only used to define $\delta$ in this proof, and is not supposed to alter the tripod $G_{n+1}$ for any other consideration.

If $|a_{i_j}-a_{n+1}| \geq \delta$ for some $j\in \{1,\dots, N_0\}$, then we have $\delta \leq 2r$. Hence, if we set $h_{n+1}$ to be so small (depending on $\delta$) that 
\begin{align*}
\mu_{n+1}\left( \bigcup_{E\sim G_{n+1}}E \right) \leq  \delta^2/4 \leq r^2,
\end{align*}
then we obtain the desired conclusion. Note that $\delta$, and thus $h_{n+1}$, is chosen depending only on $G_1,\dots,G_{n+1}$.

If $|a_{i_j}-a_{n+1}|<\delta$ for all $j\in \{1,\dots,N_0\}$, then by the definition of $\delta$ this means that $a_{n+1}$ necessarily lies on an edge $e$ of $G_{n+1}$ that intersects $G_{i_j}$, for all $j\in \{1,\dots,N_0\}$. However, the edge $e$ by construction has a ``free" vertex, so only one of the two vertices of the edge $e$ can intersect tripods $G_i$, $i\leq n$. Hence, there are $N_0+1$ tripods (including $G_{n+1}$) meeting at a vertex of the edge $e$. This implies that there are $N_0+1$ edges of the graph $T_{n+1}$, meeting at a vertex of the edge $e$. We now have a contradiction to the assumption that the degree of the planar graph $T_{n+1}$ is at most $N_0$.
\end{proof}

\subsection{Linear local connectivity}\label{Section LLC}

A metric space $(X,d)$ is \textit{linearly locally connected} (LLC) if there exists a constant $c\geq 1$ such that for each ball $B_d(x,r)$ and for any two points $z,w\in X$ we have:
\begin{enumerate}[\upshape(i)]
\item If $z,w\in B_d(x,r)$, then there exists a continuum $F\subset B_d(x,cr)$ connecting $z$ and $w$.
\item If $z,w\notin B_d(x,r)$, then there exists a continuum $F\subset X\setminus B_d(x,r/c)$, connecting $z$ and $w$. 
\end{enumerate}
In this case, we say that $X$ is $c$-LLC.

As we remarked earlier, the LLC property and the upper mass bound of Ahlfors regularity in the $2$-dimensional setting can yield the lower mass bound:

\begin{prop}\label{LLC implies 2-regular}
Let $(X,d)$ be a metric space homeomorphic to $\R^2$. If $(X,d)$ has locally finite Hausdorff $2$-measure and is $c$-LLC, then there exists a constant $c'>0$ depending only on $c$ such that
\begin{align*}
\mathcal H^2_d(B_d(x,r)) \geq c' r^2
\end{align*}
for all $x\in X$ and  $0<r<\diam(X)$.
\end{prop}

This is a folklore statement, and it is discussed in \cite[Section 16]{Ra}; see also \cite[Corollary 1.4]{Ki}. We give a quick proof for the sake of completeness.

\begin{proof}
Let $x\in X$ and $r<\diam(X)$. Consider the function $\phi(z)=d(x,z)$, which is $1$-Lipschitz. Let $y\notin B(x,r)$. Then for each $t\in (0,r)$ there exists a component $K_t$ of $\phi^{-1}(t)$ that separates $x$ from $y$ \cite[IV Theorem 26]{Mo}. We claim that $\mathcal H^1_d(K_t) \geq  C_1 r$ for all $t\in [r/4,r/2]$ and for a constant $C_1>0$ depending only on $c$. If this is the case, then by the co-area formula \cite[Proposition 3.1.5]{AT}, there exists a uniform constant $C_2>0$ such that we have
\begin{align*}
 \mathcal H^2_d(B_d(x,r)) \geq C_2\int_{r/4}^{r/2} \mathcal H^1_d( \phi^{-1}(t)) \,dt \geq \frac{C_2 C_1}{4}r^2,
\end{align*}
as desired. 

To prove our claim, note that for any $w\in K_t$, $t\in [r/4,r/2]$, we have $x,y\notin B_d(w,r/4)$. By condition (ii) in the definition of linear local connectivity, it follows that there exists a continuum $F \subset X \setminus B_d(w,r/(4c))$ that connects $x$ and $y$. However, any such continuum has to intersect $K_t$. It follows that $K_t$ cannot be contained in $B_d(w,r/(4c))$ for any $w\in K_t$, hence $\mathcal H^1_d(K_t)\geq \diam(K_t) \geq r/(4c)$. Our claim is proved with $C_1=(4c)^{-1}$.
\end{proof}

Note that each space $(S_n,d_n)$ satisfies  condition (i) with constant $1$, since the space is endowed with its internal metric the distance between two points $x,y\in S_n$ is equal to the length of the shortest path between the two points. 

\begin{prop}\label{LLC}
Assume that the degree of $T_n$ is bounded by $N_0$, for all $n\in \N$. If the heights $\{h_n\}_{n\in \N}$ are chosen to be sufficiently small, then the spaces $(S_n,d_n)$ are $c$-LLC, with the constant $c$ depending only on $N_0$.
\end{prop}

We remark, once again, that the choice of the heights is to be interpreted as follows: the height $h_n$ of the rectangles attached to $G_n$ has to be chosen to be sufficiently small, depending on $G_1,\dots,G_n$; see also the comments after Theorem \ref{General-QS Embedding}. 

We first prove a version of that proposition for the flap-plane corresponding single tripod $G$, with constants independent of the geometry of $G$. 

\begin{lemma}\label{LLC-One Graph}
Let $G\subset \R^2$ be a tripod and consider a flap-plane $S=S(G)$. If the height $h$ of the rectangles $E\sim G$  is less than the width of $E$, then $S$ is $c$-LLC for a universal constant $c> 1$.
\end{lemma}

In particular, the constant $c$ is independent of the lengths of the edges of $G$ and of their angles.

\begin{proof}
Note that we only have to prove condition (ii), i.e., for any ball $B_d(x,r)$ in $S$ and any two points $z,w \in S\setminus B_d(x,r)$ there exists a continuum $F\subset B_d(x,r/c)$ connecting $z$ and $w$, where $c>1$ is a universal constant to be determined.

In fact, it suffices to show that there exists a constant $c>1$ such that for each ball $B_d(x,r)$ and $z\notin B_d(x,r)$ there exists a polygonal path $\gamma_z \subset S\setminus B_d(x,r/c)$ that connects $z$ to a point $z'$, whose projection to $\R^2$ lies outside $B(\tilde x,r/c)$. Here $\tilde x$ is the projection of $x$ to $\R^2$. Indeed, if this is true, then the same statement holds for a point $w\notin B_d(x,r)$, and there exists a polygonal path $\gamma_w$ and a point $w'$ with the corresponding properties. One can then connect the projections $\tilde z'$ and $\tilde w'$ with a polygonal path $\tilde \gamma \subset \R^2\setminus B(\tilde x,r/c)$. By properties \ref{G-lift paths} and \ref{G-Lipschitz}, the path $\tilde \gamma$ lifts to a polygonal path $\gamma$  that connects $z'$ and $w'$ and lies outside $B(x,r/c)$. Then the concatenation of $\gamma_z,\gamma$, and $\gamma_w$ yields the desired path in the LLC (ii) condition.

Assume that $z\notin B_d(x,r)$. We denote by $P\colon S\to \R^2$ the projection of $S$ to the plane. Also, for a point $y\in S$ we use the notation $\tilde y=P(y)$. We now split the argument in two cases:

\textbf{Case 1:} $r\geq 12h$. Then the projected point $\tilde z$ does not lie in $B(\tilde x, r/2)$. Indeed, if $|\tilde x-\tilde z|<r/2$, then by the property \ref{G-Lipschitz Upper inequality} we have
\begin{align*}
d(x,z) \leq |\tilde x-\tilde z|+ 6h < r/2+ r/2=r,
\end{align*}
a contradiction. Hence, we can take $z'=z$, and $c\geq 2$.

\textbf{Case 2:} $r<12h$. We set $r_1=r/48 < h/4$. 

We connect $z$ to $x$ with a geodesic $\gamma\subset S$. If $\gamma$ (or rather its projection $\tilde \gamma$) does not intersect $G$, then by \ref{G-Geodesic} it projects isometrically to a geodesic from $\tilde z$ to $\tilde x$, which has to be a straight line segment. Then 
\begin{align*}
|\tilde x-\tilde z| =d(x,z) >r
\end{align*}
so $\tilde z\notin B(\tilde x, r)$, and we can set $z'=z$ and $c\geq 1$. 

If $\tilde \gamma$ does intersect $G$, we consider $y$ to be the first entry point of $\gamma$ into a rectangle $E\sim G$ as it travels from $z$ to $x$; we could have $y=z$ in case $z$ lies in a rectangle attached to $G$. In particular, the segment $[z,y]\subset \gamma_z$ does not intersect the rectangles attached to $G$, except possibly at the point $y$, and projects isometrically to a line segment $[\tilde z,\tilde y]$ in $\R^2$. If $y\in B_d(x,5r_1)$, then using the $1$-Lipschitz property \ref{G-Lipschitz} of the projection and also \ref{G-Geodesic} we have
\begin{align*}
|\tilde z -\tilde x| &\geq |\tilde z-\tilde y|- |\tilde y-\tilde x| \geq d(z,y) -d(y,x)\\
& \geq d(z,x)-2d(y,x)> r-10r_1>r_1,
\end{align*}
since $r= 48r_1> 11r_1$. It follows that $\tilde z$ lies outside $B(\tilde x,r_1)$. Hence, we may take $z'=z$ and $c\geq 48$.

Finally, we have to treat the case that $y\in E$ but $y\notin B_d(x,5r_1)$. We claim that $y$ can be connected to a point $y'\in E$ with a polygonal path $\gamma_y \subset E$ outside $B_d(x,r_1)$ such that $y'$ projects to a point $\tilde y'$  lying outside $B(\tilde x, r_1)$. In this case, note that we also have  $[z,y]\cap B_d(x,r_1)=\emptyset$, since $y \notin B_d(x,r_1)$ and $[z,y]\subset \gamma_z$, where $\gamma_z$ is a geodesic from $z$ to $x$. Then one can concatenate the path $[z,y]$ with $\gamma_y$ to obtain the desired polygonal path. Here, we have $z'=y'$ and $c\geq 48$.

We now prove our last claim. If $B_d(x,r_1)\cap E=\emptyset$, then we connect $y$ with a polygonal path in $E$ to a point $y'\in E$, whose projection lies outside $B(\tilde x, r_1)$. This can be done because $E$ projects onto a line segment of length equal to the width of $E$, and thus greater than the height $h$, by assumption. On the other hand, the ball $B(\tilde x,r_1)$ has diameter $2r_1 <h/2$. Next, assume that $B_d(x,r_1)\cap E\neq \emptyset$ and that the intersection contains a point $a\in E$. We have 
$$y \notin B_d(x,5r_1)\supset B_d(a,4r_1)\supset B_d(a,2r_1) \supset B_d(x,r_1).$$
The metric $d$ is isometric to the Euclidean metric when restricted to $E$ by \ref{G-Rectangles}, hence  $B_d(a,2r_1)\cap E$ is the intersection of a round ball with the rectangle $E$. Since $2r_1<h/2$, the ball $B_d(a,2r_1)$ cannot intersect both the top and bottom ``long" sides of $E$. This implies that the set $E\setminus B_d(a,2r_1)$ has at most two connected components, one of which contains  a ``long" side $\mathcal E$ of $E$, with length equal to the width of $E$, and thus greater than the height $h$.

If $E\setminus B_d(a,2r_1)$ has only one component then it is path connected. In this case, the point $y\in E\setminus B_d(a,2r_1)$  can be connected with a polygonal path $\gamma_y \subset E \setminus B_d(a,2r_1)$ to a point $y'\in E$, whose projection to the plane lies outside $B(\tilde x , r_1)$. Again, this is because the set $E\setminus B_d(a,2r_1)$ projects onto an interval in $\R^2$ whose length is larger than $h>4r_1 >\diam(B(\tilde x,r_1))$. 

The scenario in which $E\setminus B_d(a,2r_1)$ has two components can only occur if $B_d(a,2r_1)$ intersects two adjacent sides of the rectangle $E$. The point $y$ then has to lie in the ``large" component that also contains the side $\mathcal E$ of $E$. Indeed, the other component is contained in $B(a,4r_1)$, hence it cannot contain $y$. Now, as before, we can connect $y$ with a polygonal path outside $B_d(a,2r_1)\supset B_d(x,r_1)$ to a point $y'$ with the desired property.

Summarizing, we proved that $S$ is $c$-LLC with $c=48$, provided that $h$ is smaller than the length of the edges of $G$.
\end{proof}

\begin{remark}\label{LLC-Sufficient condition}
The proof of Lemma \ref{LLC-One Graph} shows that one may obtain the following stronger conclusion which implies that the space is $c$-LLC:

\noindent
There exists a constant $c>1$ such that for each ball $B_d(x,r)$ and for each $z\notin B_d(x,r)$ there exists a path $\gamma_z$ outside $B_d(x,r/c)$ that connects $z$ to a point $z'$, whose projection $\tilde z'$ to the plane lies outside $B(\tilde x,r/c)$. The path $\gamma_z$ can be taken to be polygonal.
\end{remark}

Next, we have a version of the previous lemma for $N\leq N_0$ tripods.

\begin{lemma}\label{LLC-N_0 Graphs}
Let $\{G_1,\dots,G_N\}$ be a family of tripods possessing property $(G)$ and suppose that $N\leq N_0$. Consider a flap-plane $S=S(G_1,\dots,G_N)$. If the height $h_i$ of each rectangle $E\sim G_i$  is less than the width of $E$ for all $i\in \{1,\dots,N\}$, then the flap-plane $S$ is $c_0$-LLC, with constant $c_0$ depending only on $N_0$. 
\end{lemma}

In fact, Remark \ref{LLC-Sufficient condition} also applies here. It is important in this lemma that the height $h_i$ can be chosen to depend only on $G_i$ and not on $G_j$ for $i\neq j$. In particular, we can set $h_i$ to be less than smallest among the lengths of the edges of $G_i$. Moreover, the dependence of $c_0$ on the number of tripods $N_0$ cannot be relaxed.

\begin{proof}
We give  the argument in case $N=2$ and  then sketch the almost identical induction argument required  to prove the statement for arbitrary $N\leq N_0$. Assume that we have two tripods $G_1$ and $G_2$, and let $h_1,h_2>0$ be smaller than the length of each edge of $G_1,G_2$, respectively. Without loss of generality, we assume that $h_1\leq h_2$. 

Consider a flap-plane $S= S(G_1,G_2)$ with metric $d$, and let $\Sigma= S(G_2)$ with metric $\sigma $ be the flap-plane that arises by collapsing (or projecting) in $S$ the rectangles $E\sim G_1$ to the plane. Also, consider the natural projection $P^*\colon S \to \Sigma$. For a point $x\in S$ we denote $x^*=P^*(x)$.

As remarked in the beginning of the proof of Lemma \ref{LLC-One Graph} and also in Remark \ref{LLC-Sufficient condition}, it suffices to show that for each ball $B_d(x,r)$ and $z\notin B_d(x,r)$ there exists a path $\gamma_z\subset S\setminus B_d(x,r/c_0)$ that connects $z$ to a point $z'$, whose projection to $\R^2$ lies outside $B(\tilde x,r/c_0)$. Here, $c_0>1$ is a constant that depends only on $N_0$ and, as usual, $\tilde x$ denotes the projection of $x$ to the plane.

As in the proof of Lemma \ref{LLC-One Graph} we split in two cases. If $r\geq 12 h_1$, then we have $z^* \in \Sigma \setminus B_\sigma(x^*,r/2)$. Indeed, otherwise, by \ref{G-Lipschitz Upper inequality} we would have
\begin{align*}
d(x,z)\leq \sigma(x^*,z^*) + 6 h_1 <r/2 + r/2=r,
\end{align*}
a contradiction. By Lemma \ref{LLC-One Graph} and Remark \ref{LLC-Sufficient condition} there exists a universal constant $c>1$ and  there exists a polygonal path $\gamma^*$ outside $B_\sigma (x^*,r/2c)$ that  connects $z^*$ to a point $(z^*)'$, whose projection to the plane lies outside $B(\tilde x, r/2c)$. Using \ref{G-Lipschitz} and \ref{G-lift paths}, we can lift $\gamma^*$ to a polygonal path $\gamma\subset S \setminus B_d(x,r/2c)$ that connects $z$ to a point $z'$, whose projection to the plane agrees with the projection of $(z^*)'$ (by compatibility \ref{G-compatible}), and therefore lies outside $B(\tilde x,r/2c)$. Hence, we may choose $c_0\geq 2c$.

If $r<12h_1$, then we also have $r< 12h_2$. Then the  same argument as in Case 2 of the proof of Lemma \ref{LLC-One Graph} can be used to obtain the conclusion and here we only need to choose $c_0 \geq 48$. Summarizing, one has to choose $c_0= \max\{ 2c, 48\}$.

Assume that the statement holds for any family of $N-1\leq N_0$ tripods $G_1,\dots,\\G_{N-1}$  satisfying the assumptions. Namely, there exists a constant $c$ depending only on $N_0$ such that any flap-plane $S=S(G_1,\dots,G_{N-1})$ with the correct choice of heights is $c$-LLC and satisfies the condition of Remark \ref{LLC-Sufficient condition}. We now consider $N\leq N_0$ tripods $G_1,\dots,G_N$ and a flap-plane $S= S(G_1,\dots,G_N)$ with metric $d$ as in the statement. By reordering the tripods, we may assume that $h_1\leq \dots\leq h_N$. Let $\Sigma$ with metric $\sigma$ be the flap-plane arising by collapsing all rectangles $E\sim G_1$ to the plane. Then $\Sigma=S(G_2,\dots,G_N)$ and it satisfies the condition of Remark \ref{LLC-Sufficient condition}, by the induction assumption. 

We consider $z\in S\setminus B_d(x,r)$ and split in two cases. If $r\geq 12h_1$, then $z^* \in \Sigma\setminus B_\sigma(x^*,r/2)$, where $z^*=P^*(z)$, and $P^*\colon S\to \Sigma$ is the projection. By the induction assumption it follows that $z^*$ can be connected with a path $\gamma^*\subset \Sigma \setminus B_\sigma(x^*,r/2c)$ to a point $(z^*)'$, whose projection to the plane lies outside $B(\tilde x,r/2c)$. Lifting the path $\gamma^*$ to $S$ yields the desired path. Hence, it suffices to choose the LLC constant to be $c_0\geq 2c$.

If $r<12h_1$ then in fact $r <12h_i$ for all $i\in \{1,\dots,N\}$. Thus, the argument in Case 2 of the proof of Lemma \ref{LLC-One Graph} can be used, and the LLC constant has to be $c_0\geq 48$. Summarizing, one has to choose $c_0=\max \{2c, 48\}$, which depends only on $N_0$, since by assumption $c$ depends only on $N_0$. In fact, one can choose $c_0=2^{N_0-1}48$.
\end{proof}

\begin{proof}[Proof of Proposition \ref{LLC}]
The metric of $S_n=S(G_1,\dots,G_n)$ is denoted by $d_n$ and the ball around $x$ of radius $r$ will be denoted by $B_n(x,r)$.

We argue by induction on $n$. We claim that for each $n\in \N$ there exists a constant $C_n>1$, increasing and uniformly bounded in $n$ such that if the height $h_n$ is chosen to be small enough depending on $G_1,\dots,G_n$ and also smaller than the width of the rectangles $E\sim G_n$, then the following holds:
\newline
\newline
\noindent
whenever $z\notin B_n (x,r)$ for some $x\in S_n$ and $r>0$, we can connect $z$ to a point $z'$ with a  polygonal path $\gamma \subset S_n \setminus B_n(x,r/C_n)$ such that the projection $\tilde z'$ of $z'$ to $\R^2$ lies outside $B(\tilde x,r/C_n)$.
\newline
\newline
This suffices by Remark \ref{LLC-Sufficient condition}, and shows that $S_n$ is $C_n$-LLC. Since $C_n$ is bounded above by a constant $C$, it follows that $S_n$ is $C$-LLC for all $n\in \N$, which is the desired conclusion.

Now, we focus on proving our claim. For $n=1$ the statement holds with the constant $C_1$ given by Lemma \ref{LLC-N_0 Graphs}, provided that $h_1$ is sufficiently small depending on $G_1$. We assume that the claim holds for $S_1,\dots,S_n$, so, in particular, the height $h_i$ of each rectangle $E\sim G_i$ has been chosen to be less than the width of $E$ for $i\in \{1,\dots,n\}$. Our goal is to choose the height $h_{n+1}$ of the rectangles $E\sim G_{n+1}$ so that our claim holds. To begin with, we choose $h_{n+1}$ to be smaller than the width of all rectangles $E\sim G_{n+1}$, and later we will choose it to be even smaller. Consider a ball $B_{n+1}(x,r)\subset S_{n+1}$ and $z\notin B_{n+1}(x,r)$. We split into two main cases:

\textbf{Case 1:} The projection of $B_{n+1}(x,r)$ to the plane does not intersect both $G_{n+1}$ and $\bigcup_{i<n+1}G_i$. 

Assume first that the projection of $B_{n+1}(x,r)$ to the plane intersects only $G_{n+1}$. We denote by $\Sigma = S(G_{n+1} )$ the flap-plane that arises by collapsing all rectangles $E\sim G_i$, $i<n+1$, to the plane. Also, we denote by $P^*\colon S(G_1,\dots,G_{n+1})\to \Sigma$ the natural projection, and by $\sigma$ the metric of $\Sigma$. An application of \ref{G-Ball} yields  $(P^*)^{-1}(B_\sigma(x^*,r))= B_{n+1}(x,r)$, where $x^*=P^*(x)$. We now have that $z^*=P^*(z) \notin B_{\sigma}( x^*,r)$, hence, by Lemma \ref{LLC-N_0 Graphs} and Remark \ref{LLC-Sufficient condition}, there exists a polygonal path $\gamma^* \subset \Sigma \setminus B_{\sigma}( x^*, r/C_1)$ that connects $z^*$ to a point $(z^*)'$, whose projection to $\R^2$ lies outside $B (\tilde x,r/C_1)$. Using \ref{G-lift paths} and the $1$-Lipschitz property \ref{G-Lipschitz}, we lift the path $\gamma^*$ under $P^*$ to a polygonal path $\gamma\subset S_{n+1}\setminus B_{n+1}(x,r/C_1)$ that connects $z$ to a point $z'$, whose projection to $\R^2$ is the same as the projection of $(z^*)'$ to $\R^2$ (by compatibility \ref{G-compatible}), so it lies outside $B(\tilde x,r/C_1)$. The constant $C_{n+1}$ will be chosen later so that $C_{n+1}\geq C_n\geq C_1$. Hence, $\gamma$ lies outside $B_{n+1}(x,r/C_{n+1})$ and the projection of $z'$ to the plane lies outside $B (\tilde x,r/C_{n+1})$, as desired. 

Next, assume that the projection of $B_{n+1}(x,r)$ to the plane intersects only $\bigcup_{i<n+1} G_i$. We denote here by $P^*$ the projection of $S_{n+1}$ to $S_n$. Then  by \ref{G-Ball} we have $(P^*)^{-1}(B_n(x^*,r)) = B_{n+1}(x,r)$. Since $z^*\notin B_n(x^*,r)$, by the induction assumption there exists a polygonal path $\gamma^* \subset  S_n \setminus B_n(x^*,r/C_n)$ that connects $z^*$ to a point $(z^*)'$, whose projection to $\R^2$ lies outside $B(\tilde x, r/C_n)$. Lifting this path and noting that $C_{n+1}\geq C_n$, as before, yields again the desired path $\gamma$ and point $z'$.

\textbf{Case 2:} The projection of $B_{n+1}(x,r)$ to the plane intersects both $G_{n+1}$ and $\bigcup_{i<n+1} G_i$. Let $J \subset \{1,\dots,n\}$ be the set of indices $j$ such that the projection of $B_{n+1}(x,r)$ to the plane intersects $G_j$. 

Assume first that $\#J\leq N_0-1$ (here $N_0$ is by assumption the bound on the degree of the graph $T_{n+1}$), and let $\Sigma = S(\{ G_{n+1} \}\cup\{G_j:j\in J\})$ with metric $\sigma$ be the flap-plane arising by collapsing the rectangles $E\sim G_{i}$, $i\notin J$, $i\neq n+1$, to the plane.  We note that we can apply Lemma \ref{LLC-N_0 Graphs} to $\Sigma$, since the heights of the rectangles attached to the corresponding tripods are smaller than the widths, by the induction assumption and the choice we have made for $h_{n+1}$. Since $B_{n+1}(x,r)$ (or rather its projection to $\R^2$) intersects only tripods that are also ``present" in $\Sigma$, we can use as before property \ref{G-Ball} and path-lifting to reduce the statement to $\Sigma$. By Lemma \ref{LLC-N_0 Graphs} and Remark \ref{LLC-Sufficient condition}, the conclusion holds in $\Sigma $ with the constant $C_1$ given by Lemma \ref{LLC-N_0 Graphs}.

Assume now that $\#J \geq N_0$. Then there exist points $a_j\in G_j$, $j\in J$, and a point $a_{n+1}\in G_{n+1}$ such that 
\begin{align*}
|a_j-a_{n+1}|< 2r
\end{align*}
for all $j\in J$. This follows from the $1$-Lipschitz property \ref{G-Lipschitz} of the projections.

As in the proof of Lemma \ref{Ahlfors-Upper}, we introduce auxiliary vertices on $G_{n+1}$ as follows.  We partition each edge of $G_{n+1}$ in finitely many edges such that the interior of each (new) edge of $G_{n+1}$ does not  contain any vertex of $G_i$, $i\leq n$, and also each (new) edge of $G_{n+1}$ has one ``free" vertex that does not lie on any $G_i$, $i\leq n$. For $i\leq n$ we set $\delta_i>0$ to be to be the minimum distance of the (new) edges of $G_{n+1}$ from $G_i$, excluding the edges of $G_{n+1}$ that intersect $G_i$. We then set $\delta= \min_{1\leq i\leq n} \delta_i$. The partitioning of the edges of $G_{n+1}$ is only used to define $\delta$ in this proof, and is not considered to alter the tripod $G_{n+1}$ for any other consideration. Note that $\delta$ depends only on $G_1,\dots,G_{n+1}$.

As in the proof of Lemma \ref{Ahlfors-Upper}, there has to exist some $j\in J$ such that $|a_j-a_{n+1}|>\delta$, since the degree of $T_{n+1}$ is at most $N_0$. Hence, $\delta<2r$. Now, for any number $\alpha\in (0,1)$, we can choose the height $h_{n+1}$ to be so small, depending only on $N_0$, $\delta$, and $\alpha$, that 
\begin{align}\label{LLC-B contains PB}
B_{n+1}(x,r) \supset (P^*)^{-1}(B_n( x^*, \alpha r)),
\end{align}
where $P^*$ denotes the projection from $S_{n+1}$ onto $S_n$. Indeed, for any $y^* \in B_n(x^*,\alpha r)$ and any preimage $y\in (P^*)^{-1}(y^*)$ we have by \ref{G-Lipschitz Upper inequality}
\begin{align*}
d_{n+1}(x,y) \leq d_n(x^*,y^*) + 6h_{n+1} <\alpha r + (1-\alpha)\delta/2< r,
\end{align*} 
provided that $h_{n+1}< (1-\alpha)\delta/12$. \eqref{LLC-B contains PB} implies that $z^*\notin B_n(x^*,\alpha r)$, so by the induction assumption there exists a path $\gamma^* \subset S_n\setminus B_n(x^*, \alpha r/C_n)$ that connects $z^*$ to a point $(z^*)'\notin B_n(x^*,\alpha r/C_{n})$, whose projection to $\R^2$ lies outside $B(\tilde x,\alpha r/C_n)$. The path $\gamma^*$ lifts to a path $\gamma \subset B_{n+1}(x,\alpha r/C_n)$, so our claim holds with $C_{n+1}= C_n/\alpha$. Now, we choose $\alpha=1-1/(n+1)^2$, so 
\begin{align*}
C_{n+1}= C_n \left(1-\frac{1}{(n+1)^2} \right)^{-1} >C_n\geq C_1.
\end{align*}

With this choice we have
\begin{align*}
C_{n+1} \leq \frac{C_1}{\prod_{i=1}^\infty (1-1/(i+1)^2)} \eqqcolon C <\infty
\end{align*}
for all $n\in \N$. Note that $C$ depends only on $C_1$, and thus only on $N_0$.
\end{proof}

\subsection{Proof of Theorem \ref{General-QS Embedding}}\label{Section Embedding}
We will use Theorem \ref{General-Wildrick}. We note first that the assumptions of the theorem are satisfied by the spaces $(S_n,d_n)$ with uniform constants. Indeed, each of the flap-planes $(S_n,d_n)$ is unbounded since the projection onto the plane is 1-Lipschitz by \ref{G-Lipschitz}. Also, $(S_n,d_n)$ is complete since it is obtained by attaching finitely many rectangles to the plane; cf. proof of Proposition \ref{General-Completeness}. Furthermore, if the heights $h_1,\dots,h_n$ are chosen (inductively) to be sufficiently small, then by Proposition \ref{Ahlfors} and Proposition \ref{LLC} we conclude that $(S_n,d_n)$ is Ahlfors $2$-regular and LLC with constants independent of $n$. We also choose the heights to be even smaller, if necessary, so that the conclusions of Proposition \ref{General-Completeness} and Proposition \ref{General-Gromov Hausdorff} hold.

Theorem \ref{General-Wildrick} now yields for each $n\in \N$ a quasisymmetry $f_n$ from $(S_n,d_n)$ onto $\R^2$. Since the statement of the theorem is quantitative, we may assume that the distortion function $\eta$ of $f_n$ is independent of $n$. We would like to pass to a limiting quasisymmetry $f\colon S_\infty \to \R^2$. This will be obtained by applying Lemma \ref{Pre Mapping package}, after normalizing the functions $f_n$.

Consider the limiting space $(S_\infty,d_\infty)$, given by Proposition \ref{General-Completeness}. By Proposition \ref{General-Gromov Hausdorff}, for a fixed point $p\in S_\infty$ we may choose points $p_n\in S_n$ such that the sequence $(S_n,d_n,p_n)$ converges to the  space $(S_\infty,d_\infty,p)$ in the pointed Gromov-Hausdorff sense of Definition \ref{Pre Definition pointed GH}. 

Since all of the spaces $S_n$ are Ahlfors $2$-regular with uniform constants, it follows that they are uniformly doubling; see comments after Definition \ref{Pre Definition pointed GH}. For each $n\in \N$ we consider a point $x_n\in S_n$ such that $d_n(p_n,x_n)=1$; recall that the space $S_n$ is a length space. By postcomposing $f_n$ with a M\"obius transformation of $\R^2$, we may obtain a sequence $g_n \colon S_n\to \R^2$ such that $g_n(p_n)=0$ and $g_n(x_n)=1$ for all $n\in \N$. The functions $g_n$ will still be $\eta$-quasisymmetric, since the distortion function is not affected under compositions with scalings and translations. Lemma \ref{Pre Mapping package} (with $Y_n\equiv \R^2$) now yields a subsequence of $g_n$ that converges to an $\eta$-quasisymmetry $g\colon S_\infty \to \R^2$. By Lemma \ref{Pre Convergence Ahlfors regular} it also follows that $S_\infty$ is Ahlfors $2$-regular. \qed

\section{The continuous case}\label{Section Continuous}
In this section we prove first the non-removability of the gasket for continuous $W^{1,2}$ functions (Theorem \ref{Intro gasket non-removable}) and then the non-removability of homeomorphic copies of the gasket (Theorem \ref{Intro homeo gasket non-removable}). Also, we include proofs of the general statements in Theorem \ref{Intro Theorem Positive} and Proposition \ref{Intro Theorem W^{1,infinity}}, regarding the (non)-removability of sets of positive measure.

\subsection{Terminology and geometry of the gasket}\label{Section Definitions}
We first recall the definition of the Sierpi\'nski gasket, introduce some terminology, and discuss its combinatorial properties.

The Sierpi\'nski gasket is constructed as follows. We consider an equilateral triangle of sidelength $1$ and subdivide it into four equilateral triangles of sidelength $1/2$. After removing the middle triangle, we proceed inductively with subdividing each of the remaining three triangles into four equilateral triangles of sidelength $1/2^2$, and so on. The remaining compact set $K$ is the Sierpi\'nski gasket; see Figure \ref{fig:gasket}. From the definition it is immediate that $K$ has area zero. Indeed, at the $n$-th step of the construction $K$ is contained in the union of $3^n$ equilateral triangles of sidelength $1/2^n$, hence
\begin{align*}
m_2(K)\leq 3^n \cdot \frac{\sqrt{3}}{4} \frac{1}{4^n},
\end{align*}
which converges to $0$ as $n\to\infty$. We will assume in what follows that $K\subset B(0,2)\subset \R^2$.

We call \textit{$w$-triangles} the complementary triangles of $K$ that are removed in each step. Making abuse of terminology we also call the unbounded component of $\R^2\setminus K$ a $w$-triangle of sidelength $1$. In the construction of $K$, at each step we remove a \textit{central} $w$-triangle $W_0$ from an equilateral triangle $V_0$ having double the sidelength of $W_0$, after subdividing $V_0$ into four equilateral triangles. We call \textit{$v$-triangles} the triangles arising as $V_0$. $w$-triangles and $v$-triangles are by definition open sets. Hence, using the previous notation $V_0\setminus \br W_0$ is the union of three $v$-triangles. We say that the \textit{level} of a $w$-triangle $W_0$ is equal to $n$ if the sidelength of $W_0$ is equal to $2^{-n}$. In particular, the unbounded $w$-triangle has level $0$, and the central $w$-triangle of the first step of the construction  has level $1$. For $n\geq 1$ there exist $3^{n-1}$ $w$-triangles of level $n$. Similarly, we say that the level of a $v$-triangle $V_0$ is equal to $n$ if its sidelength is $2^{-n}$. Note that there exists one $v$-triangle of level $0$ and $3^n$ $v$-triangles of level $n$, for each $n\in \N$.

We denote by $\mathcal W$ be the collection of $w$-triangles, and 
\begin{align*}
W_\infty \coloneqq \bigcup_{W\in \mathcal W} \br W.
\end{align*}
Also, we use the notation $K^\circ$ for the points of $K$ that do not lie on the boundary of any $w$-triangle, so in particular we have
\begin{align*}
K^\circ =K\setminus W_\infty.
\end{align*}

In the proofs, if $z$ is a point of the gasket, we will often have to distinguish between three cases, depending on whether $z$ is a vertex of a $w$-triangle, or a point on an edge of a $w$-triangle but not a vertex, or none of the above, i.e., $z\in K^\circ$. In the first case that a point $z\in K$ is a vertex of a $w$-triangle, we say that $z$ is \textit{of vertex type}. In the second case that $z$ lies on the boundary of a $w$-triangle but it is not a vertex, we say that it is \textit{of edge type}.

Two $w$-triangles $W_1,W_2$ are \textit{adjacent} if a vertex of $W_1$ lies on $\partial W_2$, or vice versa. Note that if $W_1$ has a vertex on $\partial W_2$ then the level of $W_2$ is strictly smaller than that of $W_1$, i.e., $W_2$ is a strictly larger triangle than $W_1$.

We now study some important properties of the combinatorics of the gasket. For each point $z\in K$ there exists a sequence $\{V_n\}_{n\in \N}$ of nested $v$-triangles with 
\begin{align*}
\{z\}= \bigcap_{n=1}^\infty \br V_n.
\end{align*}
In fact, this sequence is unique if the following hold:
\begin{enumerate}[(i)]
\item $z$ is not a vertex of a $w$-triangle, or it is a vertex of the unbounded $w$-triangle of level $0$,
\item $V_1$ has level $0$ (so it is the very first triangle in the construction of $K$), and $V_n$ has level $n-1$ for $n\in \N$,
\item $V_{n+1}\subset V_{n}$ for $n\in \N$.
\end{enumerate}
If $z$ is a vertex of a $w$-triangle of level at least $1$ then there are precisely two distinct sequences shrinking to $z$ and satisfying (ii) and (iii). 

The following two lemmas describe how the sequence $\br V_n$ shrinks to the point $z$. In fact, the first lemma refers to $v$-triangles and the second lemma to $w$-triangles. We could have incorporated both lemmas in one, but this would complicate the statements, so we state them separately. 

\begin{lemma}\label{Combinatorics V lemma}
Let $\{V_n\}_{n\geq 1}$ be a nested sequence of $v$-triangles satisfying $\mathrm{(ii)}$ and $\mathrm{(iii)}$, and converging to a point $z\in K$, in the sense that
\begin{align*}
\{z\}= \bigcap_{n=1}^\infty \br V_n.
\end{align*}
In case $z$ is a vertex of a $w$-triangle of level at least $1$ we also consider the other sequence $\{ V_n'\}_{n\in \N}$ that is distinct from $\{ V_n\}_{n\in \N}$ and converges to $z$. If $z$ is a vertex of the unbounded $w$-triangle of level $0$ we set $V_n'=V_n$ for $n\in \N$.
\begin{enumerate}[\upshape(I)] 
\item If $z$ is of vertex type, then there exist  two (possibly non-distinct) $w$-triangles $A$ and $B$ with $z\in \partial A\cap \partial B$ such that for each $n\in \N$ the set $\br A\cup \br B \cup \br V_n\cup \br V_n'$ contains all sufficiently small open neighborhoods of $z$. 

\item If $z$ is of edge type, then $z\in \partial V_n$ for all $n\in \N$ and moreover, there exists a $w$-triangle $B$ with $z\in \partial B$ such that for each $n\in \N$ the set $\br B\cup \br V_n$ contains all sufficiently small open neighborhoods of $z$. 

\item If $z\in K^\circ$ then $z\notin \partial V_n$ for all $n\in \N$, so for each $n\in \N$ the set $V_n$ contains all sufficiently small open neighborhoods of $z$. 
\end{enumerate}
\end{lemma}

\begin{figure}
\centering
\begin{tikzpicture}[line cap=round,line join=round,>=triangle 45,x=1.0cm,y=1.0cm,scale=.8]
\definecolor{zzttqq}{rgb}{0.26666666666666666,0.26666666666666666,0.26666666666666666}
\definecolor{qqqqff}{rgb}{0.3333333333333333,0.3333333333333333,0.3333333333333333}
\definecolor{xdxdff}{rgb}{0.6588235294117647,0.6588235294117647,0.6588235294117647}
\definecolor{cqcqcq}{rgb}{0.7529411764705882,0.7529411764705882,0.7529411764705882}

%\draw [color=cqcqcq,, xstep=0.5cm,ystep=0.5cm] (0.4395260322177297,-0.031267319862497915) grid (14.879166920830889,7.458621220864529);
%\draw[->,color=black] (0.4395260322177297,0.) -- (14.879166920830889,0.);
%\foreach \x in {,0.5,1.,1.5,2.,2.5,3.,3.5,4.,4.5,5.,5.5,6.,6.5,7.,7.5,8.,8.5,9.,9.5,10.,10.5,11.,11.5,12.,12.5,13.,13.5,14.,14.5}
%\draw[shift={(\x,0)},color=black] (0pt,2pt) -- (0pt,-2pt) node[below] {\footnotesize $\x$};
\clip(0.4395260322177296,-0.031267319862497915) rectangle (14.879166920830889,7.458621220864529);

\fill[line width=1.2pt,color=zzttqq,fill=zzttqq,fill opacity=0.10000000149011612] (8.,8.) -- (0.,8.) -- (4.,1.0717967697244903) -- cycle;
\fill[line width=1.2pt,color=zzttqq,fill=zzttqq,fill opacity=0.10000000149011612] (30.048858357398256,1.0717967697244903) -- (-10.07808728508504,1.0717967697244903) -- (9.985385536156599,-33.67915753294333) -- cycle;
\fill[line width=1.2pt,color=zzttqq,fill=zzttqq,fill opacity=0.10000000149011612] (6.74976672779656,5.8345324512305545) -- (9.49953345559312,1.0717967697244903) -- (12.249300183389682,5.83453245123055) -- cycle;
\fill[line width=1.2pt,color=zzttqq,fill=zzttqq,fill opacity=0.10000000149011612] (8.12465009169484,3.4531646104775224) -- (5.374883363898279,3.4531646104775224) -- (6.749766727796558,1.0717967697244917) -- cycle;
\fill[line width=1.2pt,color=zzttqq,fill=zzttqq,fill opacity=0.10000000149011612] (8.740088500143441,11.912665381127349) -- (14.99906691118624,1.0717967697244903) -- (21.25804532222905,11.912665381127342) -- cycle;

\draw [line width=1.2pt,color=zzttqq] (8.,8.)-- (0.,8.);
\draw [line width=1.2pt,color=zzttqq] (0.,8.)-- (4.,1.0717967697244903);
\draw [line width=1.2pt,color=zzttqq] (4.,1.0717967697244903)-- (8.,8.);
\draw [line width=1.2pt,color=zzttqq] (30.048858357398256,1.0717967697244903)-- (-10.07808728508504,1.0717967697244903);
\draw [line width=1.2pt,color=zzttqq] (-10.07808728508504,1.0717967697244903)-- (9.985385536156599,-33.67915753294333);
\draw [line width=1.2pt,color=zzttqq] (9.985385536156599,-33.67915753294333)-- (30.048858357398256,1.0717967697244903);
\draw [line width=1.2pt,color=zzttqq] (4.,1.0717967697244903)-- (9.49953345559312,1.0717967697244903);
\draw [line width=1.2pt,color=zzttqq] (6.74976672779656,5.8345324512305545)-- (9.49953345559312,1.0717967697244903);
\draw [line width=1.2pt,color=zzttqq] (9.49953345559312,1.0717967697244903)-- (12.249300183389682,5.83453245123055);
\draw [line width=1.2pt,color=zzttqq] (12.249300183389682,5.83453245123055)-- (6.74976672779656,5.8345324512305545);
\draw [line width=1.2pt,color=zzttqq] (8.12465009169484,3.4531646104775224)-- (5.374883363898279,3.4531646104775224);
\draw [line width=1.2pt,color=zzttqq] (5.374883363898279,3.4531646104775224)-- (6.749766727796558,1.0717967697244917);
\draw [line width=1.2pt,color=zzttqq] (6.749766727796558,1.0717967697244917)-- (8.12465009169484,3.4531646104775224);
\draw [line width=1.2pt,color=zzttqq] (8.740088500143441,11.912665381127349)-- (14.99906691118624,1.0717967697244903);
\draw [line width=1.2pt,color=zzttqq] (14.99906691118624,1.0717967697244903)-- (21.25804532222905,11.912665381127342);
\draw [line width=1.2pt,color=zzttqq] (21.25804532222905,11.912665381127342)-- (8.740088500143441,11.912665381127349);

%\begin{scriptsize}
\draw [fill=black] (0.,8.) circle (1.5pt);
\draw[color=black] (0.4665328418597742,7.2740746883105585) node {};
\draw [fill=black] (8.,8.) circle (1.5pt);
\draw[color=black] (7.848394144018622,7.373099656998055) node {};
\draw[color=black] (3.9954226350869306,5.7706956182367435) node {};

\draw [fill=black] (4.,1.0717967697244903) circle (1.5pt);
\draw[color=black] (4.040433984490338,1.2425538682539377) node [label={[label distance=0.05cm]-90:$x_{k(n),l(n)}$}]{};

\draw [fill=black] (4.,8.) circle (1.5pt);
\draw[color=black] (3.950411285683523,7.373099656998055) node {};
\draw [fill=black] (30.048858357398256,1.0717967697244903) circle (1.5pt);
\draw[color=black] (0.4395260322177297,7.5351405148503225) node {};
\draw [fill=black] (-10.07808728508504,1.0717967697244903) circle (1.5pt);
\draw[color=black] (0.6285736997120416,3.538132687827726) node {};
\draw[color=black] (9.981932105740142,-10.433390166989998) node {};
\draw [fill=black] (9.985385536156599,-33.67915753294333) circle (1.5pt);
\draw[color=black] (0.4395260322177297,7.5351405148503225) node {};

\draw [fill=black] (9.49953345559312,1.0717967697244903) circle (1.5pt);
\draw[color=black] (9.531818611706067,1.2425538682539377) node [label={[label distance=0.05cm]-90:$x_{n-1,l(n)}$}]{};

\draw [fill=black] (6.74976672779656,5.8345324512305545) circle (1.5pt);
\draw[color=black] (6.786126298098202,6.004754635134463) node [label={[label distance=-.35cm]120:$x_{n-1,k(n)}$}]{};

\draw[color=black] (9.49580953218334,4.330332437327701) node {};
\draw [fill=black] (12.249300183389682,5.83453245123055) circle (1.5pt);
\draw[color=black] (12.286513195194612,6.004754635134463) node {};
\draw [fill=black] (8.12465009169484,3.4531646104775224) circle (1.5pt);
\draw[color=black] (8.163473589842475,3.583144037231133) node {};

\draw [fill=black] (5.374883363898279,3.4531646104775224) circle (1.5pt);
\draw[color=black] (5.408779006353929,3.583144037231133) node [label={[label distance=-.35cm]120:$x_{n,k(n)}$}]{};

\draw[color=black] (6.741114948694794,2.7369306684470702) node {};

\draw [fill=black] (6.749766727796558,1.0717967697244917) circle (1.5pt);
\draw[color=black] (6.786126298098202,1.2065447887312115) node [label={[label distance=0.05cm]-90:$x_{n,l(n)}$}]{};

\draw [fill=black] (14.99906691118624,1.0717967697244903) circle (1.5pt);
\draw[color=black] (14.735130602739986,0.8104449139812244) node {};
\draw [fill=black] (8.740088500143441,11.912665381127349) circle (1.5pt);
\draw[color=black] (9.765877628603786,7.373099656998055) node {};
\draw[color=black] (14.996196429279749,8.372351613753704) node {};
\draw [fill=black] (21.25804532222905,11.912665381127342) circle (1.5pt);
\draw[color=black] (14.16798760025705,7.373099656998055) node {};

\draw (9.5, 4) node {$W_{n-1}$};
\draw (6.75, 2.5) node {$W_n$};
\draw (4,5.5) node {$A_n=W_{k(n)}$};
\draw (12,0.3) node {$W_{l(n)}=B_n$};
%\end{scriptsize}

\end{tikzpicture}
\caption{A typical situation as described in Lemma \ref{Combinatorics lemma}(III).}\label{fig:zoom}
\end{figure}
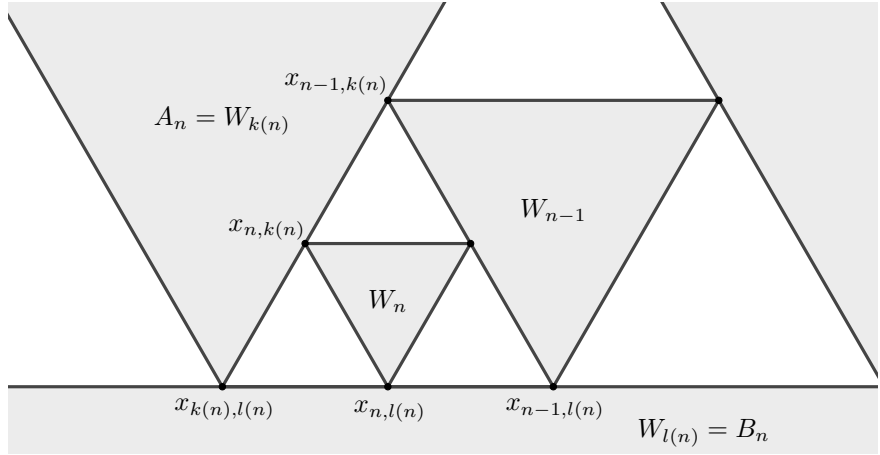

The following lemma describes essentially Figure \ref{fig:zoom}.

\begin{lemma}\label{Combinatorics lemma}
Let $\{V_n\}_{n\geq 1}$ be a nested sequence of $v$-triangles satisfying $\mathrm{(ii)}$ and $\mathrm{(iii)}$, and converging to a point $z\in K$, in the sense that
\begin{align*}
\{z\}= \bigcap_{n=1}^\infty \br V_n.
\end{align*}
Also, for each $n\in \N$ consider the central $w$-triangle $W_n\subset V_n$ of level $n$. Then, for $n\geq 2$, $W_n$ has one vertex on $\partial W_{n-1}$, and two vertices on the boundaries of some $w$-triangles $A_n,B_n$ (we could have $A_n=B_n$ if they are the unbounded $w$-triangle of level $0$). Assume that the level of $B_n$ is at most the level of $A_n$ (so $B_n$ is a larger triangle than $A_n$). Furthermore:
\begin{enumerate}[\upshape(I)]
\item If $z$ is a vertex of a $w$-triangle $A$ (i.e., $z$ is of vertex type), then there exists another $w$-triangle $B$ with $z\in \partial A\cap \partial B$ such that for all sufficiently large $n\in \N$ we have $A_n=A$ and $B_n=B$. In this case, $\partial V_{n}$ is contained in $\partial W_{n-1} \cup \partial A \cup \partial B$ and contains  $z$, $\partial W_{n-1}\cap \partial A$, $\partial W_{n-1}\cap \partial B$, and  also the vertices of $W_n\subset V_n$.

\item If $z$ is of edge type, then there exists a $w$-triangle $B$ such that $z\in \partial B$ and $B_n=B$ for all sufficiently large $n\in \N$, but no other $w$-triangle $A$ has the property that $A_n=A$ infinitely often. Moreover, for all sufficiently large $n\in \N$, $W_{n-1}$ has a vertex on $\partial B$ and a vertex on $\partial A_n$, and $A_n$ has a vertex on $\partial B$. In fact, there exists a sequence $\{k(n)\}_{n\in \N}$ with $k(n)\to\infty$ as $n\to\infty$ such that $A_n=W_{k(n)}$ for all sufficiently large $n\in \N$. In this case, $\partial V_{n}$ is contained in $\partial W_{n-1}\cup \partial W_{k(n)}\cup \partial B$ and contains the vertices $\partial W_{n-1}\cap \partial B$, $\partial W_{k(n)}\cap \partial B$, $\partial W_{n-1} \cap \partial W_{k(n)}$, and also the vertices of $W_n$.  Finally, the vertices $x_{n-1,k(n)}=\partial W_{n-1}\cap \partial W_{k(n)}$, $x_{n,k(n)}= \partial W_{n}\cap \partial W_{k(n)}$, and $x_{k(n),B}=\partial W_{k(n)}\cap \partial B$ are contained in a half-edge of $\partial W_{k(n)}$, and $x_{n,k(n)}$ lies between the two other vertices.

\item If $z\in K^\circ$, then no $w$-triangle $W$ has the property that $A_n=W$ or $B_n=W$ infinitely often. Moreover, for all sufficiently large $n\in \N$, $W_{n-1}$ has a vertex on $\partial A_n$ and a vertex on $\partial B_n$, and $A_n$ has a vertex on $\partial B_n$. In fact, there exist sequences $\{k(n)\}_{n\in \N}$, $\{l(n)\}_{n\in \N}$ that diverge to $\infty$ such that $A_n=W_{k(n)}$ and $B_n=W_{l(n)}$ for all sufficiently large $n\in \N$. In the latter case, note that $l(n)< k(n)< n-1<n$, and also $\partial V_{n} \subset \partial W_{n-1}\cup \partial W_{k(n)}\cup \partial W_{l(n)}$. Finally, the vertices $x_{n-1,k(n)}=\partial W_{n-1}\cap \partial W_{k(n)}$, $x_{n,k(n)}= \partial W_{n}\cap \partial W_{k(n)}$, and $x_{k(n),l(n)}=\partial W_{k(n)}\cap \partial W_{l(n)}$ are contained in a half-edge of $\partial W_{k(n)}$, and $x_{n,k(n)}$ lies between the two other vertices. The same statement holds with the roles of $k(n)$ and $l(n)$ reversed; see Figure \ref{fig:zoom}.
\end{enumerate}
\end{lemma}

The proofs of both lemmas are elementary and can be done by induction, so we leave them to the reader. Especially the second lemma will be crucially used in the proof of continuity of $f$ in the next theorem, which is a restatement of Theorem \ref{Intro gasket non-removable}.

\begin{theorem}[Theorem \ref{Intro gasket non-removable}]\label{Theorem Non-removability}
There exists a continuous function $f\colon \R^2\to \R$ with $f\in W^{1,2}(\R^2\setminus K)$, but $f\notin W^{1,2}(\R^2)$. In particular, $K$ is non-removable for $W^{1,2}$.
\end{theorem}

The function $f$ will be almost a constant on each $w$-triangle, and will rapidly change near the vertices. The (almost constant) value of $f$ on each $w$-triangle will be the average of the values on neighboring triangles of the previous level, with the exception of the central $w$-triangle of level $1$; see Figure \ref{fig:graph}.

The construction will be done in several steps. We will give an inductive construction of the function $f$, and ensure that it has \textit{finite energy}, i.e., $\nabla f \in L^2(\R^2)$. In fact, we will show that $\|\nabla f\|_{L^2(\R^2)}$ can be made arbitrarily small, and this will prevent $f$ from lying in $W^{1,2}(\R^2)$. Finally, we will focus on proving that our function $f$ with the inductive definition is continuous on $\R^2$. The proof of the latter property is very delicate and occupies most of the section.

\subsection{Building block}\label{Section Building block}
Here we describe the \textit{building block functions} that will be used to define $f$ on each $w$-triangle. 

Let $W \subset \R^2$ be an open equilateral triangle with vertices $x_1,x_2,x_3$. Then for each $\varepsilon>0$ and each choice of real numbers $a,c_1,c_2,c_3$ there exists a continuous function $g\colon \br W\to \R $ with $g\in W^{1,2}(W)$ and balls $B(x_i,r_i)$, $i=1,2,3$, such that 
\begin{enumerate}[($B$\upshape1)]
\item\label{B1 - g=a outside balls} $g\equiv a$ on $\br W\setminus \bigcup_{i=1}^3 B(x_i,r_i)$,
\item\label{B2 - g(x_i)=c_i} $g(x_i)=c_i$ for $i=1,2,3$,
\item\label{B3 - monotone edges} $g$ is monotone increasing or decreasing (not necessarily strictly) on each half-edge of $\partial W$, from its midpoint to a vertex,
\item\label{B4 - small energy} $\int_{W} |\nabla g|^2<\varepsilon$.
\end{enumerate}   
Furthermore, the balls $B(x_i,r_i)$ can be chosen to be arbitrarily small. Hence, by \ref{B1 - g=a outside balls} we may have that
\begin{enumerate}
\item[\mylabel{B5}{($B$5)}] $g$ has the value $a$ at the midpoints of the edges of $W$. 
\end{enumerate} 
The value $a$ is called the \textit{height} of $g$. See Figure \ref{fig:graph} for an illustration of the graph of four such functions. Finally, we require a monotonicity property:
\begin{enumerate}
\item[\mylabel{B6}{($B$6)}] $\osc_{\br W} (g)= \osc_{\partial W}(g)$.
\end{enumerate}

To construct such a function near the vertex $x_i$ of $W$, we may assume that $a=0$, $c_i=1$, and that $x_i=0\in \R^2$. The conceptual fact behind this construction is that the $2$-capacity of a point is equal to $0$; see \cite[Section 3]{FZ}. For $0<r<R$ consider the function 
\begin{align*}
g(x)=\begin{cases}
(\log(R/r)^{-1} \log(R/|x|), & r\leq |x|\leq R\\
1, & |x|<r\\
0, & |x|>R.
\end{cases}
\end{align*}
Then there exists a constant $C>0$ such that $\int |\nabla g|^2 \leq C \log (R/r)^{-1}$, which converges to $0$ as $r\to 0$. In fact, making $R$ smaller one sees that $g$ can be supported in an arbitrarily small neighborhood of $0$. Hence, the ball $B(x_i,r_i)$ with $r_i=R$ can be made arbitrarily small. One can now glue together three such functions, one near each vertex of $W$, to obtain the desired building block function. However, in order to prove the continuity of the function $f$ in Theorem \ref{Theorem Non-removability},  we will not use this particular function $g$, but we will need to make a more careful construction, towards the end of the proof, so that a certain modulus of continuity is satisfied.

\begin{remark}\label{Building block Continuity remark}
For the continuity of the function $f$ of Theorem \ref{Theorem Non-removability} we will need the properties \ref{B2 - g(x_i)=c_i}, \ref{B3 - monotone edges}, \ref{B5}, and \ref{B6} of the building block functions. Properties \ref{B1 - g=a outside balls} and \ref{B4 - small energy} are only used to show that $f$ does not lie in $W^{1,2}(\R^2)$ in the next section.
\end{remark}

\subsection{Avoidance of \texorpdfstring{$W^{1,2}(\R^2)$}{W1,2(R2)}}\label{Section Avoidance}
The function $f$ will be defined inductively in the next section so that $f\equiv 0$ in the unbounded component of $\R^2\setminus K$ and in particular outside a fixed ball $B(x_0,R_0)$, $f\equiv 1$ in a fixed ball $B(x_0,r_0)$ contained in the central $w$-triangle of sidelength $1/2$, and $0\leq f\leq 1$. Inside each $w$-triangle $W$ the function $f$ will be equal to a suitable building block function $g_W$, so that global continuity is ensured; see Figure \ref{fig:graph}. We fix $\varepsilon>0$. By choosing $\|\nabla g_W\|_{L^2(W)}$ to be sufficiently small for each $W$, we may have
\begin{align*}
\|\nabla f\|_{L^2(\R^2)}=\|\nabla f\|_{L^2(\R^2 \setminus K)} =\sum_{W\in \mathcal W} \|\nabla g_W\|_{L^2(W)} <\varepsilon.
\end{align*}
We remark that the ball $B(x_0,r_0)$ on which $f\equiv 1$ and the ball $B(x_0,R_0)$ outside of which $f\equiv 0$ are independent of $\varepsilon$.

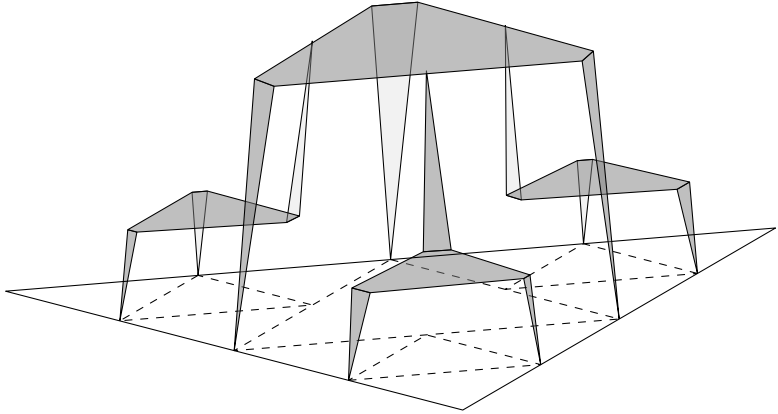
\begin{figure}
\centering
\begin{tikzpicture}[line join=bevel,z=-7,y=10, scale=10,rotate around y=20]

%base triangle
\coordinate (A1) at (0,0,0);
\coordinate (A2) at (1,0,0);
\coordinate (A3) at (0.5,0,0.866);
%midpoint vertices and central triangle
\coordinate (B1) at (0.5,0,0);
\coordinate (B2) at (0.75,0,0.433);
\coordinate (B3) at (0.25,0,0.433);
\coordinate (B11) at (0.47,1,0.05);
\coordinate (B12) at (0.53,1,0.05);
\coordinate (B21) at (0.72,1,0.40);
\coordinate (B22) at (0.70,1,0.45);
\coordinate (B31) at (0.28,1,0.40);
\coordinate (B32) at (0.30,1,0.45);

\coordinate (C1) at (0.25,0,0);
\coordinate (C2) at (0.125,0,0.216);
\coordinate (C3) at (0.375,1,0.216);
\coordinate (C11) at (0.24,0.33,0.02);
\coordinate (C12) at (0.26, 0.33, 0.02);
\coordinate (C31) at (0.36, 0.33, 0.2);
\coordinate (C32) at (0.34, 0.33, 0.23);
\coordinate (C21) at (0.1473, 0.33, 0.2151);
\coordinate (C22) at (0.1373, 0.33, 0.1978);

\draw[fill=black!10, fill opacity=.5] (C1)--(C11)--(C12)-- cycle;
\draw[fill=black!10, fill opacity=.5] (C3)--(C31)--(C32)-- cycle;
\draw[fill=black!50, fill opacity =0.5] (C2)--(C21)--(C22)-- cycle;
\draw[fill=black!50, fill opacity =0.5] (C11)--(C12)--(C31)--(C32)--(C21)--(C22)-- cycle;

%draw gasket
\draw[dashed] (B1)--(B2)--(B3)--cycle;
\draw[dashed] (C1)--(C2)--(0.375,0,.216)--cycle;

\begin{scope}[shift={(1,0,0)}, rotate around y=-120]
\coordinate (C1) at (0.25,0,0);
\coordinate (C2) at (0.125,0,0.216);
\coordinate (C3) at (0.375,1,0.216);
\coordinate (C11) at (0.23,0.33,0.02);
\coordinate (C12) at (0.27, 0.33, 0.02);
\coordinate (C31) at (0.36, 0.33, 0.2);
\coordinate (C32) at (0.34, 0.33, 0.23);
\coordinate (C21) at (0.1473, 0.33, 0.2151);
\coordinate (C22) at (0.1373, 0.33, 0.1978);

\draw[fill=black!50, fill opacity=.5] (C1)--(C11)--(C12)-- cycle;
\draw[fill=black!10, fill opacity=.5] (C3)--(C31)--(C32)-- cycle;
\draw[fill=black!10, fill opacity=.5] (C2)--(C21)--(C22)-- cycle;
\draw[fill=black!50, fill opacity =0.5] (C11)--(C12)--(C31)--(C32)--(C21)--(C22)-- cycle;

%gasket
\draw[dashed] (C1)--(C2)--(0.375,0,.216)--cycle;
\end{scope}

\begin{scope}[shift={(.5,0,.866)}, rotate around y=120]
\coordinate (C1) at (0.25,0,0);
\coordinate (C2) at (0.125,0,0.216);
\coordinate (C3) at (0.375,1,0.216);
\coordinate (C11) at (0.23,0.33,0.02);
\coordinate (C12) at (0.27, 0.33, 0.02);
\coordinate (C31) at (0.36, 0.33, 0.2);
\coordinate (C32) at (0.34, 0.33, 0.23);
\coordinate (C21) at (0.1473, 0.33, 0.2151);
\coordinate (C22) at (0.1373, 0.33, 0.1978);

\draw[fill=black!50, fill opacity=.5] (C1)--(C11)--(C12)-- cycle;
\draw[fill=black!50, fill opacity =0.5] (C3)--(C31)--(C32)-- cycle;
\draw[fill=black!50, fill opacity =0.5] (C2)--(C21)--(C22)-- cycle;
\draw[fill=black!50, fill opacity =0.5] (C11)--(C12)--(C31)--(C32)--(C21)--(C22)-- cycle;

%gasket
\draw[dashed] (C1)--(C2)--(0.375,0,.216)--cycle;
\end{scope}

\draw (A1) -- (A2) -- (A3) -- cycle;
\draw[fill=black!10, fill opacity=.5] (B1)--(B11)--(B12)-- cycle;
\draw[fill=black!50, fill opacity =0.5] (B2)--(B21)--(B22)-- cycle;
\draw[fill=black!50, fill opacity =0.5] (B3)--(B31)--(B32)-- cycle;
\draw[fill=black!50, fill opacity =0.5] (B11)--(B12)--(B21)--(B22)--(B32)--(B31)--cycle;

\end{tikzpicture}
\caption{Illustration of the graph of $f$ on $w$-triangles of level at most $2$.}\label{fig:graph}
\end{figure}

Now, we wish to prevent $f$ from lying in $W^{1,2}(\R^2)$. Suppose that $f\in W^{1,2}(\R^2)$, and in particular that $f$ is absolutely continuous on almost every line; see e.g.\ \cite[Section 26]{Va}. Let $I_t$, $0\leq t\leq r_0$, be the family of horizontal segments $[0,1]\times \{t\}$, translated and scaled, so that $I_t$ starts inside $B(x_0,r_0)$ and ends outside $B(x_0,R_0)$ for each $t$. Since $f\equiv 0$ outside $B(x_0,R_0)$ and $f\equiv 1$ in $B(x_0,r_0)$, by the absolute continuity on almost every line we obtain
\begin{align*}
1 \leq \int_{I_t} |\nabla f|\, ds
\end{align*}
for a.e.\ $t\in [0,r_0]$. Integrating over $t\in[0,r_0]$ and applying Fubini's theorem and the Cauchy-Schwarz inequality we have
\begin{align*}
r_0\leq \int_{B(x_0,R_0)} |\nabla f| \leq \|\nabla f\|_{L^2(\R^2)} \pi^{1/2}R_0 < \varepsilon \pi^{1/2}R_0.
\end{align*}
If $\varepsilon$ is sufficiently small, we obtain a contradiction. Hence, by choosing a small $\varepsilon>0$ we may have that $f\notin W^{1,2}(\R^2)$.

\subsection{Inductive choice of parameters}\label{Section Choice}
Here, we give the inductive construction of $f$.

We let $f= 0$ on the closure of the $w$-triangle of level $0$ (i.e., the closure of the unbounded component of $\R^2\setminus K$), and we define $f$ in the closure of the central $w$-triangle of level $1$ to be a building block function with parameters $a=1$ and  $c_i=0$ for $i=1,2,3$. In particular, $f\equiv 1$ in a fixed ball $B(x_0,r_0)$, by the property \ref{B1 - g=a outside balls} of the building block function.

Once $f$ has been defined on the closure of $w$-triangles of level $m-1$, we define $f$ on each triangle $W\in \mathcal W$ of level $m$ as follows. Note that the vertices $x_{i}$, $i=1,2,3$, of the triangle $W$ lie on the boundaries of triangles of level at most $m-1$. Hence, the function $f$ has already been defined on the vertices of $W$. We now set 
\begin{align*}
c_{i}&=f(x_{i}),\, i=1,2,3, \quad \textrm{and}\\
a&=\frac{1}{3}( c_{1}+c_{2}+c_{3}).
\end{align*}
Define $f$ on $\br W$ to be equal to a building block function with these parameters; see Figure \ref{fig:graph}. We also set 
\begin{align*}
\mathcal O(W) =\max_{i=1,2,3} |a-c_i|,
\end{align*}
which controls the oscillation of $f$ on $W$. In particular,
\begin{align}\label{osc(f) less O(W)}
\osc_{\br W}(f)= \osc_{\partial W}(f) \leq 2\mathcal O(W),
\end{align}
by properties \ref{B3 - monotone edges} and \ref{B6}. Proceeding inductively, $f$ is defined on $W_\infty = \bigcup_{W\in \mathcal W} \br W$.

One important observation is that the function $f$ has a monotonicity property outside the central $w$-triangle of level $1$; here monotonicity is to be understood in the sense that the maximum and minimum  on open sets is attained at the boundary of these sets. Of course, validity of such a monotonicity property depends partly on the building block functions. We now formulate more precisely and prove the form of monotonicity that we will need.

\begin{lemma}\label{Monotonicity}
For each $v$-triangle $V$ of level $m\geq 1$ we have
\begin{align*}
\sup_{x\in \br {V} \cap W_\infty }f(x) &= \max_{x\in \partial V} f(x), \quad \textrm{and}\\
\inf_{ x\in \br {V} \cap W_\infty}f(x) &=\min_{x\in \partial V} f(x).
\end{align*}
In particular,
\begin{align*}
\osc_{\br V\cap W_\infty}(f)= \sup_{x\in \br {V} \cap W_\infty }f(x)-\inf_{x\in \br {V} \cap W_\infty }f(x)= \osc_{\partial V}(f).
\end{align*}
\end{lemma}
Note that $\partial V$ is contained in the union of the boundaries of the $w$-triangles, so $f$ is already defined there and all the expressions that appear in the lemma make sense. 

\begin{proof}
Assume that $W_1\subset V$ is the $w$-triangle whose vertices  $x_{i}$, $i=1,2,3$, lie on $\partial V$. Then, by the averaging definition of $f$ and the monotonicity properties \ref{B3 - monotone edges} and \ref{B6} of the building block functions, it follows that the maximum and minimum of $f$ on $\br W_1$ are attained at the vertices of $W_1$, i.e.,
\begin{align*}
\min_{i\in \{1,2,3\}}f(x_{i}) \leq f(z) \leq \max_{i\in \{1,2,3\}} f(x_{i})
\end{align*}  
for all $z\in \br {W}_1$. Hence,
\begin{align*}
\max_{x\in \partial V} f(x)\leq f(z)\leq \max_{x\in \partial V} f(x)
\end{align*}
for all $z\in \br{W}_1$. 

Let $V_2$ be one of the three $v$-triangles of $V\setminus \br W_1$, and let $W_2\subset V_2$ be the central $w$-triangle whose vertices lie on $\partial V_2$. If the vertices of $W_2$ are $y_i$, $i=1,2,3$, then as before we have
\begin{align*}
\min_{i\in \{1,2,3\}}f(y_{i}) \leq f(z) \leq \max_{i\in \{1,2,3\}} f(y_{i})
\end{align*}  
for all $z\in \br W_2$. The vertices $y_i$ lie on $\partial V_2\subset \partial W_1\cup  \partial V$, hence
\begin{align*}
\max_{i\in \{1,2,3\}} f(y_{i}) \leq \max_{x \in \partial W_1 \cup \partial V} f(x) = \max_{x\in \partial V} f(x),
\end{align*}
by our conclusion for $W_1$. The analog of this statement also holds for the minimum, hence, we obtain in this case
\begin{align*}
\max_{x\in \partial V} f(x)\leq f(z)\leq \max_{x\in \partial V} f(x)
\end{align*}
for all $z\in \br{W}_2$. The proof of the general statement follows with the same argument by induction.
\end{proof}

\subsection{Proof of Continuity}

\begin{prop}\label{Continuity}
The function $f\colon  W_\infty \to \R$ is uniformly continuous and thus, it has a continuous extension to $\R^2$, which is the closure of $W_\infty =\bigcup_{W\in \mathcal W} \br W$.
\end{prop}
The way to interpret this statement is that \textit{there exists} a choice of building block functions that makes $f$ continuous. As we remarked in Section \ref{Section Building block}, we cannot use a ``generic" building block function, but we have to make a careful construction.

The proof of continuity relies on the next crucial lemma. Recall the definition of $\mathcal O(W)$ from Section \ref{Section Choice}.

\begin{lemma}\label{Basic lemma}
For each $\varepsilon>0$ there exist at most finitely many $w$-triangles $W$ with 
\begin{align*}
\mathcal O(W)>\varepsilon.
\end{align*}
In particular, for each $\varepsilon>0$ there exist at most finitely many $w$-triangles $W$ with 
\begin{align*}
\osc_{\br W}(f) >\varepsilon.
\end{align*}
\end{lemma}

Assuming the lemma, we prove Proposition \ref{Continuity}.

\begin{proof}[Proof of Proposition \ref{Continuity}]
Using Lemma \ref{Basic lemma}, we fist prove:
\begin{claim} \label{Claim:uniform continuity}
For each $\varepsilon>0$ there exist at most finitely many $v$-triangles $V$ with 
\begin{align}\label{Continuity osc V}
\osc_{\br V\cap W_\infty}(f)>\varepsilon.
\end{align}
\end{claim}

We argue by contradiction, assuming that there exists $\varepsilon>0$ such that the above holds for infinitely many $v$-triangles. Let $V$ be one of them. Then by the monotonicity of $f$ from Lemma \ref{Monotonicity} we have $\osc_{\partial V}(f)>\varepsilon$. Each edge of $\partial V$ is contained in an edge of a $w$-triangle, so there exist three (possibly non-distinct) $w$-triangles $W_1$, $W_2$, and $W_3$ such that
\begin{align*}
\varepsilon < \osc_{\partial V}(f) \leq \osc_{\partial V\cap \partial W_1}(f)+\osc_{\partial V\cap \partial W_2}(f)+\osc_{\partial V\cap \partial W_3}(f).
\end{align*}
In particular, for one of them, say for $W$, we have $\osc_{\partial V\cap \partial W}(f)>\varepsilon$. We thus see that each $v$-triangle of the set $\{V: \osc_{\br V\cap W_\infty}(f)>\varepsilon\}$ corresponds to a $w$-triangle $W$ such that $\osc_{\partial V\cap \partial W}(f)>\varepsilon/3$. Moreover, this correspondence is finite-to-one. Indeed, if there were infinitely many $v$-triangles $V_n$, $n\in \N$, corresponding to a single $w$-triangle $W$, then the diameters of $V_n$ would shrink to $0$. However, the uniform continuity of the restriction of $f$ to $\partial W$ would imply that $\osc_{\partial V_n\cap \partial W}(f) \to 0$, a contradiction. It follows that there exist infinitely many $w$-triangles $W$ with the property that there exists a $v$-triangle $V$ such that 
\begin{align*}
\varepsilon/3 <\osc_{\partial V\cap \partial W}(f)\leq \osc_{\partial W}(f) \leq 2\mathcal O(W),
\end{align*}
where we used \eqref{osc(f) less O(W)}. This contradicts Lemma \ref{Basic lemma}.

Now, we prove that $f$ is uniformly continuous on $W_\infty$. We argue by contradiction, assuming that there exists $\varepsilon>0$ and sequences $x_n,y_n\in W_\infty$ with $|x_n-y_n|\to 0$ such that $|f(x_n)-f(y_n)|\geq \varepsilon$ for all $n\in \N$. The sequences $x_n,y_n$ cannot escape to $\infty$ since $f$ is identically equal to $0$ in a neighborhood of $\infty$. Consider an accumulation point $z$ of $x_n$ and $y_n$, and by passing to a subsequence, assume that $x_n,y_n \to z$. Note that $z$ cannot lie in the interior of any $w$-triangle, since the function $f$ is already continuous there. Hence, $z\in K$ and we split into three cases. 

Suppose first that $z$ is of vertex type. By Lemma \ref{Combinatorics V lemma}(I) there exist two (possibly non-distinct) $w$-triangles $A$ and $B$ containing $z$ on their boundary and two (possibly non-distinct) sequences of $v$-triangles $\br V_k$ and $\br V_k'$ shrinking to $z$ such that for each $k\in \N$ the set $\br A\cup \br B\cup \br V_k\cup \br V_k'$ contains all sufficiently small neighborhoods of $z$. We fix $k$ and a small $r>0$ such that  $B(z,r)\subset \br A\cup \br B\cup \br V_k\cup \br V_k'$. Since $z\in \partial A\cap \partial B\cap  \br V_k  \cap \br V_k'$, for each $p\in B(z,r)$ we have
\begin{align*}
|f(p)-f(z)| \leq \osc_{\br A\cap B(z,r)} (f) + \osc_{\br B\cap B(z,r)}(f) +\osc_{\br V_k\cap W_\infty}(f)+\osc_{\br V_k'\cap W_\infty}(f).
\end{align*}
By choosing a sufficiently large $k$, we may have that $\osc_{\br V_k\cap W_\infty}(f)<\varepsilon/8$ and $\osc_{\br V_k'\cap W_\infty}(f)<\varepsilon/8$, by Claim \ref{Claim:uniform continuity}. By choosing $r>0$ to be sufficiently small, using the uniform continuity of the restriction of $f$ on $\br A$ and $\br B$ we may also have $\osc_{\br A\cap B(z,r)} (f)<\varepsilon/8$ and $\osc_{\br B\cap B(z,r)} (f)<\varepsilon/8$. It follows that $|f(p)-f(z)|<\varepsilon/2$ for all $p\in B(z,r)$. Now, if $n$ is sufficiently large, then $x_n,y_n\in B(z,r)$, hence 
\begin{align*}
|f(x_n)-f(y_n)|\leq |f(x_n)-f(z)|+|f(y_n)-f(z)| < \varepsilon/2+ \varepsilon/2=\varepsilon,
\end{align*}
which is a contradiction.

If $z$ is of edge type, then applying Lemma \ref{Combinatorics V lemma}(II) we obtain a $w$-triangle $B$ with $z\in \partial B$ and a sequence of $v$-triangles $\br V_k$ shrinking to $z$ such that for each $k\in \N$ the set $\br B\cup \br V_k$ contains all sufficiently small neighborhoods of $z$. One now argues exactly as in the previous case, using Claim \ref{Claim:uniform continuity} or using the uniform continuity of the restriction of $f$ on $\br B$. It follows that for each $k\in \N$ there exists a small $r>0$ such that for all  $p\in B(z,r)$ we have 
\begin{align*}
|f(p)-f(z)|\leq \osc_{\br B\cap B(z,r)}(f) +\osc_{\br V_k\cap W_\infty}(f)< \varepsilon/8+\varepsilon/8 =\varepsilon/4.
\end{align*}
Since $x_n,y_n\in B(z,r)$ for all sufficiently large $n$, we obtain again a contradiction to the assumption that $|f(x_n)-f(y_n)|\geq \varepsilon$ for $n\in \N$.

Finally, suppose that $z\in K^\circ$. By Lemma \ref{Combinatorics V lemma}(III) there exists a sequence of $v$-triangles $V_k$ shrinking to $z$ such that $z\in V_k$ for all $k\in \N$. It follows that for each $k$ there exists a large $n$ such that $x_n,y_n\in V_k$. In particular, we have
\begin{align*}
|f(x_n)-f(y_n)|\leq \osc_{\br V_k \cap W_\infty}(f).
\end{align*}
If we choose a sufficiently large $k\in \N$ then the latter is less than $\varepsilon$ by Claim \ref{Claim:uniform continuity} and we obtain a contradiction.
\end{proof}

Finally, we prove the basic Lemma \ref{Basic lemma}.

\begin{proof}[Proof of Lemma \ref{Basic lemma}]
We argue by contradiction, assuming that for some $\varepsilon_0>0$ we have 
\begin{align}\label{Basic lemma osc}
\osc_{\br{W}}(f)> \varepsilon_0
\end{align}
for infinitely many $w$-triangles $W$. We split in three cases. 

\textbf{Case 1:} There exist infinitely many  $w$-triangles satisfying \eqref{Basic lemma osc} and converging to a point $z\in K$ of vertex type. 

By Lemma \ref{Combinatorics lemma}(I), the vertex $z$ is (contained in) the intersection of the closures of two fixed $w$-triangles $A$ and $B$ (these might be non-distinct if they are the unbounded $w$-triangle). Since the restriction of $f$ is uniformly continuous on $\br{A\cup B}$, for each $\varepsilon>0$ there exists $\delta>0$ such that for all $x,y\in \br {A\cup B}$ with $|x-y|<\delta$ we have $|f(x)-f(y)|<\varepsilon$. 

Using the notation from Lemma \ref{Combinatorics lemma}(I), we consider the sequences of triangles $W_n$ and $V_n$. For sufficiently large $n$, the triangle $V_n$ has $z$ on its boundary, and is contained in the $\delta/2$-neighborhood of $z$. Note that each triangle $W_n$ has its vertices on $\partial W_{n-1}$, $\partial A$, and $\partial B$ for all sufficiently large $n$. Assume that all of the above hold for $n\geq N$.

We denote by $c_n$ the  height of $f$ on $W_n$ (recall the definition of the height of a building block function in Section \ref{Section Building block}), and note that for $n>N$ we have 
\begin{align*}
c_n=\frac{1}{3}(c_{n-1}+ c_{n,A}+c_{n,B}),
\end{align*} 
where $c_{n,A}$ and $c_{n,B}$ are the values of $f$ on the vertex of $W_n$ lying on $\partial A$ and $\partial B$, respectively. Note that the vertex of $W_n$ lying on $\partial W_{n-1}$ is the midpoint of an edge of $\partial W_{n-1}$. Hence, by property \ref{B5}, the value of $f$ at this vertex is equal to the height of $f$ on the triangle $W_{n-1}$, i.e., $c_{n-1}.$ 

Our goal is to find a sequence $\{\Delta_n\}_{n\in \N}$ such that 
\begin{align}\label{Basic lemma case1}
\mathcal O (W_n)= \max \{|c_n-c_{n-1}|,|c_n-c_{n,A}|, |c_n-c_{n,B}| \} \leq \Delta_n
\end{align}
for $n\in \N$, and $\Delta_n\leq 3\varepsilon$ for sufficiently large $n$. Then 
\begin{align*}
\osc_{\partial V_n}(f) &\leq \osc_{\partial V_n \cap \partial A}(f) +\osc_{\partial V_n\cap \partial B}(f)+ \osc_{\partial V_n\cap \partial W_{n-1}}(f)\\
&\leq 2\varepsilon+ \osc_{\br {W}_{n-1}} (f) \leq 2\varepsilon +2\Delta_{n-1}\leq 8\varepsilon 
\end{align*}
for all sufficiently large $n$, where we used \eqref{osc(f) less O(W)}.

Note that there are at most two nested sequences $\{V_n\}_{n\in \N}$ and $\{V_n'\}_{n\in \N}$ shrinking to $z$, for which the above bounds hold; see Lemma \ref{Combinatorics V lemma}(I). If $W$ is a small $w$-triangle near $z$ satisfying \eqref{Basic lemma osc}, then it has to be contained in $V_n$ or $V_n'$ for some large $n$, by Lemma \ref{Combinatorics V lemma}(I). Using the monotonicity of $f$ we see that 
$$\osc_{\br W}(f) \leq \max\{\osc_{\partial V_n} (f), \osc_{\partial V_n'}(f)\} \leq 8\varepsilon.$$ 
This contradicts \eqref{Basic lemma osc} if we choose $\varepsilon<\varepsilon_0/8$.

We proceed to the proof of \eqref{Basic lemma case1}. For $1\leq n\leq N$ we use the trivial bound $\mathcal O(W_n)\leq 1\eqqcolon \Delta_n$. Once $\Delta_{n-1}$ has been defined and satisfies \eqref{Basic lemma case1}, for $n>N$ we have
\begin{align*}
|c_n-c_{n-1}|&\leq \frac{1}{3}( |c_{n,A}-c_{n-1}| +|c_{n,B}-c_{n-1}|)\\
&\leq \frac{1}{3}( |c_{n,A}-c_{n-1,A}|+|c_{n-1,A}-c_{n-1}|\\
&\quad\quad + |c_{n,B}-c_{n-1,B}|+|c_{n-1,B}-c_{n-1}|)\\
&\leq \frac{1}{3}( \varepsilon+ \Delta_{n-1} + \varepsilon +\Delta_{n-1}) \\
&\leq  \frac{2\varepsilon}{3} + \frac{2}{3} \Delta_{n-1},\\
|c_n-c_{n,A}|&\leq \frac{1}{3} (|c_{n-1}-c_{n,A}|+ |c_{n,B}-c_{n,A}|)\\
&\leq \frac{1}{3}(\varepsilon +\Delta_{n-1} +\varepsilon)\\
&\leq \frac{2\varepsilon}{3}+ \frac{1}{3}\Delta_{n-1}, \quad \textrm{and}\\
|c_n-c_{n,B}|&\leq  \frac{2\varepsilon}{3}+ \frac{1}{3}\Delta_{n-1}.
\end{align*}
Here, we used the fact from Lemma \ref{Combinatorics lemma}(I) that the vertices $\partial W_{n}\cap \partial A$, $\partial W_{n-1}\cap \partial  A$, and $\partial W_{n}\cap \partial B$, $\partial W_{n-1}\cap \partial B$ are all contained in $\partial V_n$, which lies in the $\delta/2$-neighborhood of $z$. Thus, we may choose $\Delta_n\coloneqq  \frac{2\varepsilon}{3}+ \frac{2}{3}\Delta_{n-1}$ for $n>N$, which yields
\begin{align*}
\Delta_n= \frac{2\varepsilon}{3} + \dots +\frac{2^{n-N}\varepsilon}{3^{n-N}} +\frac{2^{n-N}}{3^{n-N}}\leq 2\varepsilon+ \frac{2^{n-N}}{3^{n-N}}.
\end{align*}
Since $\Delta_n\leq 3\varepsilon$ for sufficiently large $n$, we have the desired conclusion.

\textbf{Case 2:} There exist infinitely many $w$-triangles satisfying \eqref{Basic lemma osc} and converging to a point $z\in K\setminus K^\circ$ that is of edge type, i.e., it lies on an open edge of a $w$-triangle $B$.  

We consider the unique sequences of triangles $V_n$ and $W_n$ converging to $z$, as in Lemma \ref{Combinatorics lemma}(II). We fix a small $\varepsilon>0$ and consider $\delta>0$ such that $|f(x)-f(y)|<\varepsilon$ whenever $|x-y|<\delta$ and $x,y\in \partial B$. Assume that $V_n$ is contained in the $\delta/2$-neighborhood of $z$ for $n>N$. Then arguing as in Case 1, we wish to find a sequence $\Delta_n$ that bounds the oscillation of $f$ on $W_n$, and $\Delta_n$ is sufficiently small, depending on $\varepsilon$. 

This time we have
\begin{align*}
c_n= \frac{1}{3} (c_{n-1}+c_{n,A_n} +c_{n,B}).
\end{align*}
for $n\geq 1$; also here, $c_n$ is the height of $f$ on $W_n$ and we use again property \ref{B5}. By Lemma \ref{Combinatorics lemma}(II), $A_n=W_{k(n)}$ for, say, $n>N$, where $k(n)\to \infty$. We set $ c_{n,k(n)} \coloneqq c_{n,A_n}$. In general, if a $w$-triangle $W_n$ has a vertex on $\partial W_m$ for some $m$, then the value of $f$ at that vertex is denoted by $c_{n,m}$.

Our claim now is that for each $m\in \N$ there exists $N_m\in \N$ and $\Delta_m$ such that 
\begin{align*}
\mathcal O(W_n)\leq \Delta_m
\end{align*}
for $n>N_m$, and $\Delta_m\leq 5\varepsilon$ for sufficiently large $m$. 

We assume this for the moment. If $W$ is a small $w$-triangle sufficiently close to $z$, then $W\subset V_{n}$ for some large $n$, by Lemma \ref{Combinatorics V lemma}(II). The boundary of $V_{n}$ is contained in $\partial W_{n-1}$, $\partial W_{k(n)}$, and $\partial B$; see Lemma \ref{Combinatorics lemma}(II). Also $V_{n}$ lies in the $\delta/2$-neighborhood of $z$ for large $n$. Hence, 
\begin{align*}
\osc_{\br W}(f)&\leq \osc_{\partial V_{n}}(f) \leq \varepsilon +2\mathcal O (W_{n-1}) +2\mathcal O(W_{k(n)})\\
&\leq \varepsilon + 4\Delta_{m}
\end{align*}
provided that $n-1,k(n)>N_m$. If $W$ is sufficiently close to $z$, then $V_n$ can be chosen to be sufficiently close to $z$, so $n-1$, $k(n)$ and $m$ can be large enough, in order to have $n-1,k(n)>N_m$, and $\Delta_m\leq 5\varepsilon$. Hence,
\begin{align*}
\osc_{\br W}(f) \leq 21\varepsilon
\end{align*}
and this contradicts \eqref{Basic lemma osc}, if we choose $\varepsilon< \varepsilon_0/21$.

Now, we prove our claim. For $m=1$ we use the trivial bound $\mathcal O(W_n)\leq 1\eqqcolon \Delta_1$, which holds for all $n\in \N$. If $N_{m-1}$ has been chosen, we choose $N_m>N$ to be so large that $n-1,k(n)>N_{m-1}$ for all $n>N_m$. This can be done since $k(n)\to \infty$ by Lemma \ref{Combinatorics lemma}(II). For $n>N_m$ we have
\begin{align*}
|c_{n}-c_{n-1}|&\leq \frac{1}{3}( |c_{n,k(n)}-c_{n-1}| +|c_{n,B}-c_{n-1}|).
\end{align*}
If $|c_{n,k(n)}- c_{n-1,k(n)}|\leq \Delta_{m-1}/2$ then we have
\begin{align*}
|c_n-c_{n-1}|&\leq \frac{1}{3}( |c_{n,k(n)}-c_{n-1,k(n)}|+|c_{n-1,k(n)}-c_{n-1}| \\
&\quad \quad + |c_{n,B}- c_{n-1,B}|+|c_{n-1,B}-c_{n-1}|)\\
&\leq \frac{1}{3}(\Delta_{m-1}/2 + \Delta_{m-1}+\varepsilon+\Delta_{m-1} )\\
&\leq \frac{2\varepsilon}{3}+ \frac{5}{6} \Delta_{m-1}.
\end{align*}
Here we used the fact from Lemma \ref{Combinatorics lemma}(II) that $\br V_n$ contains the vertices $\partial W_{n-1} \cap \partial B$ and $\partial W_{n}\cap \partial B$, so they are $\delta$-close to each other. 

If $|c_{n,k(n)}- c_{n-1,k(n)}|> \Delta_{m-1}/2$, then we necessarily have $|c_{n,k(n)}- c_{k(n),B}|\leq \Delta_{m-1}/2$, where $c_{k(n),B}$ denotes the value of $f$ at the vertex of $W_{k(n)}$ lying on $\partial B$. This is because the vertices $\partial W_{n-1}\cap \partial W_{k(n)}$, $\partial W_{n}\cap \partial W_{k(n)}$,  $\partial W_{k(n)}\cap \partial B $ are ordered points, contained in a half-edge of $W_{k(n)}$ (by Lemma \ref{Combinatorics lemma}), where $f$ is monotone increasing or decreasing by property \ref{B3 - monotone edges} in Section \ref{Section Building block}. On the other hand, by the induction assumption, the oscillation of $f$ on this half-edge is bounded by $\mathcal O(W_{k(n)}) \leq \Delta_{m-1}$, since $k(n)>N_{m-1}$. In this case, we have
\begin{align*}
|c_{n}-c_{n-1}|&\leq \frac{1}{3}( |c_{n,k(n)}-c_{k(n),B}| +|c_{k(n),B}-c_{n-1,B}|+ |c_{n-1,B}-c_{n-1}|\\
& \quad \quad +|c_{n,B}-c_{n-1,B}|+|c_{n-1,B}-c_{n-1}| )\\
&\leq \frac{1}{3}( \Delta_{m-1}/2+ \varepsilon+ \Delta_{m-1}+\varepsilon +\Delta_{m-1})\\
&=\frac{2\varepsilon}{3}+ \frac{5}{6}\Delta_{m-1}.
\end{align*}
In the same way we also compute bounds for $|c_n-c_{n,k(n)}|$ and $|c_n-c_{n,B}|$, and we can show that they are all bounded by 
\begin{align*}
\mathcal O(W_n) \leq  \frac{2\varepsilon}{3}+ \frac{5}{6}\Delta_{m-1}\eqqcolon \Delta_m.
\end{align*}
Observe that
\begin{align*}
\Delta_m  \leq 4\varepsilon + \frac{5^{m-1}}{6^{m-1}}\leq 5\varepsilon
\end{align*}
for sufficiently large $m$, as desired.

\textbf{Case 3:} There exist infinitely many $w$-triangles satisfying \eqref{Basic lemma osc} and converging to a point $z\in K^\circ$. 

In the previous two cases the proof was mostly combinatorial, based on Lemma \ref{Combinatorics V lemma} and Lemma \ref{Combinatorics lemma}, on qualitative properties of the building block functions, and on the fact that the restriction of $f$ on the union of finitely many $w$-triangles is uniformly continuous. However, in this case it will be crucial to make a suitable choice of the building block functions, so that they have a certain modulus of continuity on each $w$-triangle near the vertices. 

Let $W_0$ be a $w$-triangle, and consider the corresponding oscillation $\mathcal O(W_0)$, which depends on the value of $f$ at the vertices of $W_0$, by the inductive definition of $f$. We require the following:
\subsubsection{Condition \texorpdfstring{$(\ast)$}{(*)}}\label{Section Condition *}
\begin{enumerate}
\item[($\ast$)] Assume that two $w$-triangles $W_1$ and $W_2$ are adjacent, and  each has a vertex on a triangle $\partial W_0$ of strictly lower level. If $z_1\in \partial W_1\cap\partial W_0$ and $z_2\in \partial W_2\cap\partial W_0$ are these vertices, then
\begin{align*}
|f(z_1)-f(z_2)|\leq \mathcal O(W_0)/3.
\end{align*}
\end{enumerate}
Recall that $W_1$ is adjacent to $W_2$ if $W_1$ has a vertex on $\partial W_2$, or vice versa. Note at this point the vertices of $W_0$ can only lie on triangles of strictly lower level from  that of $W_0$, by our observations in Section \ref{Section Definitions}, so in particular they do not lie on $W_1$ or $W_2$. Of course, we still require the initial properties \ref{B1 - g=a outside balls}--\ref{B6} of the building block functions from Section \ref{Section Building block}.

To construct $f$ on $W_0$ with the desired properties we work as follows. We fix an edge $I \subset \partial W_0$, and a vertex $z\in I$ of $W_0$. Assume that the edges of $W_0$ have length $1$, and that $I=[0,1]$,  $z=0$. We consider the points $2^{-k}$, $k\geq 2$, on $I$, and define a radial function with the following procedure. For a fixed $N\in \N$ we define $f$ on the annulus $A_1\coloneqq A(0;2^{-3},2^{-2})$ to be a radial function that is equal to $0$ in the outer circle and increases to $1/N$ in the inner circle, with slope $\simeq 2^3/N$. For $1\leq k\leq N$, we define $f$ in $A_k\coloneqq A(0;2^{-2-k}, 2^{-1-k})$ to be a radial function that is equal to $(k-1)/N$ in the  outer circle and increases to $k/N$ in the inner circle, with slope $\simeq 2^{2+k}/N$. In the ball $B(0,2^{-2-N})$ we set $f\equiv 1$, outside the ball $B(0,2^{-2})$ we set $f\equiv 0$, and then we restrict $f$ to the triangle $\br W_0$. Then $f\in W^{1,2}(W_0)$ and
\begin{align*}
\int_{W_0} |\nabla f|^2 = \sum_{k=1}^N  \int_{W_0\cap A_k} |\nabla f|^2 \lesssim \sum_{k=1}^N \left(\frac{2^{2+k}}{N} \right)^2 m_2(W_0\cap A_k) \simeq \sum_{k=1}^N \frac{1}{N^2} \simeq \frac{1}{N}.
\end{align*}
Hence, by choosing a large $N$ we can achieve both that the $f$ has small energy, and that 
\begin{align*}
|f(2^{-k}) -f(2^{-k-1})| \leq \frac{1}{N }\leq \frac{1}{3}
\end{align*}
for all $k\geq 1$. Note that the same bounds hold for the corresponding dyadic points lying on the other edge of $W_0$ that is connected to $0$, and is a rotation of $I$ by $60$ degrees. Of course, the assumptions that $f=1$ near $z=0$ and $f=0$ outside $B(0,2^{-2})$ are not restrictive, since by rescaling $f$ and choosing a sufficiently large $N$ we can achieve the oscillation we wish with small energy.

We now consider dyadic points as above on each half-edge of $W_0$, converging to the corresponding vertex, and do a similar construction for all vertices. Near each vertex we may have that $f$ oscillates radially from a given value to the desired height and also has the property that if $x_1$ and $x_2$ are adjacent dyadic points lying on an edge $I$ of $W_0$ then 
\begin{align}\label{Basic lemma O(W)/3}
|f(x_1)-f(x_2)| \leq  \mathcal O(W_0)/3.
\end{align} 
Again, by dyadic points we mean points of the form $2^{-k}$ and $1-2^{-k}$ for $k\in \N$, once we scale $W_0$ so that $I=[0,1]$ is an edge of $W_0$. The properties \ref{B1 - g=a outside balls}--\ref{B6} in Section \ref{Section Building block} hold by construction. Especially, note that the property \ref{B3 - monotone edges} holds, since $f$ is radially increasing or decreasing near each vertex. 

To check property ($\ast$) we first use a scaling followed by a rotation of the Sierpi\'nski gasket so that the points $z_1,z_2$ in question lie on the edge $I=[0,1]$ of  $\partial W_0$. One now has to observe that if $z_1,z_2\in I\subset \partial W_0$ are vertices of adjacent triangles $W_1,W_2$, and they are not vertices of $W_0$, then they must both lie in one of the closed dyadic intervals of the form $[2^{-k-1},2^{-k}]$ or $[1-2^{-k},1-2^{-k-1}]$, $k\geq 1$. Hence, using the monotonicity of $f$ on these intervals (property \ref{B3 - monotone edges}) and \eqref{Basic lemma O(W)/3} we obtain
\begin{align*}
|f(z_1)-f(z_2)|\leq \mathcal O(W_0)/3.
\end{align*} 

To prove the observation mentioned, we first note that the two vertices of $\partial W_0$ that are endpoints $I$, actually lie on the edges of two triangles $A$ and $B$. The triangles $W_0,A$, and $B$ bound a $v$-triangle $V$, and each $w$-triangle that has a vertex on $I$ must be contained in $V$. We consider the central $w$-triangle $W(1/2)\subset V$ that has a vertex at the midpoint of $I$. The points $2^{-2}, 1-2^{-2}\in I$ are vertices of triangles $W(2^{-2})$, $W(1-2^{-1})$, respectively, which also have a vertex on $\partial W(1/2)$. Inductively,  the points $2^{-k},1-2^{-k}\in I$, $k\geq 2$, are the vertices of triangles $W(2^{-k})$, $W(1-2^{-k})$, which have vertices on $W({2^{-k+1}})$, $W(1-2^{-k+1})$, respectively. Note that the $v$-triangles bounded by $\partial W_0$, $\partial W(2^{-k+1})$, and $\partial W(2^{-k})$, or by $\partial W_0$, $\partial W(1-2^{-k+1})$, and $\partial W(1-2^{-k})$ are disjoint for $k\geq 1$, and that the closures of these $v$-triangles cover $I$, except for its endpoints.

Now, if $W_1,W_2\subset V$ are adjacent triangles as in the statement of ($\ast$), and they are not equal to the ``dyadic" triangles $W(2^{-k})$, $W(1-2^{-k})$, $k\in \N$, that have a vertex on a dyadic point of $I$, then their vertices $z_1,z_2\in I$ cannot lie in distinct dyadic intervals. This is because in this case $W_1$ and $W_2$ would lie in disjoint $v$-triangles, and thus, they would not be adjacent. With a similar analysis one deals with the cases where one of $W_1,W_2$, or both, are equal to ``dyadic" triangles. In any case the vertices $z_1\in \partial W_1\cap \partial W_0$ and $z_2\in  \partial W_2\cap \partial W_0$ must both lie in one of the intervals of the form $[2^{-k-1},2^{-k}]$ or $[1-2^{-k},1-2^{-k-1}]$, $k\geq 1$, as desired. \qed
\newline

Now, we return to the main proof of Case 3, that \eqref{Basic lemma osc} cannot occur for infinitely many $w$-triangles near a point $z\in K^\circ= K\setminus W_\infty$.

We consider the sequence of nested $v$-triangles $\{V_n\}_{n\geq 1}$ converging to $z$ and the corresponding $w$-triangles $W_n\subset V_n$, given by Lemma \ref{Combinatorics lemma}(III). In this case, for sufficiently large  $n\in \N $ the triangle $W_n$ has a vertex on $\partial W_{n-1}$ and two vertices on boundaries of some $w$-triangles $ W_{k(n)},W_{l(n)}$, with $k(n),l(n)\to\infty$; see Figure \ref{fig:zoom}. 

Let $\Delta_1=1$. We claim that for each $m\geq 2$ there exists $N_m\in \N$ and $\Delta_{m}=7\Delta_{m-1}/9$ such that for $n>N_m$ we have
\begin{align*}
\mathcal O(W_n) \leq \Delta_{m}.
\end{align*}
If this is the case, then $\Delta_m= 7^{m-1}/9^{m-1}$ for $m\in \N$, and it follows that 
\begin{align*}
\mathcal O(W_n) \leq 7^{m-1}/9^{m-1}\leq \varepsilon
\end{align*}
for all sufficiently large $m$ and $n>N_m$, where $\varepsilon>0$ is to be chosen. The triangle $V_n$ that contains $W_n$ has boundary contained in $\partial W_{n-1}$, $\partial W_{k(n)}$, and $\partial W_{l(n)}$. Note that if $W$ is a $w$-triangle that is sufficiently close to $z$, then $W\subset V_n$ for some large $n$, by Lemma \ref{Combinatorics V lemma}(III). On the other hand, for any $w$-triangle $W\subset V_n$ by the monotonicity of $f$ in Lemma \ref{Monotonicity} and \eqref{osc(f) less O(W)} we have
\begin{align*}
\osc_{\br W}(f)\leq \osc_{\partial V_n} (f) \leq 2\mathcal O(W_{n-1})+2\mathcal O(W_{k(n)})+2\mathcal O(W_{l(n)})\leq 6\varepsilon,
\end{align*}
provided that $n$ is sufficiently large, so that $n-1,k(n),l(n)>N_m$. Choosing $\varepsilon< \varepsilon_0/6$ yields a contradiction to \eqref{Basic lemma osc}.

Now, we proceed with the proof of our claim. For $m=1$ we use the trivial bound $\mathcal O(W_n)\leq 1\eqqcolon \Delta_1$, which holds for all $n\in \N$. If $N_{m-1}$ has been chosen, we choose $N_m$ to be so large that for $n>N_m$ the vertices of the triangle $W_n$ lie on triangles $\partial W_{n-1}$, $\partial W_{k(n)}$, $\partial W_{l(n)}$, with $n-1,k(n),l(n)>N_{m-1}$. This can be done since $k(n),l(n)\to \infty$. Now we have
\begin{align*}
c_n= \frac{1}{3}( c_{n-1}+c_{n,k(n)}+c_{n,l(n)}),
\end{align*}
using the notation of Case 2. We assume that the  order of the triangles by level is $W_{l(n)},W_{k(n)},W_{n-1},W_n$, i.e., $l(n)<k(n)< n-1<n$, as in Lemma \ref{Combinatorics lemma}(III). In order to control $\mathcal O(W_n)$, we will find bounds for the differences $c_n-c_{n-1}$, $c_n-c_{n,k(n)}$, and $c_n-c_{n,l(n)}$.

We have
\begin{align*}
c_n-c_{n-1}&= \frac{1}{3}( c_{n,k(n)}-c_{n-1}+c_{n,l(n)}-c_{n-1})\\
&= \frac{1}{3}[(c_{n,k(n)}-c_{n-1,k(n)})+(c_{n-1,k(n)}-c_{n-1}) \\
&\quad \quad +(c_{n,l(n)}-c_{n-1,l(n)})+(c_{n-1,l(n)}-c_{n-1})]
\end{align*}
The first difference is bounded (in absolute value) by $\mathcal O(W_{k(n)})/3 \leq \Delta_{m-1}/3$, by the induction assumption and the property ($\ast$) of the function $f$; note that the triangles $W_n$ and $W_{n-1}$ are adjacent and they both have a vertex on $\partial W_{k(n)}$. With the same reasoning the third difference is also bounded by $\mathcal O(W_{l(n)})/3 \leq \Delta_{m-1}/3$. Finally, by the definition of $f$ on $W_{n-1}$, the sum of the second and fourth differences is equal to $c_{n-1}-c_{n-1,D}$, where $c_{n-1,D}$ is the value of $f$ on the third vertex of $W_{n-1}$ that does not lie on $\partial W_{k(n)}$ and $\partial W_{l(n)}$. Hence, for the absolute value of the sum of the second and fourth differences we obtain the upper bound $\mathcal O(W_{n-1})\leq \Delta_{m-1}$. Putting these altogether, we have
\begin{align*}
|c_n-c_{n-1}|\leq \frac{1}{3} ( \Delta_{m-1}/3+\Delta_{m-1}/3 +\Delta_{m-1}) = \frac{5}{9}\Delta_{m-1}.
\end{align*}

We now do the same analysis for the difference $c_n-c_{n,k(n)}$. We have
\begin{align}
\begin{aligned}\label{Case 3 differences}
c_n-c_{n,k(n)} &=\frac{1}{3}( c_{n-1}-c_{n,k(n)}+ c_{n,l(n)}-c_{n,k(n)})\\
&=\frac{1}{3}[ (c_{n-1}-c_{n-1,k(n)})+(c_{n-1,k(n)} -c_{n,k(n)})\\
&\quad \quad + (c_{n,l(n)}-c_{k(n),l(n)})+(c_{k(n),l(n)}-c_{n,k(n)})].
\end{aligned}
\end{align}
We bound the first difference by $\Delta_{m-1}$, using the induction assumption. The second and fourth differences have opposite sign since  the vertices $\partial W_{n-1}\cap \partial W_{k(n)}$, $\partial W_{n}\cap \partial W_{k(n)}$,  $\partial W_{k(n)} \cap \partial W_{l(n)} $ are ordered points (see Figure \ref{fig:zoom}), contained in a half-edge of $W_{k(n)}$ (by Lemma \ref{Combinatorics lemma}), where $f$ is monotone increasing or decreasing by property \ref{B3 - monotone edges} in Section \ref{Section Building block}. Using the fundamental inequality
\begin{align*}
|x+y|\leq \max\{ |x|,|y|\}
\end{align*}
whenever $x,y\in \R$ and $xy<0$ (or more generally $x,y\in \C$ and their angle is $\pi$), we conclude that 
\begin{align*}
&|(c_{n-1,k(n)} -c_{n,k(n)})+ (c_{k(n),l(n)}-c_{n,k(n)})|\\
&\quad \quad \leq \max\{|c_{n-1,k(n)} -c_{n,k(n)}|,|c_{k(n),l(n)}-c_{n,k(n)}|\} \leq \mathcal O(W_{k(n)})\leq \Delta_{m-1}.
\end{align*}
For the third  difference in \eqref{Case 3 differences}, by Lemma \ref{Combinatorics lemma}(III) we have that $W_{k(n)}$ has a vertex on $\partial W_{l(n)}$. Hence, $W_{k(n)}$ and $W_{n}$ are adjacent triangles having a vertex on the strictly larger triangle $\partial W_{l(n)}$. Property ($\ast$) now implies that
\begin{align*}
|c_{n,l(n)}-c_{k(n),l(n)}|\leq \mathcal{O}(W_{l(n)})/3 \leq \Delta_{m-1}/3,
\end{align*}
by the induction assumption. Summarizing,
\begin{align*}
|c_n-c_{n,k(n)}|\leq \frac{7}{9}\Delta_{m-1}.
\end{align*}

Finally, we look at the difference $c_n-c_{n,l(n)}$. As before, we have
\begin{align*}
c_n-c_{n,l(n)}&=\frac{1}{3}(c_{n-1}-c_{n,l(n)}+c_{n,k(n)}-c_{n,l(n)})\\
&=\frac{1}{3}[(c_{n-1}-c_{n-1,l(n)})+(c_{n-1,l(n)}-c_{n,l(n)})  \\
&\quad \quad +(c_{n,k(n)}-c_{k(n),l(n)})+(c_{k(n),l(n)}-c_{n,l(n)})].
\end{align*}
Exactly as in the previous computation, the first difference is bounded by $\Delta_{m-1}$, the sum of the second and fourth differences is bounded by $\Delta_{m-1}$ using property \ref{B3 - monotone edges}, and the third difference is bounded by $\Delta_{m-1}/3$. Hence,
\begin{align*}
|c_n-c_{n,l(n)}|\leq \frac{7}{9}\Delta_{m-1}.
\end{align*}

Summarizing, we have 
\begin{align*}
\mathcal O(W_n) \leq \frac{7}{9}\Delta_{m-1}\eqqcolon \Delta_m
\end{align*}
for all $n> N_m$, and the proof is completed.
\end{proof}

\subsection{Generalization to homeomorphic gaskets}\label{Section Generalization}
Here we show that any image of the gasket under a homeomorphism of $\R^2$ is non-removable for $W^{1,2}$ as claimed in Theorem \ref{Intro homeo gasket non-removable}. We have to split in two cases, depending on whether the ``homeomorphic gasket" has area zero or positive area. In the first case, our proof for the standard gasket applies with some modifications, while the second case can be treated with the general statement that sets of positive measure are non-removable for Sobolev spaces; see Theorem \ref{Theorem All W1p}.

\begin{theorem}\label{Theorem Generalization}
Let $h\colon  \R^2\to \R^2$ be a homeomorphism, and $K$ be the Sierpi\'nski gasket. If $m_2(h(K))=0$, then $h(K)$ is non-removable for $W^{1,2}$.
\end{theorem}

\begin{proof}
Our goal is to obtain a continuous function $f\colon \R^2\to \R$, $f\in W^{1,2}(\R^2\setminus h(K))$, with $0\leq f\leq 1$, $f\equiv 0$ outside a ball $B(x_0,R_0)$, $f\equiv 1$ in a ball $B(x_0,r_0)$, such that $\|\nabla f\|_{L^2(\R^2)} =\|\nabla f\|_{L^2(\R^2\setminus h(K))}$ is as small as we wish; here it is crucial that $m_2(h(K))=0$. Then, arguing as in Section \ref{Section Avoidance}, one can show that $f\notin W^{1,2}(\R^2)$, if $\|\nabla f\|_{L^2(\R^2)}$ is sufficiently small.

Our proof of Theorem \ref{Theorem Non-removability} and more specifically the construction of $f$ was combinatorial, except for the construction of the particular building block functions that satisfy \ref{B1 - g=a outside balls}--\ref{B6} and ($\ast$). We remark that even in Cases 1 and 2 of the proof of Lemma \ref{Basic lemma} (which is the heart of the proof of continuity) we only used properties \ref{B1 - g=a outside balls}--\ref{B6} of the building block functions, together with the continuity of the restriction of $f$ on each particular $w$-triangle, but we did not need any specific modulus of continuity. The property $(\ast)$ was only needed in Case 3 and requires some special care.

Since the combinatorics of the gasket are preserved under homeomorphisms of $\R^2$, it remains to show that building block functions satisfying \ref{B1 - g=a outside balls}--\ref{B6} and ($\ast$) exist when the domain is an arbitrary Jordan region $\Omega$, rather than a triangle. Assume that $W$ is a $w$-triangle of $K$ such that $h(W)=\Omega$. It suffices to do the construction near each vertex of $\Omega$, i.e.,  near the images of the vertices of $W$. Let $z$ be a vertex of $W$, and consider the two half-edges $I,J\subset \partial W$ that meet at $z$. Also, consider the dyadic points $x_k\in I$, $k\in \N$, and $y_k\in J$, $k\in \N$, converging to $z$, as in the proof of ($\ast$) in Case 3; for example, if $I=[0,1]$, then $x_k=2^{-k}$ for $k\in \N$.  Let $z'=h(z)$, $I'=h(I)$, $J'=h(J)$, $x_k'=h(x_k)$, and $y_k'=h(y_k)$. 

Using a conformal map, we map $\Omega$ onto the upper half plane $\UHP$, and assume that $z'=0$ and that $I'\subset [0,\infty)$, $J'\subset (-\infty,0]$ are closed intervals meeting at $0$. Furthermore, $x_k' \in I'$ is a strictly decreasing sequence converging to $0$, and $y_k' \in J'$ is strictly increasing and converging to $0$. 

We wish to construct a continuous function $g\colon \br \UHP \to \R$ with the following properties:
\begin{enumerate}[$(B1')$]
\item g is supported in an arbitrarily small neighborhood of $0$,
\item $g(0)=1$,
\item $g$ is monotone increasing or decreasing on each of $I'$ and $J'$
\item $\int_\UHP |\nabla g|^2$ is arbitrarily small,
\item $g$ has the value $0$ at the endpoints of $I'$ and $J'$ that are distinct from $0$,
\item $g$ is monotone, in the sense that $\osc_{\UHP} (g)= \osc_{\partial \UHP} (g)$, and
\item[($\ast'$)]  $|g(x_k')-g(x_{k+1}')| \leq 1/3$ and $|g(y_k')-g(y_{k+1}')|\leq 1/3$ for all $k\in \N$.
\end{enumerate}
Since a conformal map from $\UHP$ onto $\Omega$ extends to a homeomorphism (using the spherical metric) from $\br \UHP \cup \{\infty\}$ onto $\br\Omega$, and also it does not change the Dirichlet energy $\int |\nabla g|^2$, all these properties can be transferred to $\Omega$, and yield a function with the corresponding properties. The proof of the analog of ($\ast$) then follows, as in the proof of Case 3 of the previous section; see Section \ref{Section Condition *}.

The construction of $g$ is very similar to the construction we did in Case 3 of the previous section. We fix a small $R_1>0$ and define $r_1=R_1/2$. In the annulus $A_1\coloneqq A(0;r_1,R_1)$ we define $g$ to be a radial function that is equal to $0$ in the outer circle, and increases to $1/N$ in the inner circle, where $N\in \N$ is fixed. The slope of $g$ is $\simeq \frac{1}{N(R_1-r_1)}$ in $A_1$. Then we define $R_2<r_1$ to be so small that the ``transition" annulus $A(0;R_2,r_1)$ contains a point $x_k'\in I'$ and a point $y_l'\in J'$. Here we set $g$ to be constant, equal to $1/N$. Then we define $r_2=R_2/2$ and $A_2\coloneqq A(0;r_2,R_2)$, and we set $g$ to be a radial function that increases from $1/N$ in the outer circle to $2/N$ in the inner circle. By construction, no interval $[x_{m+1}',x_{m}']$ or $[y_m',y_{m+1}']$ can intersect both annuli $A_1$ and $A_2$; this is because the sequences $x_m'$ and $y_m'$ are strictly monotone. The use of the transition annulus is crucial as we will see. We proceed in the same way, until the last annulus $A_N=A(0;r_N,R_N)$, where the function $g$ increases from $(N-1)/N$ in the outer circle to $1$ in the inner circle, with slope $\simeq \frac{1}{N(R_N-r_N)}$. Finally, we set $g \equiv 0$ outside $B(0,R_1)$, $g\equiv 1$ inside $B(0,r_N)$, and then restrict $g$ to $\br \UHP$.

By construction, each of the intervals $[x_{k+1}',x_{k}'],[y_k',y_{k+1}']$, $k\in \N$,  intersects at most one annulus $A_m$, $m\in \{1,\dots,N\}$, where the function $g$ increases by $1/N$. In the transition annuli of the form $A(0;R_{m},r_{m-1})$ the function $g$ is constant. Hence, we have
\begin{align*}
|g(x_k')-g(x_{k+1}')| \leq \frac{1}{N} \quad \textrm{and}\quad |g(y_k')-g(y_{k+1}')| \leq \frac{1}{N}
\end{align*}
for all $k\in \N$. In particular, these are less than $1/3$ if $N$ is sufficiently large.

Regarding the Dirichlet energy, we compute:
\begin{align*}
\int |\nabla g|^2 &=\sum_{m=1}^N \int_{\UHP\cap A_m} |\nabla g|^2 \lesssim \sum_{m=1}^N \frac{1}{N^2(R_m-r_m)^2} m_2(\UHP\cap A_m)\\
& \simeq \sum_{m=1}^N \frac{1}{N^2 (R_m-r_m)^2} (R_m^2-r_m^2) \simeq \sum_{m=1}^N \frac{1}{N^2 r_m^2}r_m^2 \simeq \frac{1}{N}, 
\end{align*}
where we used the fact that $R_m=2r_m$. If $N$ is sufficiently large, then $\int |\nabla g|^2$ can be as small as we wish, completing the proof.
\end{proof}

The case that $h(K)$ has positive Lebesgue measure has to be treated separately, and, in fact, the following more general statement holds in $\R^n$:

\begin{theorem}[Theorem \ref{Intro Theorem Positive}]\label{Theorem All W1p}
Let $K\subset \R^n$ be a compact set of positive Lebesgue measure and $1\leq p<\infty$. Then $K$ is non-removable for $W^{1,p}$.
\end{theorem}

\begin{proof}
We may assume that $\inter(K)=\emptyset$, otherwise $K$ is trivially non-removable for $W^{1,p}$, $1\leq p\leq \infty$, since one can simply consider a continuous function with no partial derivatives, supported on $\inter(K)$.

Define $\Omega \coloneqq \R^n\setminus K$. Let $x_0\in K$ be a Lebesgue point, i.e.,
\begin{align*}
\frac{m_n(B(x_0,r)\cap \Omega)}{m_n(B(x_0,r))} \to 0
\end{align*}  
as $r\to 0$. Hence, for each $i\in \N$ there  exists  an arbitrarily small $r_i>0$ such that
\begin{align}\label{Theorem Lebesgue Point}
m_n(B(x_0,r_i) \cap \Omega) \leq 2^{-ip} r_i^n.
\end{align}
Without loss of generality, we assume that $x_0=0$ and we set $B_i=B(x_0,r_i)$. We can also assume that the sequence $\{r_i\}_{i\in \N}$ satisfies 
\begin{align}\label{Theorem All W1p - r_i}
r_{i+1}<r_i/2< 1/2
\end{align}
for $i\in \N$.

Let $\phi\colon \R\to \R$ be the $1$-periodic extension of $|t|\x_{[-1/2,1/2]}(t)$. Let $c_i$ be a sequence of positive numbers and $m_i$ be a sequence of positive integers, to be determined. We define 
$$\phi_i(x)= c_i \phi\left( m_i \left( \frac{2|x|}{r_i}-1 \right) \right) \cdot \x_{[ r_i/2,r_i]} (|x|)$$  
for $x\in \R^n$ and $i\in \N$. Roughly speaking, we changed the amplitude and frequency of $x\mapsto \phi(|x|)$, and also translated its support to the annulus $A(0;r_i/2,r_i)$. Observe that $\phi_i\in W^{1,p}(\R^n)$ for all $1\leq p<\infty$, $\phi_i$ is continuous in $\R^n$, and $|\phi_i|\leq c_i$. Furthermore, for a.e.\ $x$ in the annulus $A(0;r_i/2,r_i)$ we have
\begin{align*}
|\nabla \phi_i(x)|\simeq \frac{c_im_i}{r_i},
\end{align*}
with uniform constants. Hence,
\begin{align*}
\|\nabla \phi_i \|_{L^p(\R^n)} \simeq c_im_i r_i^{n/p-1}
\end{align*}
with constants depending only on the dimension $n$, and 
\begin{align*}
\|\nabla \phi_i \|_{L^p(\Omega)} \lesssim \frac{c_im_i}{r_i}2^{-i} r_i^{n/p} \simeq 2^{-i}c_im_i r_i^{n/p-1},  
\end{align*}
by \eqref{Theorem Lebesgue Point}.

We define 
\begin{align*}
f= \sum_{i=1}^\infty \phi_i
\end{align*}
and note that $f$ is pointwise defined with $f(0)=0$, since $\phi_i$ have disjoint supports by \eqref{Theorem All W1p - r_i}. Observe that if $c_i\to 0$, then the series $\sum_{i=1}^\infty \phi_i$ converges uniformly to a continuous function. 

We have
\begin{align*}
\|\nabla f\|_{L^p(\Omega)} \leq  \sum_{i=1}^\infty \|\nabla \phi_i\|_{L^p(\Omega)} \lesssim \sum_{i=1}^\infty  2^{-i}c_im_i r_i^{n/p-1},
\end{align*}
so we wish to have that the latter series converges. If this is the case, then we will have indeed $f\in W^{1,p}(\Omega)$ by the completeness of the space.

If $f\in W^{1,p}(\R^n)$, then 
\begin{align*}
\|\nabla f\|_{L^p(\R^n)}^p = \int  \biggl|\sum_{i=1}^\infty \nabla \phi_i \biggr|^p= \sum_{i=1}^\infty \|\nabla \phi_i\|_{L^p(\R^n)}^p \simeq \sum_{i=1}^\infty  c_i^pm_i^p r_i^{n-p}
\end{align*}
because the functions $\phi_i$ have disjoint support. We wish the latter to be a divergent series, so that $f\notin W^{1,p}(\R^n)$.

Summarizing, we have to choose $c_i,m_i$ such that $c_i\to 0$, 
\begin{align*}
\sum_{i=1}^\infty  2^{-i}c_im_i r_i^{n/p-1} <\infty,\quad \textrm{and} \quad \sum_{i=1}^\infty  c_i^p m_i^p r_i^{n-p}=\infty.
\end{align*}
If $p\geq  n$, then we can choose $c_i= r_i^{1-n/p}\cdot i^{-1/p}$ and $m_i=1$ for all $i\in \N$. If $1\leq p<n$, then we choose $c_i=i^{-1/p}$ and $m_i$ to be the smallest integer such that $m_ir_i^{n/p-1} \geq 1$. Then $(m_i-1)r_i^{n/p-1} <1$, so $m_ir_i^{n/p-1}\leq 2$. 
\end{proof}

However, the conclusion fails for $W^{1,\infty}$:

\begin{prop}\label{W-1-infinity positive}There exists a compact set $K\subset \R^n$ of positive Lebesgue measure that is $W^{1,\infty}$-removable.
\end{prop}
\begin{proof}
Let $C\subset \R$ be a Cantor set of positive Lebesgue measure, and define $K\coloneqq C^n$, so $m_n(K)>0$. We claim that $K$ is $W^{1,\infty}$-removable. Let $f$ be a continuous function on $\R^n$ that lies  in  $W^{1,\infty}(\R^n\setminus K)$. We wish to show that $f$ is $M'$-Lipschitz, where $M'>0$ depends on $M=\| f\|_{W^{1,\infty}(\R^n\setminus K)}$. We fix a coordinate direction, say $e_1$, and a line $L$ parallel to $e_1$. The function $f$ is $M$-Lipschitz on each component of $L\setminus K$. On the other hand, by perturbing the line $L$ we may obtain a line $L'$ arbitrarily close and parallel to $L$ such that $L'\cap K=\emptyset$. Hence, $f$ is $M$-Lipschitz on $L'$, and by continuity it is also $M$-Lipschitz on $L$. If $x,y\in \R^n$ are arbitrary points, then one can connect them with a polygonal path $\gamma$, each of whose segments is parallel to a coordinate direction such that the length of $\gamma$ is comparable to $|x-y|$. The conclusion follows by using the Lipschitz bound on each of the segments of $\gamma$.
\end{proof}
\begin{remark}
In fact, the complement of the Cantor set $K$ is a \textit{quasiconvex set} in $\R^n$, i.e., there exists a constant $M>0$ such that for any two points $x,y\in \R^n\setminus K$ there exists a rectifiable path $\gamma \subset \R^n\setminus K$ that connects $x$ and $y$, with
\begin{align*}
\length(\gamma) \leq M|x-y|.
\end{align*} 
The argument in the proof of Proposition \ref{W-1-infinity positive} can be modified to show that if the complement of compact set $K$ with empty interior is quasiconvex, then $K$ is $W^{1,\infty}$-removable.
\end{remark}

\section{Quasiconformal non-removability}\label{Section QC non}

We quickly sketch the strategy of constructing a homeomorphism $F\colon \R^2\to \R^2$ that is quasiconformal on $\R^2\setminus K$, but not globally quasiconformal; see also Section \ref{Section Sketch}.

First we will define a continuous map $f\colon \R^2\to \R^2$ that is the identity on the unbounded complementary component of the gasket $K$, but collapses each $w$-triangle to a tripod; see Figure \ref{fig:FoldGasket}. This is done in Section \ref{Section QC Collapsing}. 

Of course, this map is not a homeomorphism on the $w$-triangles, so we have to correct it. We do that by ``folding" each $w$-triangle on top of each tripod; see Figures \ref{fig:FoldTriangle} and \ref{fig:folding}. The folding map will be $M$-quasiconformal, in the sense of Definition \ref{Pre Definition QC metric}, restricted on each $w$-triangle. The folding of a single equilateral triangle on top of a tripod is explained in Section  \ref{Section QC Folding}. Moreover, the folding has to be compatible, in a sense, with $f$ on the boundary of each $w$-triangle.

If the heights of the rectangles attached to each tripod are chosen to be sufficiently small, then we will obtain a homeomorphism $\Phi$ from $\R^2$ onto a limiting flap-plane $S$ (Figure \ref{fig:FoldGasket}), which is constructed out of infinitely many tripods, and thus falls into the setting of Proposition \ref{General-Completeness}. The map $\Phi$ is the result of patching together the map $f$ outside the $w$-triangles with the folding map of each $w$-triangle. The construction of the map $\Phi$ and of the flap-plane $S$ is discussed in Section \ref{Section QC homeo}.

Finally, if one chooses the heights of the rectangles to be even smaller, then by Theorem \ref{General-QS Embedding} one obtains a quasisymmetric embedding $\Psi$ of $S$ onto $\R^2$. The composition $F= \Psi\circ \Phi$ will be a homeomorphism of $\R^2$ that is $M'$-quasiconformal on each $w$-triangle for some uniform $M'>0$, but it cannot be globally quasiconformal, because it has to blow the gasket $K$ to a set of positive area. Section \ref{Section QC Finish} contains these details that finish the proof of non-removability.

\begin{figure}
\centering
\input{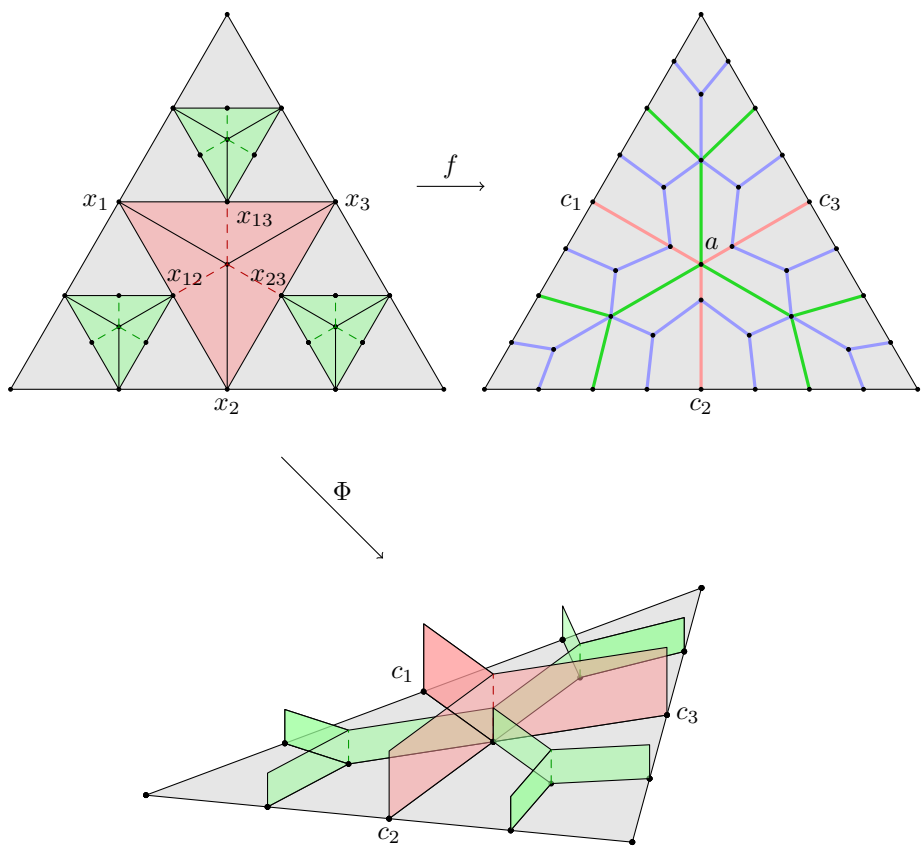}
\caption{Illustration of the collapsing map $f\colon \R^2\to \R^2$  and of the homeomorphism $\Phi$ from $\R^2$ onto the flap-plane $S$. For each $i=1,2,3$ the collapsing map $f$ sends the vertex $x_i$ of the central pink triangle to the vertex $c_i$ of the central pink tripod. For $i<j$ the midpoint $x_{ij}$ of $x_i$ is mapped to the barycenter $a$ of the central tripod. The map $\Phi$ is a homeomorphism so in particular the green rectangles are \textit{not} glued to the red rectangles, except at the three points $\Phi(x_{ij})$, $i<j$; these are the points that correspond to the barycenter $a$ of the pink tripod.}\label{fig:FoldGasket}
\end{figure}

\subsection{Collapsing of \texorpdfstring{$w$}{w}-triangles}\label{Section QC Collapsing}
Recall the definitions of $w$-triangles and $v$-tri\-angles from Section \ref{Section Definitions}. In this section we define a continuous map $f\colon \R^2\to \R^2$ that is equal to the identity in the unbounded $w$-triangle, and collapses each bounded $w$-triangle to a tripod $G$.  Further properties of $f$ will be that it is injective on $K^\circ$, and it maps the latter to a set of positive measure. Interestingly, the construction of such a map is rather a modification of the construction of the  continuous function that we constructed in Section \ref{Section Continuous}.

\subsubsection{Building block}
Recall from Section \ref{Section Flap-planes} that a tripod $G$ is by definition the union of three line segments in the plane, which have a common endpoint, but otherwise they are disjoint; note that their length need not be the same. We call the common endpoint the central vertex of the tripod $G$. Every triple $c_1,c_2,c_3\in \R^2$ of non-collinear points defines a \textit{canonical} tripod, whose central vertex is the barycenter of $c_1,c_2,c_3$, i.e., it is 
\begin{align*}
a= \frac{1}{3}(c_1+c_2+c_3).
\end{align*}
In what follows, we will only be using canonical tripods, even if we do not mention it explicitly.

We consider an analog of the building block function discussed in Section \ref{Section Building block}. Let $W\subset \R^2$ an open equilateral triangle with vertices $x_1,x_2,x_3$. Then for each triple of non-collinear points $c_1,c_2,c_3 \in \R^2$ and for the canonical tripod $G$ corresponding to these points there exists a continuous map $g \colon \br W\to G$ such that
\begin{enumerate}
\item[\mylabel{BB2}{$(\widetilde B2)$}] $g(x_i)=c_i$ for $i=1,2,3$,
\item[\mylabel{BB3}{$(\widetilde B3)$}] $g$ is injective and \textit{monotone} (see comments below), on each half-edge of $\partial W$, from its midpoint to a vertex,
\item[\mylabel{BB5}{$(\widetilde B5)$}]$g$ maps the midpoints of the edges of $W$ to the central vertex $a$ of $G$, and
\item[\mylabel{BB6}{$(\widetilde B6)$}]$g$ is monotone in the sense that $\osc_{\br W}(g)= \osc_{\partial W} (g)$. Here, $\osc_{\br W}(g)= \sup \{|g(x)-g(y)|:x,y\in \br W\}$.
\end{enumerate}
These properties should be compared to the properties of the building block function in Section \ref{Section Building block}. Note that \ref{BB6} follows immediately from continuity and \ref{BB2}, since they imply that $g(\partial W)=G$. Also, \ref{BB2}, \ref{BB5}, and the injectivity from \ref{BB3} imply that $g$ maps each half-edge of $W$ homeomorphically onto an edge of the tripod $G$. In particular, \ref{B3 - monotone edges} from Section \ref{Section Building block} holds here, in the sense that if $I\subset \partial W$ is a half-edge of $W$ and $J_1\subset J_2\subset I$ are segments, then $\osc_{J_1}(g) \leq \osc_{J_2}(g)$. This explains the use of the word \textit{monotone} in the statement of \ref{BB3}.

From now on, a \textit{building block map} will be a map $g$ as above, and we will say that its \textit{parameters} are $c_1,c_2,c_3$. At this moment we are not interested in the definition of the map $g$ in the interior of the triangle $W$ (which could be anything as long as $g(\br W)=G$ and $g$ is continuous), but we only focus on its boundary. The construction of such a continuous map $g$ is elementary. For example, one can first collapse $\br W$ to the canonical tripod defined by its vertices $y_1,y_2,y_3$, so that the midpoints of the edges are mapped to the barycenter of the triangle $W$, and so that the map is injective on each edge of $\partial W$. Then one can use an affine map to map this tripod to the canonical tripod $G$ defined by $c_1,c_2,c_3$ such that the vertices $y_1,y_2,y_3$ are mapped to $c_1,c_2,c_3$, respectively.  Note that affine maps preserve barycenters.

\subsubsection{Inductive definition}
A fundamental lemma that we will use is the following. A \textit{convex quadrilateral} is the open region in the plane that is bounded by a polygon with four sides and (interior) angles strictly less than $\pi$ (we wish to exclude degenerate cases).
\begin{lemma}\label{QC collapsing convex}
Let $U\subset \R^2$ be  convex quadrilateral and consider points $c_1,c_2,c_3 \in \partial U$ such that $c_1$ is a vertex of $U$, and $c_2,c_3$ lie on the interior of distinct sides of $U$ that are not congruent to the vertex $c_1$. Also, consider  the canonical tripod $G$ corresponding to $c_1,c_3,c_3$, which are necessarily non-collinear points. Then $G\subset \br U$ and each component $Z$ of $U\setminus G$ is a convex quadrilateral. Furthermore, two of the sides of such a component $Z$ are two of the edges of the tripod $G$ and  are congruent to its central vertex, while the other two sides of $Z$ are contained in distinct sides of $U$.  
\end{lemma}

The proof is elementary and is omitted. Now, we define the desired map $f$ inductively, in a very similar way, as the map we defined in Section \ref{Section Choice}. We define $f$ to be the identity in the closure of the $w$-triangle of level $0$ (i.e., the closure of the unbounded component of $\R^2\setminus K$). 

Note that the map $f$ is already defined on the vertices of the central $w$-triangle $W_1$ of level $1$. We define $f$ on $\br W_1$ to be a building block map that collapses this triangle to a tripod $G_1$. This tripod is contained in the $v$-triangle $V_1$ of level $0$, which is convex. Each component $U_2$ of $V_1\setminus G_1$ is a convex quadrilateral. We will be calling $U_2$ a $u$-quadrilateral. Note that each $v$-triangle $V_2$ of level $1$ (i.e., $V_2$ is a component of $V_1\setminus \br W_1$) corresponds to a $u$-quadrilateral $U_2$, and in fact $f$ maps $\partial V_2$ to $\partial U_2$ homeomorphically, in an orientation-preserving way, by property \ref{BB3}. Moreover, the midpoint $x_1$ of an edge of $V_2$ is mapped to a vertex $c_1$ of $U_2$ and the other two edges of $V_2$ are mapped to the other edges of $U_2$ that are not congruent to $c_1$.

We claim that we can define $f$ on all $w$-triangles, so that each $v$-triangle corresponds to a $u$-quadrilateral as above. We now formulate and prove the inductive step.

Let $V$ be a $v$-triangle and suppose that $f\big|_{\partial V}$ has been defined and maps $\partial V$ homeomorphically onto the boundary of a convex quadrilateral $U$. Moreover, suppose that $f$ maps the midpoint $x_1$ of an edge of $V$ to a vertex $c_1$ of $U$ and that each of the other two edges of $V$ is mapped to one of the other two sides of $U$ that are not congruent to $c_1$. 
    
Consider the central $w$-triangle $W\subset V$ that has its vertices $x_1,x_2,x_3$ on $\partial V$. By the assumptions on the map $f\big|_{\partial V}$, the points $c_2\coloneqq f(x_2)$ and $c_3\coloneqq f(x_3)$ lie on distinct sides of $U$ that are not congruent to $c_1=f(x_1)$. We define $f$ on $\br W$ to be a building block map that collapses the triangle $\br W$ to a tripod $G$ with vertices $c_1,c_2,c_3$.

By Lemma \ref{QC collapsing convex} we see that each component $U'$ of $U\setminus G$ is a convex quadrilateral. Two of the sides of $U'$ are edges of the tripod $G$, each of which is the homeomorphic image of a half-edge of $W$ under $f$, by property \ref{BB3} of the building block map on $\br W$. Suppose, for instance, that the edge of $W$ is the segment $[x_1,x_2]$. The other two sides of $U'$ are contained in two distinct sides of $U$. These two sides of $U'$ have to correspond under the homeomorphism $f\big|_{\partial V}$ to an arc of $\partial V$ that connects $x_1$ and $x_2$. Among the two such arcs, there is only one possibility, since it follows by  the assumptions that $f^{-1}\big|_{\partial U}$ maps each side of $U$ into one edge of $V$.  

We thus see that there exists a $v$-triangle $V' \subset V\setminus \br W$ such that $f$ maps $\partial V'$ homeomorphically onto $\partial U'$. The midpoint $x_1'=\frac{x_1+x_2}{2}$ of the edge $\partial V'\cap \br W$ of $V'$ is mapped to a point $c_1'$ that is a vertex of $U'$ (in fact it is the central vertex of $G$ as follows from Lemma \ref{QC collapsing convex}) and the other two edges of $V'$ are mapped to the other two sides of $U'$ by the mapping properties of $f$. 

This completes the proof of the inductive step and shows that $f$ can be defined on all of $W_\infty= \bigcup_{W\in \mathcal W} \br W$.

\subsubsection{Properties of $f$}\label{Section Properties of f}
(a) Each $w$-triangle $\br W$ is mapped to a tripod $G$ with vertices $c_1,c_3,c_3$ and central vertex $a=(c_1+c_2+c_3)/3$. Following the notation of Section \ref{Section Continuous}, we define
\begin{align*}
\mathcal O(W)&= \max_{i=1,2,3} |a-c_i|.
\end{align*}
Note that $\mathcal O(W)$ is the length of the largest edge of the tripod $G$. The map $f$ is \textit{monotone} in the following sense. For each $v$-triangle $V$ of level $m\geq 1$ we have
\begin{align*}
\osc_{\br V \cap W_\infty}(f)= \sup_{x,y\in \br V\cap W_\infty} |f(x)-f(y)|= \osc_{\partial V}(f).
\end{align*}
This is the analog of Lemma \ref{Monotonicity}. The reason it holds in our case is that, by its inductive definition, $f$ maps $\partial V$ homeomorphically onto a quadrilateral $\partial U$, and all $w$-triangles contained in $ V$ are mapped to tripods contained in $\br U$.

We define the \textit{level} of a $u$-quadrilateral $U$, which corresponds as above to a $v$-triangle $V$, to be an integer equal to the level of $V$; see also Section \ref{Section Definitions}. More precisely, if the sidelength of $V$ is $2^{-n}$, $n\in \N$, then the level of $V$ and $U$ is $n$.

(b) The map $f$ has a continuous extension to $\R^2$. The proof is  the same as the proof of Proposition \ref{Continuity}, with minor modifications. One observes that we have given essentially the same definition for $f$ and it satisfies the properties \ref{BB2}, \ref{BB3}, \ref{BB5}, and \ref{BB6}, as in Section \ref{Section Building block}; see also Remark \ref{Building block Continuity remark}. Namely, what we called ``height" of $f$ on a $w$-triangle $W$ in Section \ref{Section Building block} is now the barycenter $a=(c_1+c_2+c_3)/3$ of the tripod $G$. Hence, the ``height" of $f$ here is the average of the values of $f$ on the vertices of $W$, precisely as it was the case in the definition of $f$ in Section \ref{Section Choice}. The only  difference is that $f$ is now complex-valued, instead of real-valued, but this does not affect the proofs. The proof of continuity goes through, if we ensure that $f$ has on each $w$-triangle a certain modulus of continuity, as described in Case 3 of the proof of the basic Lemma \ref{Basic lemma} (see Section \ref{Section Condition *}):
\begin{enumerate}
\item[($\ast$)] Assume that the triangles $W_1,W_2$ are adjacent and each has a vertex $z_1,z_2$, respectively, lying on a triangle $\partial W_0$ of a strictly lower level. Then 
\begin{align*}
|f(z_1)-f(z_2)|\leq \mathcal O(W_0)/3.
\end{align*}
\end{enumerate}
As already discussed in Section \ref{Section Condition *}, this boils down to looking at the dyadic points $\{x_k\}_{k\in \N}$ contained in each half-edge $I$ of $\partial W_0$ and accumulating in the corresponding vertex of $W_0$, and requiring that 
\begin{align*}
\mylabel{BB*}{\tag{\textrm{Dyadic  $\ast$}}} |f(x_k)-f(x_{k+1})| \leq \mathcal O(W_0)/3,
\end{align*}
for all $k\in \N$, where $x_{k}$ and $x_{k+1}$ are ``consecutive" dyadic points (for example, if $I=[0,1]$ and $0$ is a vertex of $W_0$, then $x_k=2^{-k}$ for $k\in \N$). This will be ensured in the next section, where we construct more carefully the map $f$ on each $w$-triangle, so that it still has properties \ref{BB2}, \ref{BB3}, \ref{BB5}, and \ref{BB6}, and also has this particular modulus of continuity.

(c) By continuity, the map $f\colon \R^2\to \R^2$ is surjective. The $u$-quadrilaterals  induce subdivisions in the image side, exactly as the $v$-triangles do in the domain. Let $U_1$ be the equilateral triangle of sidelength $1$ that is the image of the $v$-triangle $V_1$ of level $0$ (in fact, $U_1=V_1$ since $f$ is the identity on $\partial V_1$). Then $U_1$ contains three disjoint $u$-quadrilaterals of level $1$, which lie in the complement of a ``removed" tripod. Each of these quadrilaterals is the image of a $v$-triangle of level $1$. This follows by the continuity and the monotonicity of $f$. In general, the closure of each $u$-quadrilateral of level $m$ is the union of the closures of three disjoint $u$-quadrilaterals of level $m+1$. These $u$-quadrilaterals are the images of $v$-triangles of level $m+1$; see Figure \ref{fig:FoldGasket}. By continuity, the diameters of $u$-quadrilaterals of level $m$ converge to $0$ as $m\to\infty$.

For each point $z\in \br U_1$ we can find a sequence $U_n$, $n\in \N$, of nested $u$-quadri\-laterals such that 
\begin{align*}
\{z\} = \bigcap_{n=1}^\infty \br U_n.
\end{align*}
We set $L^\circ$ to be the set of points of $\br U_1$ that do not lie on any ``removed" tripod. The correspondence between $U_n$ and $V_n$ and the uniqueness of a sequence $V_n$ shrinking to a point $x\in K^\circ$ (see the comments before Lemma \ref{Combinatorics V lemma}) imply that $f$ maps $K^\circ $ onto $L^\circ$, and in fact $f$ is injective on $K^\circ$.

(d) The tripods contained in $\br U_1$, which are the images of the $w$-triangles, have $\sigma$-finite length. Thus, the area of the tripods is equal $0$. This implies that the image of $K^\circ$ has full measure inside $\br U_1$.

\subsection{Folding equilateral triangles to tripods}\label{Section QC Folding}
So far, we have a continuous surjective map $f\colon \R^2\to \R^2$ that is injective outside the union $W_\infty$ of the closures of the $w$-triangles; of course continuity is still subject to choosing suitably the building block maps so that they have a certain modulus of continuity. We wish to change the definition of the map $f$ only inside each $w$-triangle so that $f$ becomes injective everywhere. However, this is not possible if the target is $\R^2$. Thus, we change the target to a flap-plane $S$ by attaching rectangles to each of the tripods; see Section \ref{Section Flap-planes} for the definition of a flap-plane. We then change the definition of $f$ inside the $w$-triangles and allow them to be mapped onto the rectangles attached to the corresponding tripod. We  wish to do this in such a way that the resulting map $\Phi\colon \R^2\to S$ is a homeomorphism, quasiconformal inside the $w$-triangles. This will be discussed in detail in the next section.

In this section we show how one can ``fold" an equilateral triangle $W\subset \R^2$ onto a space $X$ obtained by attaching rectangles $E$ to a single tripod $G$. The space $X$ is constructed very similarly to the flap-planes discussed in Section \ref{Section General}. One first cuts the plane along the edges of the tripod $G$, and then attaches two rectangles on each slit arising from an edge. The width of each rectangle is equal to the length of the corresponding edge of $G$, and the height is a prescribed constant $h>0$, which is the same for all $6$ rectangles. The barycenter $a$ of $G$ must ``lift" to three line segments, along which neighboring rectangles are glued; see Figure \ref{fig:Tripodflaps} for the gluing pattern. We remark that the lengths of the edges of $G$, and thus the widths of the rectangles, need not be equal to each other. The space $X$ is a topological disk, and we endow it with its natural length metric $d$, so that each rectangle $E\sim G$ (i.e., $E$ is glued to an edge of $G$) is isometric to a rectangle with the Euclidean metric.   

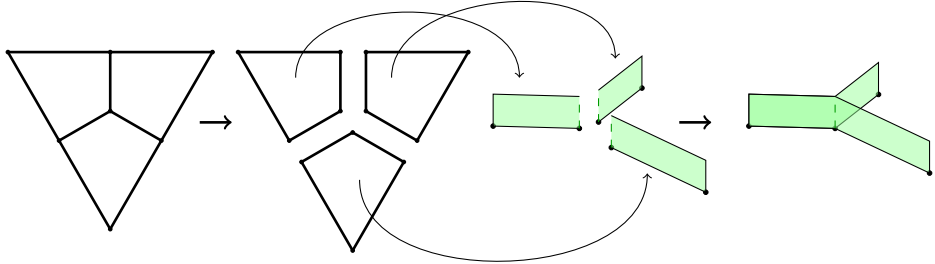
\begin{figure}
\centering
\begin{tikzpicture}[scale=.85]

%
%
%Flaps
%
%
\begin{scope}

\begin{scope}[xshift=8cm,yshift=2cm, rotate around y=-20]
%\draw [fill=black] (2.,0,0.) circle (1.0pt) node (C3){};
%\draw [fill=black] (4.,0,2.3094010767585043) circle (1.0pt) node (C2){};
\draw [fill=black] (1.,0,1.7320508075688776) circle (1.0pt) node (C1){};
\draw [fill=black] (2.333333333333333,0,1.3471506281091274) circle (1.0pt) node (C0){};
\node (C11) at ($(C1.center)+(0,.5,0)$){};
%\node (C22) at ($(C2.center)+(0,.5,0)$){};
%\node (C33) at ($(C3.center)+(0,.5,0)$){};
\node (C00) at ($(C0.center)+(0,.5,0)$){};
%\draw[fill=green!40, fill opacity=.5] (C1.center)--(C11.center)--(C00.center)--(C22.center)--(C2.center)--(C0.center)--cycle;
%\draw[fill=green!40, fill opacity=.5] (C1.center)--(C11.center)--(C00.center)--(C33.center)--(C3.center)--(C0.center)--cycle;
%\draw[dashed, color=green!60!black] (C0.center)--(C00.center);

\draw[fill=green!40, fill opacity=.5] (C0.center)--(C1.center)--(C11.center)--(C00.center);
\draw[dashed, color=green!60!black] (C0.center)--(C00.center);
\end{scope}

\begin{scope}[xshift=8.5cm, yshift=1.7cm, rotate around y=-20]
%\draw [fill=black] (2.,0,0.) circle (1.0pt) node (C3){};
\draw [fill=black] (5.,0,2.3094010767585043) circle (1.0pt) node (C2){};
%\draw [fill=black] (1.,0,1.7320508075688776) circle (1.0pt) node (C1){};
\draw [fill=black] (2.333333333333333,0,1.3471506281091274) circle (1.0pt) node (C0){};
%\node (C11) at ($(C1.center)+(0,.5,0)$){};
\node (C22) at ($(C2.center)+(0,.5,0)$){};
%\node (C33) at ($(C3.center)+(0,.5,0)$){};
\node (C00) at ($(C0.center)+(0,.5,0)$){};
%\draw[fill=green!40, fill opacity=.5] (C1.center)--(C11.center)--(C00.center)--(C22.center)--(C2.center)--(C0.center)--cycle;
%\draw[fill=green!40, fill opacity=.5] (C1.center)--(C11.center)--(C00.center)--(C33.center)--(C3.center)--(C0.center)--cycle;
%\draw[dashed, color=green!60!black] (C0.center)--(C00.center);

\draw[fill=green!40, fill opacity=.5] (C0.center)--(C2.center)--(C22.center)--(C00.center);
\draw[dashed, color=green!60!black] (C0.center)--(C00.center);
\end{scope}

\begin{scope}[xshift=8.3cm, yshift=2.1cm, rotate around y=-20]
\draw [fill=black] (2.,0,0.) circle (1.0pt) node (C3){};
%\draw [fill=black] (4.,0,2.3094010767585043) circle (1.0pt) node (C2){};
%\draw [fill=black] (1.,0,1.7320508075688776) circle (1.0pt) node (C1){};
\draw [fill=black] (2.333333333333333,0,1.3471506281091274) circle (1.0pt) node (C0){};
%\node (C11) at ($(C1.center)+(0,.5,0)$){};
%\node (C22) at ($(C2.center)+(0,.5,0)$){};
\node (C33) at ($(C3.center)+(0,.5,0)$){};
\node (C00) at ($(C0.center)+(0,.5,0)$){};
%\draw[fill=green!40, fill opacity=.5] (C1.center)--(C11.center)--(C00.center)--(C22.center)--(C2.center)--(C0.center)--cycle;
%\draw[fill=green!40, fill opacity=.5] (C1.center)--(C11.center)--(C00.center)--(C33.center)--(C3.center)--(C0.center)--cycle;
%\draw[dashed, color=green!60!black] (C0.center)--(C00.center);

\draw[fill=green!40, fill opacity=.5] (C0.center)--(C3.center)--(C33.center)--(C00.center);
\draw[dashed, color=green!60!black] (C0.center)--(C00.center);
\end{scope}

\begin{scope}[xshift=12cm, yshift=2cm, rotate around y=-20]
\draw [fill=black] (2.,0,0.) circle (1.0pt) node (C3){};
\draw [fill=black] (5.,0,2.3094010767585043) circle (1.0pt) node (C2){};
\draw [fill=black] (1.,0,1.7320508075688776) circle (1.0pt) node (C1){};
\draw [fill=black] (2.333333333333333,0,1.3471506281091274) circle (1.0pt) node (C0){};
\node (C11) at ($(C1.center)+(0,.5,0)$){};
\node (C22) at ($(C2.center)+(0,.5,0)$){};
\node (C33) at ($(C3.center)+(0,.5,0)$){};
\node (C00) at ($(C0.center)+(0,.5,0)$){};
\draw[fill=green!40, fill opacity=.5] (C1.center)--(C11.center)--(C00.center)--(C33.center)--(C3.center)--(C0.center)--cycle;
\draw[fill=green!40, fill opacity=.5] (C1.center)--(C11.center)--(C00.center)--(C22.center)--(C2.center)--(C0.center)--cycle;

\draw[dashed, color=green!60!black] (C0.center)--(C00.center);

%\draw[fill=green!40, fill opacity=.5] (C0.center)--(C3.center)--(C33.center)--(C00.center);
%\draw[dashed, color=green!60!black] (C0.center)--(C00.center);
\end{scope}

\end{scope}

%
%
%splitted triangle
%
%

\begin{scope}[scale=.8, xshift=0cm,line cap=round,line join=round,>=triangle 45,x=1.0cm,y=1.0cm]
\clip(-0.1,-1) rectangle (12.111625304043951,5.061947321901701);
\fill[line width=2.pt,fill=white] (0.,3.) -- (2.,-0.4641016151377553) -- (4.,3.) -- cycle;
\draw [line width=1.pt] (0.,3.)-- (2.,-0.4641016151377553);
\draw [line width=1.pt] (2.,-0.4641016151377553)-- (4.,3.);
\draw [line width=1.pt] (4.,3.)-- (0.,3.);
\draw [line width=1.pt] (2.,1.845299461620748)-- (2.,3.);
\draw [line width=1.pt] (2.,1.845299461620748)-- (1.,1.2679491924311224);
\draw [line width=1.pt] (2.,1.845299461620748)-- (3.,1.2679491924311224);
\draw [line width=1.pt] (4.5,3.)-- (6.5,3.);
\draw [line width=1.pt] (4.5,3.)-- (5.5,1.2679491924311224);
\draw [line width=1.pt] (6.5,3.)-- (6.5,1.8452994616207472);
\draw [line width=1.pt] (5.5,1.2679491924311224)-- (6.5,1.8452994616207472);
\draw [line width=1.pt] (7.,3.)-- (9.,3.);
\draw [line width=1.pt] (9.,3.)-- (8.,1.2679491924311228);
\draw [line width=1.pt] (7.,1.845299461620744)-- (8.,1.2679491924311228);
\draw [line width=1.pt] (7.,1.845299461620744)-- (7.,3.);
\draw [line width=1.pt] (5.740174552951677,0.8496170377519137)-- (6.74017455295168,1.4269673069415396);
\draw [line width=1.pt] (6.74017455295168,1.4269673069415396)-- (7.740174552951688,0.8496170377519157);
\draw [line width=1.pt] (5.740174552951677,0.8496170377519137)-- (6.740174552951674,-0.8824337698169641);
\draw [line width=1.pt] (6.740174552951674,-0.8824337698169641)-- (7.740174552951688,0.8496170377519157);
%\begin{scriptsize}
\draw [fill=black] (0.,3.) circle (1.0pt);
\draw [fill=black] (2.,-0.4641016151377553) circle (1.0pt);
\draw [fill=black] (4.,3.) circle (1.0pt);
\draw [fill=black] (2.,3.) circle (1.0pt);
\draw [fill=black] (3.,1.2679491924311224) circle (1.0pt);
\draw [fill=black] (1.,1.2679491924311224) circle (1.0pt);
\draw [fill=black] (2.,1.845299461620748) circle (1.0pt);
\draw [fill=black] (4.5,3.) circle (1.0pt);
\draw [fill=black] (6.5,3.) circle (1.0pt);
\draw [fill=black] (5.5,1.2679491924311224) circle (1.0pt);
\draw [fill=black] (6.5,1.8452994616207472) circle (1.0pt);
\draw [fill=black] (7.,3.) circle (1.0pt);
\draw [fill=black] (9.,3.) circle (1.0pt);
\draw [fill=black] (7.,1.845299461620744) circle (1.0pt);
\draw [fill=black] (8.,1.2679491924311228) circle (1.0pt);
\draw [fill=black] (5.740174552951677,0.8496170377519137) circle (1.0pt);
\draw [fill=black] (6.74017455295168,1.4269673069415396) circle (1.0pt);
\draw [fill=black] (7.740174552951688,0.8496170377519157) circle (1.0pt);
\draw [fill=black] (6.740174552951674,-0.8824337698169641) circle (1.0pt);
%\end{scriptsize}
\end{scope}

\draw[->] (4.5,2) to [out=90, in=90] (8,2);
\draw[->] (6,2) to[out=90, in=90] (9.5, 2.3);
\draw[->] (5.5,0.4) to [out=-90, in=-90] (10,0.5);
\draw[->, line width=1pt] (3,1.3) to (3.5,1.3);
\draw[->, line width=1pt] (10.5,1.3) to (11,1.3);

\end{tikzpicture}
\caption{An equlateral triangle $W$ is split into three quadrilaterals $Z_i$, $i=1,2,3$. Then each $Z_i$ is folded over an edge of the tripod $G$. Finally, the resulting rectangles are glued to obtain the space $X$.}\label{fig:FoldTriangle}
\end{figure}

We wish to construct a homeomorphism $\phi\colon \br W\to X$ that has the following properties:
\begin{enumerate}[(A)]
\item the composition of $\phi$ with the natural projection $X\to G$ satisfies \ref{BB2}, \ref{BB3},  \ref{BB5}, and \ref{BB6}, 
\item the same composition has the modulus of continuity in \eqref{BB*}, and
\item $\phi$ is $M$-quasiconformal (in the sense of Definition \ref{Pre Definition QC metric}) in the interior of the preimage of each rectangle $E\sim G$, where $M$ is independent of the tripod $G$ and of the height $h$ of the rectangles $E\sim G$,
\end{enumerate}
provided that $h$ is sufficiently small.

To do this, we first draw the heights of $W$, which split it into three quadrilaterals, each of which is a rotation of the other, and contains in its boundary a vertex of $W$ together with two congruent half-edges of $W$; see Figure \ref{fig:FoldTriangle}. Let $Z$ be one of these quadrilaterals. We will ``fold" $Z$ with a piecewise linear map to the two rectangles attached to an edge of $e$ of $G$. Assume that the length of $e$ is $\ell$. We assume that the height $h$ of the rectangles is less than $\ell/6$. Then there exists a unique $N\in \N$ such that 
\begin{align*}
\ell= Nh+q,
\end{align*}
where $h\leq q<2h$ is a remainder term.

We now divide each of the two rectangles $E$ attached to $e$ into $N$ squares of dimensions $h\times h$, and a rectangle of dimensions $q\times h$, as in Figure \ref{fig:folding}. We also subdivide $Z$ by considering $N+1$ dyadic points on each half-edge of $W$ that is contained in $\partial Z$, and drawing trapezoids as in Figure \ref{fig:folding}. Each of these $2N$ trapezoids is similar to the trapezoid $ABCD$, in the sense that it can be obtained by applying a Euclidean similarity to $ABCD$.

Each of these trapezoids can be mapped with a piecewise linear map to the corresponding square of dimensions $h\times h$. In fact, this can be done by drawing one diagonal in each trapezoid and in each  square, and  and then gluing two linear maps. This piecewise linear map is $M$-quasiconformal for a uniform $M>0$. 

We have to treat specially the two triangles near the vertex of $W$ that lies in $\partial Z$. The triangle $HIK$ is a right triangle with angle $\widehat{HIK}$ equal to $\pi/6$. Hence, it can be mapped to the rectangle $H'I'J'K'$ (so that vertices are mapped to the corresponding vertices) with a piecewise linear map that is $M$-quasiconformal for a universal $M>0$. Recall here that the ratio of the sides of the rectangle $H'I'J'K'$ is by construction bounded  between $1$ and $2$. The construction of such a map is done by converting the triangle $HIK$ to a quadrilateral, by introducing the midpoint $J$ of the segment $IK$. Then one can draw the diagonals $HJ$ and $H'J'$ and glue together two linear maps, one from the triangle $HJK$ to the triangle $H'J'K'$, and one from the triangle $HIJ$ to the triangle $H'I'J'$.

\begin{figure}
\centering
\begin{tikzpicture}[scale=2,line cap=round,line join=round,>=triangle 45,x=1.0cm,y=1.0cm]

%\draw [color=cqcqcq,, xstep=1.0cm,ystep=1.0cm] (-0.5350602830367487,-0.7124032774582475) grid (11.24602056103301,5.249848071453491);
%\draw[->,color=black] (-0.5350602830367487,0.) -- (11.24602056103301,0.);
%\foreach \x in {,1.,2.,3.,4.,5.,6.,7.,8.,9.,10.,11.}
%\draw[shift={(\x,0)},color=black] (0pt,2pt) -- (0pt,-2pt) node[below] {\footnotesize $\x$};
%\draw[->,color=black] (0.,-0.7124032774582475) -- (0.,5.249848071453491);
%\foreach \y in {,1.,2.,3.,4.,5.}
%\draw[shift={(0,\y)},color=black] (2pt,0pt) -- (-2pt,0pt) node[left] {\footnotesize $\y$};
%\draw[color=black] (0pt,-10pt) node[right] {\footnotesize $0$};
%\clip(-0.5350602830367487,-0.7124032774582475) rectangle (11.24602056103301,5.249848071453491);

\draw[line width=0.8pt] (5.,1.5) -- (6.,1.5) -- (6.,2.5) -- (5.,2.5) -- cycle;
\draw[line width=0.8pt] (6.,1.5) -- (7.,1.5) -- (7.,2.5) -- (6.,2.5) -- cycle;
\draw[line width=0.8pt] (5.,2.5) -- (6.,2.5) -- (6.,3.5) -- (5.,3.5) -- cycle;
\draw[line width=0.8pt] (6.,2.5) -- (7.,2.5) -- (7.,3.5) -- (6.,3.5) -- cycle;
\draw[line width=0.8pt,fill=red,fill opacity=0.1] (5.,3.5) -- (6.,3.5) -- (6.,4.5) -- (5.,4.5) -- cycle;
\draw[line width=0.8pt,fill=green,fill opacity=0] (6.,3.5) -- (7.,3.5) -- (7.,4.5) -- (6.,4.5) -- cycle;
\draw[line width=0.8pt,color=black,fill=red,fill opacity=0.1] (3.,4.267949192431123) -- (1.5,3.4019237886466858) -- (2.25,2.1028856829700264) -- (3.,2.535898384862245) -- cycle;
\draw[line width=0.8pt,color=black,fill=green,fill opacity=0] (3.,4.267949192431123) -- (3.,2.535898384862245) -- (3.75,2.1028856829700255) -- (4.5,3.401923788646684) -- cycle;
\draw[line width=0.8pt,color=black,fill=black!10] (5.,1.5) -- (6.,1.5) -- (6.,0.) -- (5.,0.) -- cycle;
\draw[line width=0.8pt] (6.,1.5) -- (7.,1.5) -- (7.,0.) -- (6.,0.) -- cycle;
\draw[line width=0.8pt,color=black,fill=black!10] (2.8125,1.128607103712532) -- (3.,1.2368602791855867) -- (3.,0.8038475772933671) -- cycle;

\draw [line width=0.8pt] (3.,4.267949192431123)-- (1.5,3.4019237886466858);
\draw [line width=0.8pt] (3.,4.267949192431123)-- (4.5,3.401923788646684);
\draw [line width=0.8pt] (2.25,2.1028856829700264)-- (3.,2.535898384862245);
\draw [line width=0.8pt] (3.,2.535898384862245)-- (3.75,2.1028856829700255);
\draw [line width=0.8pt] (2.625,1.4533666301316968)-- (3.,1.6698729810778061);
\draw [line width=0.8pt] (3.,1.6698729810778061)-- (3.375,1.4533666301316963);
\draw [line width=0.8pt] (2.8125,1.128607103712532)-- (3.,1.2368602791855867);
\draw [line width=0.8pt] (3.1875,1.1286071037125316)-- (3.,1.2368602791855867);
\draw [line width=0.8pt] (1.5,3.4019237886466858)-- (3.,0.8038475772933671);
\draw [line width=0.8pt] (3.,0.8038475772933671)-- (4.5,3.401923788646684);
\draw [line width=0.8pt] (3.,4.267949192431123)-- (3.,0.8038475772933671);
\draw [line width=0.8pt] (5.,1.5)-- (6.,1.5);
\draw [line width=0.8pt] (6.,1.5)-- (6.,2.5);
\draw [line width=0.8pt] (6.,2.5)-- (5.,2.5);
\draw [line width=0.8pt] (5.,2.5)-- (5.,1.5);
\draw [line width=0.8pt] (6.,1.5)-- (7.,1.5);
\draw [line width=0.8pt] (7.,1.5)-- (7.,2.5);
\draw [line width=0.8pt] (7.,2.5)-- (6.,2.5);
\draw [line width=0.8pt] (6.,2.5)-- (6.,1.5);
\draw [line width=0.8pt] (5.,2.5)-- (6.,2.5);
\draw [line width=0.8pt] (6.,2.5)-- (6.,3.5);
\draw [line width=0.8pt] (6.,3.5)-- (5.,3.5);
\draw [line width=0.8pt] (5.,3.5)-- (5.,2.5);
\draw [line width=0.8pt] (6.,2.5)-- (7.,2.5);
\draw [line width=0.8pt] (7.,2.5)-- (7.,3.5);
\draw [line width=0.8pt] (7.,3.5)-- (6.,3.5);
\draw [line width=0.8pt] (6.,3.5)-- (6.,2.5);
\draw [line width=0.8pt,color=black] (5.,3.5)-- (6.,3.5);
\draw [line width=0.8pt,color=black] (6.,3.5)-- (6.,4.5);
\draw [line width=0.8pt,color=black] (6.,4.5)-- (5.,4.5);
\draw [line width=0.8pt,color=black] (5.,4.5)-- (5.,3.5);
\draw [line width=0.8pt,color=black] (6.,3.5)-- (7.,3.5);
\draw [line width=0.8pt,color=black] (7.,3.5)-- (7.,4.5);
\draw [line width=0.8pt,color=black] (7.,4.5)-- (6.,4.5);
\draw [line width=0.8pt,color=black] (6.,4.5)-- (6.,3.5);
\draw [line width=0.8pt,color=black] (3.,4.267949192431123)-- (1.5,3.4019237886466858);
\draw [line width=0.8pt,color=black] (1.5,3.4019237886466858)-- (2.25,2.1028856829700264);
\draw [line width=0.8pt,color=black] (2.25,2.1028856829700264)-- (3.,2.535898384862245);
\draw [line width=0.8pt,color=black] (3.,2.535898384862245)-- (3.,4.267949192431123);
\draw [line width=0.8pt,color=black] (3.,4.267949192431123)-- (3.,2.535898384862245);
\draw [line width=0.8pt,color=black] (3.,2.535898384862245)-- (3.75,2.1028856829700255);
\draw [line width=0.8pt,color=black] (3.75,2.1028856829700255)-- (4.5,3.401923788646684);
\draw [line width=0.8pt,color=black] (4.5,3.401923788646684)-- (3.,4.267949192431123);
\draw [line width=0.8pt,color=black] (5.,1.5)-- (6.,1.5);
\draw [line width=0.8pt,color=black] (6.,1.5)-- (6.,0.);
\draw [line width=0.8pt,color=black] (6.,0.)-- (5.,0.);
\draw [line width=0.8pt,color=black] (5.,0.)-- (5.,1.5);
\draw [line width=0.8pt] (6.,1.5)-- (7.,1.5);
\draw [line width=0.8pt] (7.,1.5)-- (7.,0.);
\draw [line width=0.8pt] (7.,0.)-- (6.,0.);
\draw [line width=0.8pt] (6.,0.)-- (6.,1.5);
\draw [line width=0.8pt,color=black] (2.8125,1.128607103712532)-- (3.,1.2368602791855867);
\draw [line width=0.8pt,color=black] (3.,1.2368602791855867)-- (3.,0.8038475772933671);
\draw [line width=0.8pt,color=black] (3.,0.8038475772933671)-- (2.8125,1.128607103712532);

\draw [fill=black] (3.,0.8038475772933671) circle (1.0pt);
\draw[color=black] (3.0709748970611255,0.9318258316488557) node [label={[label distance=0.1cm]-90:$I$}]{};
\draw [fill=black] (3.,4.267949192431123) circle (1.0pt);
\draw[color=black] (3.0709748970611255,4.394439157992787) node [label={[label distance=-0.3cm]90:$D$}]{};
\draw [fill=black] (1.5,3.4019237886466858) circle (1.0pt);
\draw[color=black] (1.57528985077053,3.533908035469443) node[label={[label distance=0.1cm]200:$A$}]{};
\draw [fill=black] (4.5,3.401923788646684) circle (1.0pt);
%\draw[color=black] (4.566659943351721,3.533908035469443) node {$G$};
\draw [fill=black] (2.25,2.1028856829700264) circle (1.0pt);
\draw[color=black] (2.323132373915828,2.2328669335591496) node[label={[label distance=0.1cm]-120:$B$}]{};
\draw [fill=black] (2.625,1.4533666301316968) circle (1.0pt);
\draw [fill=black] (2.8125,1.128607103712532) circle (1.0pt);
\draw[color=black] (2.8865753708061206,1.2596472116577484) node [label={[label distance=0.05cm]183:$H$}]{};
\draw [fill=black] (3.75,2.1028856829700255) circle (1.0pt);
%\draw[color=black] (3.8188174202064236,2.2328669335591496) node {$K$};
\draw [fill=black] (3.375,1.4533666301316963) circle (1.0pt);
\draw [fill=black] (3.1875,1.1286071037125316) circle (1.0pt) node[label={[label distance=-0.1cm]0:$L$}]{};
\draw [fill=black] (3.,2.535898384862245) circle (1.0pt);
\draw[color=black] (3.0709748970611255,2.6631324948208213) node[label={[label distance=-0.3cm]0:$C$}]{};
\draw [fill=black] (3.,1.6698729810778061) circle (1.0pt);
\draw [fill=black] (3.,1.2368602791855867) circle (1.0pt);
\draw[color=black] (3.0709748970611255,1.3620913929105276) node {$K$};
\draw [fill=black] (6.,1.5) circle (1.0pt);
\draw [fill=black] (5.,1.5) circle (1.0pt);
\draw[color=black] (5.068636431490345,1.6284462641677533) node [label={[label distance=-0.35cm]0:$H'$}]{};
\draw [fill=black] (7.,1.5) circle (1.0pt);
\draw [fill=black] (6.,1.5) circle (1.0pt);
\draw[color=black] (6.072589407767594,1.6284462641677533) node [label={[label distance=-0.37cm]0:$K'$}]{};
\draw [fill=black] (6.,2.5) circle (1.0pt);
\draw [fill=black] (5.,2.5) circle (1.0pt);
\draw [fill=black] (7.,2.5) circle (1.0pt);
\draw [fill=black] (6.,2.5) circle (1.0pt);
\draw [fill=black] (6.,3.5) circle (1.0pt);
\draw [fill=black] (5.,3.5) circle (1.0pt);
\draw[color=black] (5.094247476803541,3.651718843910139) node [label={[label distance=-0.4cm]0:$B'$}]{};
\draw [fill=black] (7.,3.5) circle (1.0pt);
%\draw[color=black] (7.102153429358038,3.651718843910139) node {$F_1$};
\draw [fill=black] (6.,3.5) circle (1.0pt);
\draw[color=black] (6.098200453080789,3.651718843910139) node [label={[label distance=-0.4cm]0:$C'$}]{};
%\draw[color=black] (5.652568264631194,4.087106614234449) node {$poly8$};
\draw [fill=black] (6.,4.5) circle (1.0pt);
\draw [fill=black] (5.,4.5) circle (1.0pt);
\draw[color=black] (5.094247476803541,4.655671820187373) node[label={[label distance=-0.4cm]0:$A'$}]{};
%\draw[color=black] (6.6462768227831654,4.087106614234449) node {$poly9$};
\draw [fill=black] (7.,4.5) circle (1.0pt);
%\draw[color=black] (7.102153429358038,4.655671820187373) node {$J_1$};
\draw [fill=black] (6.,4.5) circle (1.0pt);
\draw[color=black] (6.098200453080789,4.655671820187373) node [label={[label distance=-0.4cm]0:$D'$}]{};
%\draw[color=black] (2.4972874820455546,3.1651089829594388) node {$q1$};
%\draw[color=black] (3.613929057700862,3.1651089829594388) node {$q2$};
\draw [fill=black] (6.,0.) circle (1.0pt);
\draw[color=black] (6.098200453080789,0.15837226319037406) node [label={[label distance=-0.37cm]0:$J'$}]{};
\draw [fill=black] (5.,0.) circle (1.0pt);
\draw[color=black] (5.094247476803541,0.15837226319037406) node [label={[label distance=-0.35cm]0:$I'$}]{};
%\draw[color=black] (5.560368501503691,0.8396260685213546) node {$q3$};
\draw [fill=black] (7.,0.) circle (1.0pt) node [label={[label distance=-0.15cm]45:$\tilde I'$}]{};
\draw [fill=black] (2.999748970611255,1.04) circle (1.0pt) node [label={[label distance=-0.3cm]60:{}}]{};
%\draw[color=black] (3.004386179246818,1.121347566966497) node {$O_1$};
\draw[<-, dashed, opacity=.5] (3.05,1)--(3.5,0.5);
\draw (3.5,0.5) node [label={[label distance=-0.35cm]-30:{$J$}}]{};

\end{tikzpicture}
\caption{Illustration of the folding map from a quadrilateral $Z$ onto two rectangles.}\label{fig:folding}
\end{figure}
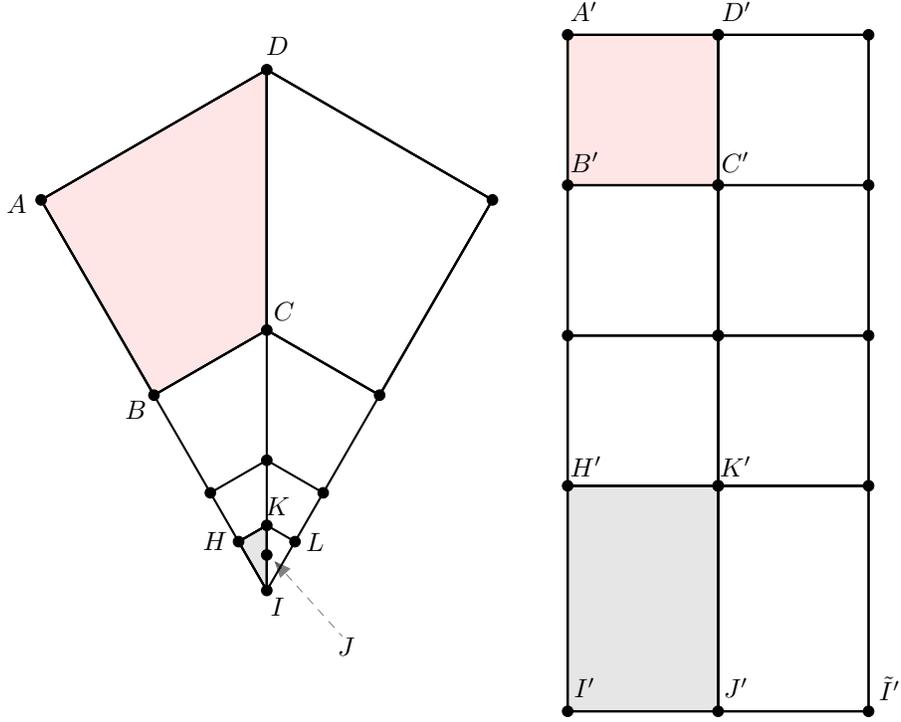

Hence, we obtain two piecewise linear maps, one defined on the triangle $HIK$ and one defined on the triangle $KIL$. The first one maps linearly the segment $IJ$ to the segment $I'J'$ and the second maps linearly the segment $IJ$ to the segment $\tilde I'J'$. Upon folding, the segment $I'J'$ gets glued  to the segment $\tilde I'J'$, so we obtain a homeomorphism $\phi$ from $\br Z$ onto the two folded rectangles that are glued over an edge of the tripod $G$. We remark that the segment $A'D'$ is not glued to anything at this moment, and it will be glued to the corresponding segment that arises from the rectangles attached to another edge of $G$.

With the same procedure, we construct such a homeomorphism $\phi_i$ for each one of the three quadrilaterals $Z_i$, $i=1,2,3$, in the subdivision of the original triangle $W$. The map $\phi_i$ maps $\br Z_i$ to the two folded rectangles attached to the edge $e_i$ of $G$, for $i=1,2,3$. We remark that the heights of all these rectangles are equal to $h$, but their  widths  might vary if they are attached to distinct edges of $G$. The quadrilateral $Z_1$ is glued to $Z_2$ along  edges of the form $AD$. The maps $\phi_1$ and $\phi_2$ are linear on these edges and they map $AD$ to edges of ``type" $A'D'$, whose length is $h$. Hence, the maps $\phi_1$ and $\phi_2$ can be ``glued" together to obtain a homeomorphism from $\br Z_1\cup \br Z_2$ onto the rectangles that are attached on two of the edges of the tripod $G$. Finally, one glues the third piece $\br Z_3$ to obtain the desired homeomorphism $\phi\colon  \br W\to X$; see Figure \ref{fig:FoldTriangle}. The map $\phi$ is $M$-quasiconformal, in the sense of Definition \ref{Pre Definition QC metric}, in the interior of  the preimage of each rectangle $E\sim G$. Note that the set where $\phi$ is quasiconformal is the complement in $W$ of finitely many line segments, along which the triangle $W$ is ``folded".
     
By construction, (A) and (C) hold, so it only remains to check that $\phi$ has the desired  modulus of continuity (B). We reformulate the claim, using the metric $d$ of the space $X$. The quantity $\mathcal O(W)$ is equal to the maximum length of the edges of the tripod $G$. We need to check that if $x_k,x_{k+1}$ are consecutive ``dyadic" points in a half-edge of $W$, then 
\begin{align*}
 d(\phi (x_k),\phi (x_{k+1})) \leq \mathcal O(W)/3.
\end{align*}
Since the segment $[\phi(x_k),\phi(x_{k+1})] \subset X$ projects isometrically into an edge of $G\subset \R^2$, the desired claim will then follow. Note that $\phi(x_k)$ and $\phi(x_{k+1})$ are contained either in an edge of a square of dimensions $h\times h$, or in an edge of  a rectangle of dimensions $q\times h$, where $h\leq q<2h$; recall the subdivision of $Z$ into ``dyadic" trapezoids and the definition of $\phi$ in each trapezoid. Hence, $d(\phi (x_k),\phi (x_{k+1})) \leq 2h \leq \ell/3$, since we chose $h\leq \ell/6$, where $\ell$ is the length of the edge of $G$ that contains $x_k$ and $x_{k+1}$. On the other hand, $\ell \leq \mathcal O(W)$, so our claim is proved. 

\subsection{Homeomorphism onto a flap-plane}\label{Section QC homeo}

Here we show how to patch together the folding maps $\phi$ of each $w$-triangle $W$ (from Section \ref{Section QC Folding}) with the continuous map $f\colon \R^2 \to \R^2$ in order to obtain a homeomorphism  $\Phi$ from $\R^2$ onto a flap-plane $S$. This will be done in three steps. (a) First, we explain how the folding maps $\phi$ of the $w$-triangles can induce the building block maps of $f\colon \R^2\to \R^2$, so that $f$ is continuous; recall the comments in Section \ref{Section Properties of f}(b). (b) Then we discuss how to construct a flap-plane $S$ by gluing rectangles to the tripods provided by $f$.  (c) Finally, we explain how one can ``patch" together the map $f$ with the folding maps $\phi$ to obtain a homeomorphism $\Phi\colon \R^2 \to S$ that is quasiconformal in $\R^2\setminus K$.

(a) Let $W$ be an equilateral triangle and $G$ be a tripod. Consider a map $\phi$ that maps $\br W$ to a metric space $X$ as in Section \ref{Section QC Folding}, satisfying (A), (B), and (C). There is a natural projection $P\colon X\to G$, so that the composition $P\circ \phi$ satisfies \ref{BB2}, \ref{BB3}, \ref{BB5}, \ref{BB6}, and condition ($\ast$). Note that $\phi$ depends on the height $h$ of the rectangles that we attach to the tripod $G$, and condition ($\ast$) is subject to choosing a sufficiently small height $h$.

Using the compositions $P\circ \phi$ as the building block maps of $f$ we obtain a continuous map $f\colon \R^2\to \R^2$. More specifically, once $f$ has been defined on $w$-triangles of level $m-1$, then we know the tripods (as sets) that will correspond to the $w$-triangles of level $m$. Then one considers the folding maps $\phi$ with respect to the $w$-triangles of level $m$ and the corresponding tripods. The compositions $P\circ \phi$ yield the building block maps that are used to define $f$ on $w$-triangles of level $m$. The comments in Section \ref{Section Properties of f}(b) justify why $f$ will extend continuously on all of $\R^2$.

A subtlety here is that if we change the height $h$ of the rectangles attached to a specific tripod, then this changes the folding map $\phi$, and alters the map $f$ completely!

(b) Let $\mathcal G$ be the family of tripods arising from $f$, that is, the family $f(\br W)$, where $W$ is a $w$-triangle of level at least $1$. The family $\mathcal G$ is a family of tripods that has property $(G)$, i.e., any two tripods have at most one point of intersection, in which case it is a non-central vertex of one of them; recall the definitions from Section \ref{Section:flap-plane:multipletripods}. Also, the ``graph" $T_\infty =\bigcup_{G\in \mathcal G} G$ has degree uniformly bounded by $6$, as one can see inductively. Proposition \ref{General-Completeness} implies that if the heights of the rectangles attached to each tripod are sufficiently small, then one obtains a limiting flap-plane $(S,d)$, which is a complete metric space. Recall Remark \ref{General-Graph dependence}, which allows us to choose inductively the tripods $G$ and the heights $h$ of the corresponding rectangles, and still obtain the limiting flap-plane. The limiting space $S$ can be regarded as the union of $\R^2\setminus \bigcup_{G\in \mathcal G} G$ with the rectangles attached to each tripod $G$, after proper identifications; see also the comments in Section \ref{Section:flap-plane:inverse}.

(c) The map $f$ can be ``patched" with the maps $\phi$ to yield naturally a map $\Phi\colon  \R^2 \to  S$. Namely, the maps $f$ and $\Phi$ agree outside the closures of $w$-triangles, and inside a $w$-triangle $\br W$ the map $\Phi$ is defined to be equal to the folding map $\phi$ that folds $\br W$ on top of the corresponding tripod $f(\br W)$. There is possibly an ambiguity in the definition of $\Phi$, whenever two $w$-triangles $\br W_1,\br W_2$ intersect at one point. In this case the corresponding tripods $G_1,G_2$ also intersect at one point (as described in property $(G)$). Then $\Phi(\br W_1)$ and $\Phi(\br W_2)$ also have to intersect at precisely one point in the space $S$, by our basic rules in the construction of a flap-plane; see Section \ref{Section General} and Figure \ref{fig:TwoTripodflaps}. Hence, $\Phi$ can unambiguously be defined.

We claim that $\Phi$ is injective. Recall from Section \ref{Section Properties of f}(c) that $f$ is injective on $K^\circ$ with $f(K^\circ)=L^\circ$ and it is the identity in the unbounded complementary component of the gasket $K$.  Also, each of the maps $\phi$ is a homeomorphism from a $w$-triangle $\br W$ onto the rectangles attached to the corresponding tripod. In order to show that $\Phi$ is injective, it remains to prove that if $x_1\in \partial W_1$ and $x_2\in \partial W_2$, where $W_1,W_2$ are distinct $w$-triangles with $\Phi(x_1)=\Phi(x_2)$, then $x_1=x_2$. From the construction of $S$ (looking at a finite stage of the construction) one sees that the equality $\Phi(x_1)=\Phi(x_2)$ is only possible if the triangles $W_1$ and $W_2$ are adjacent, so we necessarily have that $x_1=x_2$ and it is the intersection point $\partial W_1\cap \partial W_2$. This completes the proof of injectivity. 

In fact, $\Phi$ is also surjective, since $f$ maps $K^\circ$ surjectively onto $L^\circ$; see the comments in Section \ref{Section Properties of f}(c). Note that $\Phi$ maps, in a sense, $\infty$ to $\infty$  and is continuous in a neighborhood of $\infty$, as it agrees with the identity there. Hence, if we show that $\Phi\colon \R^2\to S$ is continuous, then it will be a a proper bijective map, and hence a homeomorphism, as desired.

The proof of continuity is very similar to the proof of Proposition \ref{General-Completeness}, so we only provide a sketch. Assume for the sake of contradiction that $x_k$ is a sequence in $\R^2$ converging to a point $x\in \R^2$, but the image points $y_k=\Phi(x_k)$, $y=\Phi(x)$ satisfy $d(y_k,y)\geq \delta$ for all $k\in \N$ and for some $\delta>0$. The map $\Phi$ is already continuous in the interior of all $w$-triangles, as it agrees there with the homeomorphisms $\phi$, hence $x$ cannot lie in the interior of a $w$-triangle. It follows that $x$ must lie on the gasket $K$. 

With the same reasoning, we cannot have that infinitely many terms $x_k$ lie in the same $w$-triangle $\br W$. Hence, we either have a subsequence of $x_k$ all of whose terms lie in $K^\circ$, or there exists a subsequence of $x_k$, still denoted by $x_k$, whose terms lie in distinct $w$-triangles $\br W_k$. 

In the first case, we assume that $x_k\in K^\circ $ for all $k\in \N$ and we consider two subcases. The first subcase is that $x\in K^\circ$, in which case $y\in L^\circ \subset \R^2 \setminus \bigcup_{G\in \mathcal G} G$; recall that $\Phi(K^\circ )= f(K^\circ)=L^\circ$ from Section \ref{Section Properties of f}(c). Then we necessarily have $y_k\to y$ with the Euclidean metric. This is because the map $\Phi$ ``agrees" there with $f$, and $f$ is continuous. We would like to argue that $y_k\to y$ with respect to the metric $d$ of $S$. This follows  because the metric $d$ restricted to $\R^2 \setminus \bigcup_{G\in \mathcal G} G$ is topologically equivalent to the Euclidean metric; see Remark \ref{General-Metrics equivalent}. Hence, we obtain a contradiction to the assumption that $d(y_k,y)\geq \delta$ for all $k\in \N$. 

The other subcase is that $x$ lies on an edge of a $w$-triangle $W$. Then one can find a nested sequence $V_k$ of $v$-triangles and a subsequence of $x_k$, still denoted by $x_k$, such that $x_k\in V_k$ and $x\in \partial V_k$ for all $k\in \N$; see comments before Lemma \ref{Combinatorics V lemma}. The triangles $V_k$ correspond, under $f$, to a nested sequence of $u$-quadrilaterals $U_k$ such that the projection $\tilde y$ of the point $\Phi(x)=y\in S$ to the plane lies in $\partial U_k$. In fact, $y$ lies in the boundary of a rectangle $E$ attached to the tripod $G$ that corresponds to $W$, and also $y$ is ``accessible" from $U_k$; this statement can be made more precise if one looks at a finite stage $(S_n,d_n)$ of the construction of $(S,d)$, as discussed in Section \ref{Section Flap-planes}. The line segments $(\tilde y,y_k]\subset \R^2$ are contained in the convex quadrilateral $\br U_k$ and have the property that they do not intersect any tripod infinitely often. As in the proof of Proposition \ref{General-Completeness}, this implies that for any given $\varepsilon>0$ we have
\begin{align*}
d(y,y_k) \leq |\tilde y-y_k|+\varepsilon \leq \diam(U_k) +\varepsilon
\end{align*} 
for all sufficiently large $k$. Since $\diam(U_k)\to 0$ by the continuity of $f$, we obtain a contradiction.

The second case is that  $x_k \in \br W_k$, which are distinct $w$-triangles. We will reduce this to the previous case. One can find points $x_k' \in \partial W_k$ with $x_k' \to x$, since the diameters of $W_k$ shrink to $0$, such that the corresponding image points $y_k'$ satisfy  $d(y_k',y_k) \to 0$. This is because the heights of the rectangles attached to distinct tripods have to shrink to $0$; see the statement of Proposition \ref{General-Completeness}. It suffices to show that $d(y_k',y)\to 0$. Finally, arbitrarily close to $x_k'$ one can find points $x_k'' \in K^\circ$ that converge to $x$ such that the corresponding image points $y_k''\in L^\circ$ satisfy $d(y_k',y_{k}'')\to 0$. This follows from the argument in the previous paragraph. However, we are now reduced to the previous case, so $d(y_k'',y)\to 0$, and therefore $d(y_k,y)\to 0$, a contradiction.

\subsection{Finishing the proof of non-removability}\label{Section QC Finish}
We have constructed a homeomorphism $\Phi\colon \R^2 \to S$ that is $M$-quasiconformal in  each $w$-triangle $W$, except at the finitely many line segments along which $W$ is folded; recall property (C) of the folding maps in Section \ref{Section QC Folding}. If the heights of the rectangles attached to the tripods are  chosen inductively to be sufficiently small, then by Theorem \ref{General-QS Embedding} there exists an $\eta$-quasisymmetry $\Psi\colon S\to \R^2$. Now, we consider the composition $F\coloneqq \Psi\circ \Phi$ which is a homeomorphism of $\R^2$. 

The map $\Psi$, restricted on a rectangle $E$ attached to a tripod, is $\eta$-quasisymmetric, and thus $M'$-quasiconformal, where $M'$ depends only on $\eta$; this follows from Lemma \ref{Pre QS implies QC} and the fact that the metric $d$, restricted to $E$, is isometric to the Euclidean metric. Hence, by Lemma \ref{Pre Compositions}, in each $w$-triangle $W$ the map $F$ is $M\cdot M'$-quasiconformal in the complement of finitely many line segments; these are precisely the segments along which the triangle $W$ is folded. Since these segments have finite length, the are removable for quasiconformal maps by Lemma \ref{Pre Removability}, so $F$ is $M\cdot M'$-quasiconformal on each $w$-triangle $W$. 

We finally claim that $F$ cannot be quasiconformal on $\R^2$. First, recall the continuous map $f\colon \R^2\to \R^2$ that is in fact the composition of $\Phi$ with the natural projection $P$ from $S$ to $\R^2$; see Section \ref{Section General} and Remark \ref{General-Properties S infinity} for the definition of the projection. We note that $f(K^\circ)=L^\circ $ has positive Lebesgue measure, as remarked in Section \ref{Section Properties of f}(d). Since the projection $P$ is $1$-Lipschitz and $L^\circ$ projects to itself, if $\mu$ denotes the Hausdorff $2$-measure of $S$, we have
\begin{align*}
\mu (L^\circ) \geq m_2(L^\circ) >0.
\end{align*}
Compare to the property \ref{G-Measure} of the projections of flap-planes in Section \ref{Section General}. 

Finally, by Lemma \ref{Pre Absolutely continuous}, we conclude that the pushforward measure $\Psi_*\mu$ and the Lebesgue measure on $\R^2$ are mutually absolutely continuous. This implies that $F(K^\circ)= \Psi(L^\circ) \subset \R^2$ has positive Lebesgue measure. Thus, the map $F$ blows the set $K^\circ$ of measure zero to a set of positive measure. Using Lemma \ref{Pre Measure zero}, we conclude that $F$ cannot be globally quasiconformal. \qed


\begin{thebibliography}{99}\setlength{\itemsep}{-1pt}
\small


\bibitem[AT04]{AT}
	L.~Ambrosio, P.~Tilli,
	\emph{Topics on analysis in metric spaces},
	Oxford University Press, Oxford, 2004.

\bibitem[AIM09]{AIM}
  	K.~Astala, T.~Iwaniec, and G.~Martin,
  	\emph{Elliptic Partial Differential Equations and
Quasiconformal Mappings in the Plane},
  	Princeton University Press, Princeton, NJ,  2009. 

\bibitem[Be31]{Be}
	A.S.~Besicovitch,
	\emph{On Sufficient Conditions for a Function to be Analytic, and on Behaviour of Analytic Functions in the Neighbourhood of Non-Isolated Singular Points},
	Proc.\ London Math.\ Soc.\ (2) 32 (1931), no.\ 1, 1--9.

\bibitem[Bi94]{Bi}
	C.~Bishop,
	\emph{Some homeomorphisms of the sphere conformal off a curve},
	Ann.\ Acad.\ Sci.\ Fenn.\ Ser.\ A I Math.\ 19 (1994), no.\ 2, 323--338.

\bibitem[Bi98]{Bi:R3}
	C.~Bishop,
	\emph{Non-removable sets for quasiconformal and locally biLipschitz mappings in $\R^3$},
	Stony Brook IMS, 1998, preprint: \url{http://www.math.stonybrook.edu/cgi-bin/preprint.pl?ims98-6}.

\bibitem[Bi15]{Bi2}
	C.~Bishop,
	NSF research proposal, 2015,
	\url{http://www.math.stonybrook.edu/~bishop/vita/nsf15.pdf}.
	
\bibitem[BK02]{BK}
	M.~Bonk, B.~Kleiner,
	\emph{Quasisymmetric parametrizations of two-dimensional metric spheres},
	Invent.\ Math.\ 150 (2002), no.\ 1, 127--183.
	
\bibitem[BM17]{BM}
	M.~Bonk, D.~Meyer,
	\emph{Expanding Thurston Maps},
	Mathematical Surveys and Monographs, 225,
	American Mathematical Society, Providence, RI, 2017.
	
\bibitem[BBI01]{BBI}
	D.~Burago, Y.~Burago, S.~Ivanov,
	\emph{A course in metric geometry},
	Graduate Studies in Mathematics , 33,
	American Mathematical Society,
	Providence, RI, 2001.

\bibitem[DS97]{DS}
	G.~David, S.~Semmes,
	\emph{Fractured fractals and broken dreams},
	Oxford University Press, New York, 1997.

\bibitem[FZ72]{FZ}
	H.~Federer, W.P.~Ziemer,
	\emph{The Lebesgue set of a function whose distribution derivatives are $p$-th power summable},
	Indiana Univ.\ Math.\ J.\ 22 (1972/73), 139--158.

\bibitem[Fo99]{Fo}
	G.B.~Folland,
	\emph{Real Analysis},
	2nd ed., Wiley, New York, 1999.

\bibitem[Ge60]{Ge:remov}
	F.W.~Gehring,
	\emph{The definitions and exceptional sets for quasiconformal mappings},
	Ann.\ Acad.\ Sci.\ Fenn.\ Ser.\ A I No.\ 281, 1960, 0--28.	

\bibitem[GS09]{GS}
	J.~Graczyk, S.~Smirnov,
	\emph{Non-uniform hyperbolicity in complex dynamics},
	Invent.\ Math.\ 175 (2009), no.\ 2, 335--415.

\bibitem[He01]{He}
	J.~Heinonen,
	\emph{Lectures on Analysis on Metric Spaces},
	Springer Verlag, New York, 2001.

\bibitem[HK98]{HK}	
	J.~Heinonen, P.~Koskela,
	\emph{Quasiconformal maps in metric spaces with controlled geometry},
	Acta Math.\ 181 (1998), no.\ 1, 1--61.	

\bibitem[HS94]{HS}
	Z.-X.~He, O.~Schramm,
	\textit{Rigidity of circle domains whose boundary has $\sigma$-finite linear measure},
	Invent.\ Math.\ 115 (1994), no.\ 2, 297--310.

\bibitem[Jo91]{Jo}
	P.~Jones,
	\emph{On removable sets for Sobolev spaces in the plane},
	Essays on Fourier analysis in honor of Elias M.~Stein (Princeton, NJ, 1991), 250--276,
	Princeton Math.\ Ser., 42, Princeton Univ.\ Press, Princeton, NJ, 1995.

\bibitem[JS00]{JS}
	P.~Jones, S.~Smirnov,
	\emph{Removability theorems for Sobolev functions and quasiconformal maps},
	Ark.\ Mat.\ 38 (2000), no.\ 2, 263--279.

\bibitem[Kah98]{Kah}
	J.~Kahn,
	\emph{Holomorphic Removability of Julia Sets},
	preprint arXiv:math/9812164.

\bibitem[Kau84]{Kau}
	R.~Kaufman,
	\emph{Fourier-Stieltjes coefficients and continuation of functions},
	Ann.\ Acad.\ Sci.\ Fenn.\ Ser.\ A I Math.\ 9 (1984), 27--31.

\bibitem[KW96]{KW}
	R.~Kaufman, J.-M.~Wu,
	\emph{On removable sets for quasiconformal mappings},
	Ark.\ Mat.\ 34 (1996), no.\ 1, 141--158.

\bibitem[Ki18]{Ki}
	K.~Kinneberg,
	\emph{Lower bounds for codimension-1 measure in metric manifolds},
	Rev.\ Mat.\ Iberoam.\ 34 (2018), no.\ 3, 1103--1118.

\bibitem[KL04]{KL}
	S.~Keith, T.~Laakso,
	\emph{Conformal Assouad dimension and modulus},
	Geom.\ Funct.\ Anal. 14 (2004), no.\ 6, 1278--1321.

\bibitem[KRZ17]{KRZ}
	P.~Koskela, T.~Rajala, Y.~Zhang,
	\emph{A density problem for Sobolev functions on Gromov hyperbolic domains},
	Nonlinear Anal.\ 154 (2017), 189--209.

\bibitem[KN05]{KN}
	P.~Koskela, T.~Nieminen,
	\emph{Quasiconformal removability and the quasihyperbolic metric},
	Indiana Univ.\ Math.\ J.\ 54 (2005), no.\ 1, 143--151.

\bibitem[LW17]{LW}
	A.~Lytchak, S.~Wenger,
	\emph{Canonical parametrizations of metric discs},
	preprint arXiv:1701.06346.

\bibitem[MW13]{MW}
	S.~Merenkov, K.~Wildrick,
	\emph{Quasisymmetric Koebe uniformization},
	Rev.\ Mat.\ Iberoam.\ 29 (2013), no.\ 3, 859--909.


\bibitem[Mo62]{Mo}
	R.L.~Moore,
	\emph{Foundations of point set theory},
	American Mathematical Society Colloquium Publications,
	American Mathematical Society, Providence, R.I., 1962.

\bibitem[Nt17]{Nt}
	D.~Ntalampekos,
	\emph{A removability theorem for Sobolev functions and detour sets},
	preprint arXiv:1706.07687.

\bibitem[Nt18]{Nt2}
	D.~Ntalampekos,
	\emph{Non-removability of Sierpinski carpets},
	preprint arXiv:1809.05605.
	
\bibitem[NY18]{NY}
	D.~Ntalampekos, M.~Younsi,
	\emph{Rigidity theorems for circle domains},
	preprint arXiv:1809.05573.

\bibitem[Ra17]{Ra}
	K.~Rajala,
	\emph{Uniformization of two-dimensional metric surfaces},
	Invent.\ Math.\ 207 (2017), no.\ 3, 1301--1375.

\bibitem[Sh16]{Sh}
	S.~Sheffield,
	\emph{Conformal weldings on random surfaces: SLE and the quantum gravity zipper},
	Ann.\ Probab.\ 44 (2016), no.\ 5, 3474--3545.
	
	
\bibitem[SS90]{SS}
	W.~Smith, D.~Stegenga,
	\emph{H\"older domains and Poincar\'e domains},
	Trans.\ Amer.\ Math.\ Soc.\ 319 (1990), no.\ 1, 67--100.	

\bibitem[V\"a71]{Va}
	J.~V\"ais\"al\"a,
	\emph{Lectures on $n$-dimensional quasiconformal mappings},
	Lecture Notes in Mathematics, Vol.\ 229, 
	Springer-Verlag, Berlin-New York, 1971.

\bibitem[Wh58]{Wh2}
	G.T.~Whyburn,
	\emph{Topological characterization of the Sierpi\'nski curve},
	Fund.\ Math.\ 45 (1958), 320--324.

\bibitem[Wi08]{Wi}
	K.~Wildrick,
	\emph{Quasisymmetric parametrizations of two-dimensional metric planes},
	Proc.\ Lond.\ Math.\ Soc.\ (3) 97 (2008), no.\ 3, 783--812.

\bibitem[Wu98]{Wu}
	J.-M.~Wu,
	\emph{Removability of sets for quasiconformal mappings and Sobolev spaces},
	Complex Variables Theory Appl.\ 37 (1998), no.\ 1--4, 491--506.

\bibitem[Yo15]{Yo}
	M.~Younsi,
	\emph{On removable sets for holomorphic functions},
	EMS Surv.\ Math.\ Sci.\ 2 (2015), no.\ 2, 219--254.
	
\bibitem[Yo16]{YoKoebe}
	M.~Younsi,
	\emph{Removability, rigidity of circle domains and Koebe's conjecture},
	Adv.\ Math.\ 303 (2016), 1300--1318.	
	
	
\end{thebibliography}
\end{document}